\documentclass[11pt]{amsart}
\usepackage{amssymb,amsmath,amsthm,amscd,mathrsfs,graphicx,color,multicol,booktabs}
\usepackage[cmtip,all,matrix,arrow,tips,curve]{xy}

\numberwithin{equation}{subsection}
\theoremstyle{plain}
\newtheorem{theorem}{Theorem}[subsection]

\newtheorem{lemma}[theorem]{Lemma}
\newtheorem{claim}[theorem]{Claim}

\newtheorem{proposition}[theorem]{Proposition}
\newtheorem{corollary}[theorem]{Corollary}

\newtheorem{prediction}[theorem]{Prediction}
\newtheorem*{prediction*}{Prediction}

\theoremstyle{remark}
\newtheorem{remark}[theorem]{Remark}

\newtheorem*{remintro}{Remark}
\theoremstyle{definition}
\newtheorem{definition}[theorem]{Definition}
\newtheorem{notation}[theorem]{Notation}

\newtheorem{example}[theorem]{Example}

\newtheorem{recipe}[theorem]{Recipe}

\def\Id{\operatorname{Id}}
\def\Pic{\operatorname{Pic}}

\def\Hom{\operatorname{Hom}}

\newcommand{\bP}{\mathbb{P}}

\newcommand{\bR}{\mathbb{R}}

\newcommand{\bQ}{\mathbb{Q}}
\newcommand{\bZ}{\mathbb{Z}}

\newcommand{\sF}{\mathscr{F}}

\newcommand{\bC}{\mathbb{C}}

\newcommand{\fH}{\mathfrak{H}}

\newcommand{\gp}{\mathfrak{p}}

\newcommand{\calF}{\mathcal{F}}

\newcommand{\calM}{\mathcal{M}}
\newcommand{\calO}{\mathcal{O}}
\newcommand{\calL}{\mathcal{L}}

\newcommand{\Aut}{\mathrm{Aut}}
\newcommand{\Stab}{\mathrm{Stab}}
\newcommand{\Span}{\mathrm{Span}}

\newcommand{\PGL}{\mathrm{PGL}}

\newcommand{\Sat}{\mathrm{Sat}}
\newcommand{\SL}{\mathrm{SL}}

\newcommand{\sD}{\mathscr{D}}

\newcommand{\Proj}{\mathrm{Proj}}

\newcommand{\im}{\operatorname{Im}}

\newcommand{\Gr}{\mathrm{Gr}}

\newcommand{\id}{\mathrm{id}}

\newcommand{\gquot}{/\!\!/}

\newcommand{\II}{\textrm{II}}

\newcommand{\CC}{{\mathbb C}}

\newcommand{\cD}{{\mathscr D}}
\newcommand{\Center}{\mathtt{Center}}

\newcommand{\cF}{{\mathscr F}}

\newcommand{\cH}{{\mathscr H}}

\newcommand{\cK}{{\mathscr K}}
\newcommand{\cL}{{\mathscr L}}

\newcommand{\cO}{{\mathscr O}}

\newcommand{\cR}{{\mathscr R}}
\newcommand{\cS}{{\mathscr S}}

\newcommand{\cU}{{\mathscr U}}

\newcommand{\cX}{{\mathscr X}} 
\newcommand{\cY}{{\mathscr Y}}

\newcommand{\dra}{\dashrightarrow}

\newcommand{\es}{\varnothing}
\newcommand{\FF}{{\mathbb F}}

\newcommand{\gM}{\mathfrak{M}}

\newcommand{\hra}{\hookrightarrow}
\newcommand{\la}{\langle}

\newcommand{\lagr}{\mathbb{LG}(\bigwedge^ 3 \CC^6)}
\newcommand{\lagrdual}{\mathbb{LG}(\bigwedge^ 3 (\CC^6)^{\vee})}

\newcommand{\lra}{\longrightarrow}

\newcommand{\n}{\noindent}
\newcommand{\NN}{{\mathbb N}}
\newcommand{\ov}{\overline}
\newcommand{\PP}{{\mathbb P}}
\newcommand{\QQ}{{\mathbb Q}}
\newcommand{\ra}{\rangle}
\newcommand{\RR}{{\mathbb R}}

\newcommand{\Tower}{\mathtt{Tower}}

\newcommand{\wh}{\widehat}
\newcommand{\wt}{\widetilde}
\newcommand{\ZZ}{{\mathbb Z}}

\DeclareMathOperator{\CH}{CH}

\DeclareMathOperator{\cod}{cod}

\DeclareMathOperator{\Div}{Div}

\DeclareMathOperator{\divisore}{div}

\DeclareMathOperator{\mult}{mult}

\DeclareMathOperator{\sing}{sing}
\DeclareMathOperator{\supp}{supp}

\usepackage{ifthen}
\usepackage{rotating}
\usepackage{supertabular}
\newcommand{\cit}[1]{{\rm \textbf{#1}}}
\newcommand{\Ref}[2]{\cit{%
\ifthenelse{\equal{#1}{thm}}{Theorem}{}%
\ifthenelse{\equal{#1}{ass}}{Assumption}{}%
\ifthenelse{\equal{#1}{chp}}{Chapter}{}%
\ifthenelse{\equal{#1}{prp}}{Proposition}{}%
\ifthenelse{\equal{#1}{lmm}}{Lemma}{}%
\ifthenelse{\equal{#1}{cnj}}{Conjecture}{}%
\ifthenelse{\equal{#1}{crl}}{Corollary}{}%
\ifthenelse{\equal{#1}{dfn}}{Definition}{}%
\ifthenelse{\equal{#1}{expl}}{Example}{}%
\ifthenelse{\equal{#1}{hyp}}{Hypothesis}{}%
\ifthenelse{\equal{#1}{rmk}}{Remark}{}%
\ifthenelse{\equal{#1}{clm}}{Claim}{}%
\ifthenelse{\equal{#1}{pred}}{Prediction}{}%
\ifthenelse{\equal{#1}{prb}}{Problem}{}%
\ifthenelse{\equal{#1}{table}}{Table}{}%
\ifthenelse{\equal{#1}{exe}}{Exercise}{}%
\ifthenelse{\equal{#1}{qst}}{Question}{}%
\ifthenelse{\equal{#1}{rcp}}{Recipe}{}%
\ifthenelse{\equal{#1}{sec}}{Section}{}%
\ifthenelse{\equal{#1}{subsec}}{Subsection}{}%
\ifthenelse{\equal{#1}{subsubsec}}{Subsubsection}{}%
\ifthenelse{\equal{#1}{univ}}{Universal Property}{}%
\ifthenelse{\equal{#1}{trm}}{Terminology}{}%
\ifthenelse{\equal{#1}{tbl}}{Table}{}%
\  \ref{#1:#2}%
}}

\usepackage{fullpage}
  
\begin{document}
 \title[Moduli of quartic $K3$s]{Birational geometry of the moduli space of quartic $K3$ surfaces }
 
 \author[R. Laza]{Radu Laza}
\address{Stony Brook University,  Stony Brook, NY 11794, USA}
\email{radu.laza@stonybrook.edu}

 \author[K. O'Grady]{Kieran O'Grady}
\address{``Sapienza'' Universit\'a di Roma, Rome, Italy}
\email{ogrady@mat.uniroma1.it}
\date{\today}
\begin{abstract}
By work of Looijenga and others, one has a good understanding of the relationship between GIT and  Baily-Borel compactifications for the moduli spaces of degree $2$ $K3$ surfaces, cubic fourfolds, and a few other related examples. The similar-looking cases of degree $4$ $K3$ surfaces and double EPW sextics  turn out to be much more complicated for arithmetic reasons. In this paper, we refine work of Looijenga to allow us to handle these cases. Specifically, in analogy with the so-called Hassett-Keel program for the moduli space of curves, we study the variation of log canonical models for locally symmetric varieties of Type IV associated to $D$ lattices. In particular, for the dimension $19$ case, we conjecturally obtain a continuous one-parameter interpolation (via a directed MMP) between the GIT and Baily-Borel compactifications for the moduli of  degree $4$ $K3$ surfaces. The predictions for the similar $18$ dimensional case, corresponding geometrically to hyperelliptic degree $4$ $K3$ surfaces, can be verified by means of VGIT -- this will be discussed elsewhere.
\end{abstract}

\thanks{Research of the first author is supported in part by NSF grants DMS-125481 and DMS-1361143. Research of the second author is supported in part by PRIN 2013.}
  \maketitle
  
  \bibliographystyle{amsalpha}
\section*{Introduction}
An important problem in algebraic geometry is to construct a geometric compactification for the moduli space of polarized degree $d$ $K3$ surfaces $\cK_d$. By Global Torelli, $\cK_d$ is isomorphic to a locally symmetric variety $\cF_d$, and hence it has natural compactifications, such as the Baily-Borel compactification $\sF_d^*$, Mumford's toroidal compactifications, and more generally Looijenga's semitoric compactifications. However, a priori,   these compactifications are only birational to the 
 \lq\lq geometric\rq\rq\ compactifications, obtained, for example, by  GIT. It is natural  to  compare the two kinds of compactifications. An understanding of the birational relationship between the Baily-Borel and GIT compactifications leads to deep results about the period map (e.g.~see \cite{shah}, \cite{lcubic}), and to results about the structure of the GIT quotient. The simplest instance of such comparison results is the isomorphism
$$(\fH/\SL(2,\bZ))^*\cong |\calO_{\bP^2}(3)|\gquot \SL(3)(\cong \bP^1)$$
between the compactified $j$-line and the GIT moduli space of plane cubic curves.
In  a vast generalization of this fact, Looijenga (\cite{looijenga1, looijengacompact}) has  devised a comparison framework that applies to locally symmetric varieties associated to type $IV$ or $I_{1,n}$  Hermitian symmetric domains. This framework was  successfully applied in the case of moduli of degree $2$ $K3$ surfaces (\cite{looijengavancouver}, \cite{shah}), cubic fourfolds (\cite{lcubic}, \cite{cubic4fold}), and  a few other related examples (e.g. cubic threefolds, see \cite{act},  \cite{ls}, del Pezzo surfaces, etc.). The nominal purpose of the present paper is to investigate the analogous cases of degree $4$ $K3$ surfaces (\cite{shah4}) and double EPW sextics (\cite{epwperiods, epw}). 
While attempting to study these cases in detail, we uncovered a rich and intriguing picture. 

\medskip

The starting point of our investigation are two limitations in Looijenga's construction. First of all, a certain technical assumption in Looijenga's work does not hold for quartic $K3$'s (nor for double EPW sextics),  while it does hold for $K3$ surfaces of arbitrary  degree $d\not=4$ (see Lemma 8.1 in~\cite{looijengacompact}). Namely, for arithmetic reasons,  the combinatorics of the hyperplane arrangement involved in Looijenga's construction \cite{looijengacompact} is much richer for polarized $K3$ surfaces   of degree $4$ than for polarized $K3$ surfaces of degree $d\not=4$ (similarly,  the hyperplane arrangement involved in the period map for double EPW sextics is much richer than  the hyperplane arrangement relevant to cubic fourfolds). Secondly, and this is a consideration which applies to $K3$ surfaces of any degree, 
 there exists a plethora of GIT models. In the low degree cases considered here and in the literature, there might be a \lq\lq natural\rq\rq\  choice for GIT, but this is misleading (see \cite{g4git} for a hint of what would happen already in degree $6$). The solution that we propose to handle these two issues is to give flexibility to Looijenga's construction by considering a continuous variation of models. More precisely, we recall that for a locally symmetric variety $\sF=\cD/\Gamma$, Baily-Borel have shown that the eponymous compactification $\sF^*$ is the $\Proj$ of the ring of automorphic functions, i.e.~$\sF^*=\Proj R(\sF,\lambda)$, where $\lambda$ is the Hodge bundle. Looijenga's deep insight was to observe that in certain situations of geometric interest, a certain GIT quotient $\overline \calM$ is nothing but the $\Proj$ of the ring of meromorphic automorphic forms with poles on a (geometrically meaningful) Heegner (or Noether--Lefschetz) divisor $\Delta$, and thus $\overline \calM=\Proj R(\sF,\lambda+\Delta)$. Furthermore, Looijenga has shown that under a certain assumption on $\Delta$ (which fails for quartic $K3$'s, and for double EPW sextics), $\Proj R(\sF,\lambda+\Delta)$ has an explicit combinatorial/arithmetic description. Our approach is to continuously interpolate between the two models by controlling the order of poles for the meromorphic automorphic function, i.e.~to consider $\Proj R(\sF,\lambda+\beta \Delta)$ where $\beta\in[0,1]$. This allows us to understand the case of quartics (or double EPW sextics) and, more importantly, to capture more GIT quotients. 

\medskip

While a variation of models $\Proj R(\sF,\lambda+\beta \Delta)$ makes sense  for general Type IV locally symmetric varieties (and also ball quotients), we focus here on the so called $D$-tower of locally symmetric varieties, i.e.~Type IV locally symmetric varieties associated to lattices $U^2\oplus D_n$. More precisely, we let $\cF(N)$ be the $N$-dimensional  locally symmetric variety corresponding 
to the lattice $\Lambda_N:=U^2\oplus D_{N-2}$ (so that $\dim\cF(N)=N$), and an arithmetic group $\Gamma_N$, which is intermediate between the   orthogonal group $O^{+}(\Lambda_N)$, and the stable orthogonal subgroup  $\wt{O}^{+}(\Lambda_N)$. 
With these definitions in place (see~\Ref{sec}{dtower} for details), $\cF(19)$ is the period space for quartic $K3$ surfaces, $\cF(20)$ is the period space for  double EPW sextics modulo the duality involution (or EPW cubes), and $\cF(18)$ is the period space for hyperelliptic quartic surfaces. Thus, we compare the Baily-Borel compactifications $\cF(N)^{*}$ for $N\in\{18,19,20\}$, to the GIT moduli spaces $\gM(N)$, where 
\begin{equation*}
\scriptstyle
 \gM(18)  :=  |\cO_{\PP^1}(4)\boxtimes\cO_{\PP^1}(4)|\gquot\Aut(\PP^1\times\PP^1),\quad \gM(19)  :=  |\cO_{\PP^3}(4)|\gquot\PGL(4),\quad 
 \gM(20)=\lagr\gquot\PGL_6(\CC).
\end{equation*}
(Actually, for $N=20$, we should take the quotient of $\gM(20)$ modulo the duality involution, see~\eqref{per20}.)
We let $\lambda(N)$ be the Hodge (or automorphic)  orbiline bundle on $\cF(N)$, and we choose  the boundary divisor $\Delta(N)$  so that $\Proj R(\sF(N),\lambda(N)+ \Delta(N))\cong\gM(N)$.

 The main result of the present paper is to predict the behavior of the models
 \begin{equation*}
 \sF(N,\beta):=\Proj R(\sF(N),\lambda(N)+\beta \Delta(N))
\end{equation*}
as $\beta$ varies between $0$ and $1$. Below we summarize our predictions.
\begin{prediction*}\label{pred:prelimpred}
Let $N\ge 15$. The ring of sections   $R(\cF(N), \lambda(N)+\beta\Delta(N))$ is finitely generated for $\beta\in[0,1]\cap \bQ$, and  the walls of the Mori chamber decomposition of the cone 
\begin{equation*}
\{\lambda(N)+\beta\Delta(N) \mid \beta\in\QQ,\ \beta>0\} 
\end{equation*}
 are generated by 
$\lambda(N)+\frac{1}{k}\Delta(N)$, where
$k\in\{1,\ldots,N-10\}$ and $k\not=N-11$.
The behaviour of $\lambda(N)+\frac{1}{k}\Delta(N)$, for $k$ as above, is described as follows. For $k=1$, $\cF(N,1)$ is obtained from $\cF(N,1-\epsilon)$ by contracting the strict transform of $\supp\Delta(N)$. 
If $2\le k$, then the birational map between $\cF(N,\frac{1}{k}-\epsilon)$ and 
$\cF(N,\frac{1}{k}+\epsilon)$ is a flip whose center is described in~\Ref{pred}{mainpred}.
\end{prediction*}

As explained below, while these predictions are only conjectural, we have strong evidence that they are accurate for $N\in\{18,19\}$, and they represent a significant improvement (both qualitative and quantitative) on what was previously known. Below we spell out our predictions for $N=19$, i.e.~$K3$ surfaces of degree $4$.
\begin{prediction*}
The variation of  models $\Proj R(\sF(19),\lambda(19)+\beta \Delta(19))$ interpolating between Baily-Borel ($\beta=0$) and GIT for quartic surfaces ($\beta=1$) undergoes birational transformations (flips, except the two boundary cases)
at the following critical values for $\beta$:
$$\beta\in\left\{0,\frac{1}{9},\frac{1}{7},\frac{1}{6},\frac{1}{5},\frac{1}{4},\frac{1}{3},\frac{1}{2},1\right\}.$$
Furthermore, the centers of the flips for $\beta<\frac{1}{5}$ correspond to $T_{3,3,4}$, $T_{2,4,5}$, and $T_{2,3,7}$ marked $K3$ surfaces (loci inside the period space $\cF(19)$) and those loci are flipped to the loci in $\gM(19)$ of quartics with $E_{14}$, $E_{13}$, an $E_{12}$ singularities  respectively (see Shah \cite{shah}). 
\end{prediction*}

\begin{remintro} In the companion announcement \cite{announcement}, we give a complete geometric (part conjectural, part provable) matching between  Shah's strata in the GIT quotient $\overline \calM$ (e.g. the $E_r$ stata mentioned, $r=12,13,14$) and the strata resulting from the flips predicted  above. Furthermore, the extremal cases can be inferred from the work of Shah and Looijenga. Specifically, the morphism $\sF(\epsilon)\to \sF(0)\cong \sF^*$ 
is  Looijenga's $\bQ$-factorialization of the (closure of the) Hegneer divisor $\Delta$ as discussed in \cite{looijengacompact} (in particular, $\sF(\epsilon)$ is a semi-toric compactification in the sense of Looijenga). At the other extreme, the map $\cF(1-\epsilon)\to \cF(1)\cong \overline \calM$ is the divisorial contraction (of the strict transforms of) hyperelliptic and unigonal divisors in the moduli of quartic $K3$s to the GIT polystable points corresponding to the double quadric and to the tangent developable to the twisted cubic respectively. Alternatively, $\cF(1-\epsilon)\to \cF(1)\cong \overline \calM$ is a Kirwan type blow-up of  the GIT quotient $\overline\calM$ at those two special points (as discussed in \cite[Sect. 3 and 4]{shah4}). For comparison's sake, we recall that in the case of  degree $2$ $K3$ surfaces, there is no intermediary flip and thus the two extremal cases suffice to compare the GIT and the Baily-Borel compactifications (see \cite{shah} and \cite{looijengavancouver}). Similarly, in the case of for cubic fourfolds (\cite{gitcubic,cubic4fold}, \cite{lcubic}), there is only one  flip, and hence the picture is much simpler. 
\end{remintro}

The variational approach to studying birational maps is the natural one from the perspective of contemporary MMP (\cite{bchmk}), and its power in the context of moduli has been in evidence since the appearance of  Thaddeus' work on VGIT~\cite{thaddeus}. An inspiration for our work is the so called Hassett-Keel program for moduli of curves (which studies the variation of log canonical models $\Proj R(\overline M_g,K+\alpha \Delta)$ of the Deligne-Mumford compactification $\overline M_g$). Experts in the field have speculated on the existence of an analogue of the Hassett-Keel program (see \cite{hh1,hh2}) for (special) surfaces; for instance, one can hope to understand the elusive KSBA compactification by starting with a GIT compactification and interpolating (see~\cite{gallardo}). Our study can be viewed as a first example of a Hassett-Keel program for surfaces. Indeed, beyond the obvious analogy (i.e. $\lambda+\beta \Delta$ can be easily rewritten as $K_\calF+\alpha \Delta'$), the modular behavior is also similar: for instance, the first birational wall crossing for quartic $K3$s is associated to Dolgachev (or triangle) singularities in a manner similar to the case of curves with cusp singularities (which gives the first birational modification for $\Proj R(\overline M_g,K+\alpha \Delta)$ at $\alpha=\frac{9}{11}$). The main point that we want to emphasize here is that the birational transformations that occur are controlled by the arithmetic and combinatorics of the hyperplane arrangement associated to $\Delta$. While the picture that we discover for quartic $K3$ surfaces is more complicated and subtle than that for  $K3$ surfaces of degree $2$, the fundamental insight of Looijenga that arithmetic controls the birational models of $\sF_d^*$ still holds true. We view our work as a quantitative and qualitative refinement of Looijenga's seminal work \cite{looijengacompact}. 

\medskip

The birational geometry of the moduli of $K3$ surfaces was previously studied from the perspective of the Kodaira dimension, most notably by Gritsenko--Hulek--Sankaran \cite{ghs}. We share with them the main technical tool (beyond standard lattice theory), namely Borcherds' construction (\cite{borcherds}) of automorphic forms for Type IV domains  (and subsequent improvements due to Bruinier \cite{bruinierbook} and Bergeron et al.~\cite{pick3nl}). One point that is worth noting in contrast to \cite{ghs} or Hassett-Keel is that for $K3$s it is more natural to focus on the Hodge bundle $\lambda$ versus the canonical bundle $K_{\calF}$, leading to a parameter ``$\beta$'' vs. the usual ``$\alpha$'' (of course, the two are tightly connected). 

\medskip

A purist' approach would be to analyze the   rings of sections $R(\sF(N), \lambda(N)+\beta \Delta(N))$  from the standpoint of 
 locally symmetric varieties, and their toroidal (or semitoric) compactifications. 
We  have not followed this path, in fact we failed to establish  the first step in a purist' approach, namely the proof that the ring of sections $R(\sF(N), \lambda(N)+\beta \Delta(N))$ is finitely generated for $\beta\in[0,1]\cap\QQ$.  We do expect finite generation   to hold in wide generality, and there exists a plausible approach based on \cite{haconxu} (improvements of \cite{bchmk}); this would also follow from the abundance conjecture.  
On the other hand, a purist' approach would only accomplish half of what is desirable. We are not content with a decomposition of the period map (or its inverse) as a composition of flips, we want to identify the centers of the flips as geometrically significant loci in the GIT moduli space. We can illustrate this point with the example of the locus in the GIT moduli space of quartics that parametrizes surfaces with an $E_{12}$ singularity, and the problem of deciding   which is the locus in the period space which corresponds to it.
Our methods predict that the relevant   locus is that of $T_{2,3,7}$-marked $K3$ surfaces (this  is related to another paper of Looijenga \cite{ltriangle1} and some unpublished work of Shepherd-Barron), and that the corresponding flip occurs for $\beta=1/9$. We believe that a \lq\lq mixed\rq\rq\ approach, involving both  the period space and the various GIT quotients will be more productive.  This will be the subject of future work. Following this approach  we have already confirmed part of the predictions for the period space of hyperelliptic quartics.

\subsection*{Structure of the paper}

In~\Ref{sec}{dtower}, we introduce the $D$-tower of period spaces $\{\cF(N)\}_{N\ge 3}$. The focus here is on the arithmetic of $D$ lattices. We define the main Heegner divisors, namely the nodal, hyperelliptic and unigonal divisors  in $\cF(N)$, denoted respectively $H_n(N)$,  $H_h(N)$ and $H_u(N)$. If $N=19$ (the period space for quartic surfaces) they have a well-known geometric meaning, in other dimensions they are the arithmetic analogue of the divisors for $N=19$. 
 The salient point in the $D$ tower  is that $\cF(N-1)$ is isomorphic to the hyperelliptic Heegner divisor $H_h(N)$  of $\cF(N)$. Our boundary divisor $\Delta(N)$ is equal to $H_h(N)/2$ except  when $N\equiv 3,4\pmod{8}$, in which case    $\Delta(N)=(H_h(N)+H_u(N))/2$. (N.B.~the factor $\frac{1}{2}$ appears because the quotient map from the relevant Type IV domain to $\cF(N)$ is ramified over the hyperelliptic divisor, and also over the  unigonal divisor, if $N\equiv 3,4\pmod{8}$.)
 This leads to an inductive behavior.

 In the following~\Ref{sec}{gitper}, we go through  results which are (more or less) known, namely that  $\cF(19)$ is the period space of degree $4$ polarized $K3$ surfaces, that $\cF(18)$  is the period space of 
  hyperelliptic quartics, and that $\cF(20)$  is the period space of double EPW sextics up to the duality involution (alternatively, it is the priod space of EPW cubes).
 
In~\Ref{sec}{picdtower} we study the Picard group of of $\cF(N)$, for $N\le 25$. First we compute the Picard groups for the D-tower (\Ref{thm}{thmrankpic}) for $N\le 20$, 
by applying  work of Borcherds, Brunier, and Bergeron. Next  we study the quasi-pullback of Borcherds' celebrated automorphic given by two embeddings 
of the lattice $\Lambda_N$ into $\II_{2,26}$. The resulting relations between the automorphic line bundle on $\cF(N)$, and the Heegner divisors $H_n(N)$, $H_h(N)$, $H_u(N)$ are key results for our work. 

\Ref{sec}{whydelta} contains a proof that   $\cF(19,1)=\Proj R(\cF(19), \lambda(19)+\Delta(19))$ is isomorphic to the GIT moduli space $\gM(19)$. An analogous proof shows that $\cF(18,1)$  is isomorphic to the GIT moduli space $\gM(18)$, and we expect that a similar result  holds for $N=20$, i.e.~that  $\cF(20,1)$  is isomorphic to  $\gM(20)$   modulo  the duality involution. 
In other words, if $N\in\{18,19\}$, then   the schemes $\cF(N,\beta)$, for $\beta\in(0,1)\cap\QQ$ interpolate between the Baily-Borel compactification $\cF(N)^{*}$, and the GIT compactification $\gM(N)$. 

In the last section, i.e.~\Ref{sec}{predictions},   we identify the critical values of $\beta$ and the corresponding centers. To describe our work, let us review the basic picture in Looijenga. The starting point for him is to note that frequently the GIT quotients $\overline \calM$ are obtained by contracting a certain Heegner divisor $H$ in $\calF$ (see also \cite{ls2} and \cite{zarhin} for some Hodge theoretic explanation of why this tends to happen). For instance, for degree $2$ $K3$ surfaces, the unigonal divisor $H_u$ (corresponding to elliptic $K3$s) is contracted to the point corresponding to the triple conic (inside the GIT moduli for plane sextics). Thus, the question is how can one achieve the contraction of $H$ starting from $\sF=\cD/\Gamma$? Now the second fundamental observation of Looijenga is the following: We have that the Type IV domain $\cD$ is an open subset of a quadric in $\bP^N$, i.e. $\cD\subset \check \cD\subset \bP^N$ where $\check\cD$ is the compact dual. In this realization, the Hodge or automorphic bundle $\lambda$ on $\sF$ is just the descent of the restriction of $\calO(-1)$ on $\bP^N$. On the other hand, the Heegner divisor $H$ is nothing but $\cH/\Gamma$ the descent of a $\Gamma$-invariant hyperplane arrangement $\cH$ in $\bP^N$, and thus should morally speaking correspond to $\lambda^{-1}$ (i.e. to $\calO_{\bP^N}(1)$). It follows that the line bundle that contracts $H$ is precisely (in a stack sense) $\lambda+H$, leading to the description of the GIT quotient as a $\Proj$ of a ring of meromorphic automorphic forms. The complication of recovering the GIT quotient, say $\overline \calM\cong \Proj R(\sF, \lambda+H)$, from $\sF^*=\Proj R(\sF,\lambda)$ comes from the fact that the hyperplanes in the arrangement $\cH$ intersect, according to \cite{looijengacompact} the intersection strata have to be flipped starting with the largest codimension. For degree $2$ $K3$s (\cite{looijengavancouver}, \cite{shah}) there are no intersections, and thus no flip. For cubic fourfolds (\cite{gitcubic,cubic4fold}, \cite{lcubic}) there is only intersection in codimension $2$, and thus a single flip that can be understood geometrically. For quartic surfaces (and the related examples, EPWs, etc.), we have the opposite behavior: the hyperelliptic hyperplanes whose union is $\cH_h$ intersect up to dimension $0$ -  in particular a critical assumption in~\cite{looijengacompact} is not satisfied. It is clear (inspired by \cite{thaddeus} and the subsequent related examples) that in order to understand the quartic examples, or other examples where lots of intersections occurs (not necessarily up to dimension 0, and thus even cases nominally covered by \cite{looijengacompact}) it is wise to consider a continuous variation of parameter $\lambda+\beta H$. Roughly, $\beta=0$ is the known case, $\beta=1$ is the target, and in between we understand the flipping behavior by wall-crossings. By Looijenga we know that the various intersection strata of $\cH$ are flipped roughly speaking in order of codimension. We  predict that the value of $\beta$ for which
a specified intersection stratum $Z$  will be flipped is determined by the log canonical threshold of $Z\subset \bP^N$ (in the sense of hyperplane arrangements). Of course, additionally, one has to  take into account the ramification behavior and the interaction with the hyperplane arrangement. 
We view these predictions as a quantitative refinement of Looijenga's work -- let us call them \emph{first-order} predictions. The surprising thing, and it is here that our extensive  analysis of the D-tower is essential, is that there exists also a qualitative refinement, as explained below. 

\medskip

By comparing the GIT quotient for quartic surfaces of Shah \cite{shah4} with the predicted Looijenga \cite{looijengacompact} model (while the full results on \cite{looijengacompact} do not apply here, the basic framework is still valid), we have observed a puzzling thing: the  flips associated to intersections of components of the hyperelliptic arrangement  predicted by Looijenga do  occur, but,  roughly speaking, only up to half the dimension (codimension 8).  The explanation for this is in the relations between the automorphic line-bundle and the Heegner divisors that were established   in~\Ref{sec}{picdtower}. In short, while, for large $N$, $H_h(N)$  is birationally contractible in $\cF(N)$,  
the opposite is true for $N$ small: in fact $H_h(N)$ is  ample on $\cF(N)$ (and proportional to $\lambda(N)$) for $N\le 10$.  This last fact was previously observed by Gritsenko \cite{gritsenko} in a different context. A related phenomenon occurs  in the work of Gritsenko--Hulek--Sankaran: the higher the dimension for the Type IV domain the easier is to establish general type results (e.g.~compare the bounds that GHS obtain for K3 surfaces and for polarized HK deformation equivalent to $K3^{[2]}$). Or in other words, the smaller the dimension of $\cD$, the more unexpected relations in $\Pic(\sF)$ occur leading to corrections to the first-order predictions. (In summary, Looijenga/first-order predictions come from combinatorics of the hyperplane arrangement - this is the dominating factor for high dimensions/small codimension, while for small dimensions there are arithmetic corrections that become dominating).

\subsection*{Future work} While the results of this paper only give predictions on the behavior of the variation of models given by $\Proj R(\cF(N), \lambda(N)+\beta\Delta(N))$, we believe that these are quite accurate. In particular, we are able to essentially verify them in the case of hyperelliptic quartics, i.e.~$N=18$. Specifically, in that situation we are discussing about the moduli of $(2,4)$ complete intersection curves. Thus, we can obtain the finite generation of the ring sections considered here, as well as the critical values by invoking GIT/VGIT for $(2,4)$ complete  intersections. The techniques here are quite similar to those \cite{g4git,g4ball}. The details will appear elsewhere. The verification of the predictions for quartics (and then EPW sextics) is more challenging and it might require new ideas. Specifically, the natural approach via GIT (analogue to the case of hyperelliptic quartics) would involve computing GIT quotients for quartic surfaces embedded by $\calO(\nu)$ (for $\nu\ge 2$), which is beyond the available technology (see however Gieseker \cite{gieseker}, Donaldson \cite{donaldson} for some results on asymptotic GIT for surfaces).  Our hope is that a mixture of GIT and birational geometry will allow us to complete the analysis for quartics (and double EPW sextics). For instance, control of GIT in codimension $1$ would suffice to obtain finite generation for $R(\cF(19), \lambda(19)+\beta \Delta(19))$. In fact, we believe it would be worthwhile to revisit the  case of $K3$ surfaces of degree $2$, and that of cubic fourfolds. 

\medskip

In a different direction, one can investigate the behavior of $\Proj R(\cF, \lambda+\beta\Delta)$ for more general Type IV locally symmetric varieties, with an eye on the asymptotic GIT for $K3$ surfaces. More specifically, as we explained, a key assumption that we are making here is the finite generation of the ring of sections of $R(\calF, \lambda+\beta \Delta)$. Since this is a strictly log canonical case, the finite generation does not follow from \cite{bchmk}, but would follow from Abundance Conjecture. As explained above, for the D-tower, one can probably obtain finite generation by means of GIT. In general, we believe that the finite generation can be obtained by using Hacon-Xu \cite{haconxu} improvements of \cite{bchmk} (essentially the dlt case), and appropriate toroidal compactifications. The striking consequence of finite generation is that if there exists a sequence of projective models $\sF^{(\nu)}$, isomorphic in codimension $1$ to the Baily-Borel model $\sF^*$,  and such that polarization $\calL^{(\nu)}$ converges to the Hodge bundle $\lambda$, then they eventually stabilize to $\sF^{(\infty)}$  which will further map to $\sF^*$. We expect this to have consequences on the asymptotic GIT for $K3$ surfaces (N.B. by Mumford \cite{mumford_stab}, Shepherd-Barron \cite{sbnef}, and Wang--Xu \cite{wangxu}, it is known that the naive stabilization doesn't work for $K3$ surfaces). Namely, by a computation essentially due to Mumford-Knudsen, it is known that the $\calL^{(\nu)}$  on the Chow GIT for $\nu$-embedded $K3$ surfaces is a polynomial in $\nu$ with coefficients the tautological classes ($\lambda$, $\kappa_{1,1}$ and $\kappa_{3,0}$) introduced in Marian-Oprea-Pandharipande \cite{mop}. The coefficient of the leading term is $\lambda$ and thus we are in the situation mentioned above once a certain codimension $1$ GIT was established.  

\medskip

Finally, a topic that we are not touching in this paper is the birational geometry of $\sF^*$ at the boundary. This is so because the boundary here is always of high codimension. However, by passing to a toroidal compactification $\overline{\sF}^\Sigma$, the boundary becomes divisorial, and in certain situations it might be of interest to vary the coefficient of the exceptional divisor(s) $E$ of $\overline{\sF}^\Sigma\to \sF^*$. The simpler case of ball quotients (and without an additional Heegner divisor $H$) was analyzed by Di Cerbo-Di Cerbo. In particular, this leads to the intriguing question of the existence of a canonical compactification (in sense of MMP) for $K3$ surfaces, analogous to the results of Shepherd-Barron for abelian varieties. 

\medskip

In summary, we view our paper as a proof of concept for a Hassett--Keel--Looijenga program for studying the log canonical models of moduli spaces of polarized $K3$ surfaces.  In fact, we believe that Hassett--Keel--Looijenga program for $K3$ surfaces has a richer structure than the Hassett--Keel program for curves, and it will lead to deeper applications. This is due to the fact that the confluence of geometry and arithmetic is essentially unique to the case of $K3$ (or rather $K3$ type) situation. (Of course, also for abelian varieties the moduli space is a locally symmetric variety. But, as observed by Looijenga, for $K3$ type (and ball quotients) the (Heegner) divisors control the geometry of the locally symmetric variety; this is not true for abelian varieties.
 This is due to (an essentially group theoretic fact) that the Hodge special loci inside the period domains for $K3$s are much simpler than in the abelian varieties case (see Zarhin \cite{zarhin}).)

\subsection*{Acknowledgement}
This project was started and first developed while both authors were visiting the IAS as participants in the special program on ``{\it Topology of Algebraic Varieties}''. We are grateful to the IAS for hosting us and offering an ideal research environment. The IAS stay of the first author was partially supported by NSF institute grant DMS-1128155. The stay of the second author was partially supported by the Giorgio and Elena Petronio Fellowship Fund, the Giorgio and Elena Petronio
Fellowship Fund II, and the Fund for Mathematics of the IAS.

\smallskip

We are indebted in more than one way to E.~Looijenga. First, our work has been inspired by his groundbreaking papers~\cite{looijenga1,looijengacompact}, secondly 
  we had several discussions with him on the subject of this work.  Lastly, Looijenga shared some unpublished notes with us (relevant to future investigation of the quartic case). Additionally, we have discussed with numerous people at various stages of the project. In particular, we thank V. Alexeev, S. Casalaina-Martin, G. Farkas, V. Gritsenko, K. Hulek, B. Hassett, S. Kondo, Zhiyuan Li, D. Oprea, R. Pandharipande, and Letao Zhang.

\section{$D$ lattices and the associated locally symmetric varieties}\label{sec:dtower}
We introduce   $D$  locally symmetric varieties. There is one such variety (of dimension $n+2$) for each $D_n$ lattice. The period spaces 
 of hyperelliptic quartic $K3$'s,  quartic $K3$'s, and desingularized EPW sextics   correspond to $D_{16}$,  $D_{17}$, and  $D_{18}$ respectively, see~\Ref{sec}{gitper}. We will introduce the Heegner divisors on $D$  locally symmetric varieties which are relevant for our work: the nodal, hyperelliptic, and unigonal divisors. These divisors are the generalization of the familiar divisors with the same name on the period space for quartic $K3$ surfaces. We will prove that the hyperelliptic Heegner divisor on a $D$ period space of dimension $N$  is isomorphic to the  $D$  locally symmetric variety of dimension $N-1$; thus we have an infinite tower of  nested  $D$  locally symmetric varieties. 
  
  $D$  locally symmetric varieties of dimension up to $10$ have appeared in work of Gritsenko, see~\cite[\S3]{gritsenko}.
\subsection{$D$ lattices}
\setcounter{equation}{0}
Let $\Lambda$ be a lattice, i.e.~a finitely generated torsion-free abelian group equipped with an integral non-degenerate bilinear symmetric form $(,)$. If $v\in \Lambda$ we let $v^2=(v,v)$.
  Given  a ring $R$ we let $\Lambda_R:=\Lambda\otimes_{\ZZ}R$, and we extend by linearity the quadratic form to $\Lambda_R$. We let   $A_{\Lambda}:=\Hom(\Lambda,\ZZ)/\Lambda$ be the \emph{discriminant group} of $\Lambda$.  The quadratic form defines an embedding   $\Hom(\Lambda,\ZZ)\subset \Lambda_{\QQ}$. Let $v\in \Lambda$.  The \emph{divisibility} of $v$ is defined to be the positive generator of $(v,\Lambda)$, and is denoted
 $\divisore(v)$. Thus $v/\divisore(v)$ is an element of $\Lambda_{\QQ}$ which belongs to $\Hom(\Lambda,\ZZ)$, and hence it determines an element   $v^{*}\in A_\Lambda$. Now suppose that $\Lambda$ is an \emph{even} lattice. The embedding   $\Hom(\Lambda,\ZZ)\subset \Lambda_{\QQ}$ induces a  \emph{discriminant quadratic form}   $q_{\Lambda}\colon A_{\Lambda}\to\QQ/2\ZZ$. We recommend the classical paper by Nikulin~\cite{nikulin} as a reference for definitions and results on lattices.  
 
 For us, $U$ is the standard hyperbolic plane, and the standard $ADE$ root lattices are negative definite. Thus $E_8$ is the unique even unimodular negative definite lattice of rank $8$; in the literature, sometimes, it is denoted by $E_8(-1)$ or $-E_8$. If $0<p\le q$ and  $p\equiv q \pmod 8$, then $\II_{p,q}$ stands for an even unimodular lattice of signature $(p,q)$; such a lattice  is unique up to isomorphism. For $m\in\ZZ$, we let $(m)$ be the rank one lattice with quadratic form that takes the value $m$ on a generator. We let 
 \begin{equation}\label{eccodi}
 D_{n}:=\left\{x\in\bZ^n\mid \sum_i x_i\equiv 0\pmod 2\right\}\subset \bZ^n,
\end{equation}
and we equip $D_n$ with the restriction of the negative of the standard Euclidean pairing on $\bZ^n$.  Thus $D_1\cong( -4)$, $D_2\cong A_1\oplus A_1$, $D_3=A_3$, and  $D_n$ is the negative definite root lattice with Dynkin  diagram $D_n$ if $n\ge 4$.
 Let $\alpha_n,\beta_n\in A_{D_n}$ be the classes  of $(1/2,1/2,\ldots,1/2)$ and   $(-1/2,1/2,\ldots,1/2)$  respectively.
If $n$ is odd then  $4\alpha_n=4\beta_n=0$, and $\alpha_n+\beta_n=0$. If $n$ is even, then $2\alpha_n=2\beta_n=0$. Moreover
\begin{equation}\label{discdien}
\scriptstyle
q_{D_{n}}(\alpha_n)=q_{D_{n}}(\beta_n)\equiv -n/4 \pmod{2\ZZ},\qquad q_{D_{n}}(\alpha_n+\beta_n)\equiv n+1 \pmod{2\ZZ}.
\end{equation}
The following result is an easy exercise.
\begin{claim}\label{clm:disdin}
If $n$ is odd, then $A_{D_n}$ is cyclic of order $4$, generated by  $\alpha_n$. If $n$ is even, then  $A_{D_{n}}$ is the Klein group,   generated by $\alpha_n$ and $\beta_n$.
\end{claim}
For $N\ge 3$, let
\begin{equation}
\Lambda_N:=U^2\oplus D_{N-2}.
\end{equation}
Notice that 
\begin{equation}\label{rugby}
(A_{\Lambda_N},q_{\Lambda_N})\cong (A_{D_{N-2}},q_{D_{N-2}}).
\end{equation}
Below is the key definition of the present subsection.
\begin{definition}\label{dfn:decoro}
Let  $N\ge 3$. A  lattice  $\Lambda$ is a  \emph{dimension-$N$ $D$-lattice} if it is  isomorphic to $\Lambda_N$, and in that case a \emph{decoration} of  $\Lambda$  is an element $\xi\in A_{\Lambda}$  of square $1$ (modulo $2\ZZ$). A \emph{dimension-$N$ decorated $D$-lattice} is a couple $(\Lambda,\xi)$, where $\Lambda$ is a 
dimension-$N$ $D$-lattice, and  $\xi\in A_{\Lambda}$ is a decoration. 
\end{definition} 
\begin{remark}\label{rmk:unicadec}
Let $\Lambda$ be a dimension-$N$ $D$-lattice. By~\eqref{rugby} and~\eqref{discdien} there exists  a decoration of $\Lambda$, and it is unique unless $N\equiv 6 \pmod 8$. Of course, when $\xi$ is unique the decoration is irrelevant. However, including the decorations allows us to treat the exceptional cases (i.e. $N\equiv 6 \pmod 8$) uniformly, and to make certain hidden structures more transparent. 
\end{remark}
\begin{remark} 
The dimension $N$ in~\Ref{dfn}{decoro} refers to the dimension of the associated period space, see~\Ref{dfn}{decper},  \emph{not} to the rank of the lattice $\Lambda$ (which is $N+2$). 
\end{remark}
\begin{remark}\label{rmk:periodeight}
Let $N\ge 3$; then 
\begin{equation}\label{periodeight}
\Lambda_N\oplus E_8\cong \Lambda_{N+8}.
\end{equation}
In fact the two  lattices have  same rank and signature, they are   even, and their discriminant groups and discriminant quadratic forms are isomorphic; thus they are isomorphic lattices by Theorem 1.13.2 of~\cite{nikulin}. Now write $N-2=8k+a$ with $0\le k$ and $a\in\{0,\ldots,7\}$.   By repeated application of~\eqref{periodeight}, one gets the isomorphism 
\begin{equation}
\Lambda_N\cong II_{2,2+8k}\oplus D_a.
\end{equation}
\end{remark}
\begin{remark}\label{rmk:solodim}
Any two $N$-dimensional decorated $D$ lattices are isomorphic, i.e.~given decorations $\xi,\xi'$ of a $D$ lattice $\Lambda$, there exists an isometry $g\in O(\Lambda)$ such that $g(\xi)=\xi'$.  In fact the group $O(q_{\Lambda})$ acts transitively on the set of decorations of $\Lambda$ (easy), and  the natural homomorphism  
$O(\Lambda)\to O(q_{\Lambda})$ is surjective. The last statement follows from  Theorem 1.14.2 of~\cite{nikulin}, or one can argue directly. In fact the result is trivial if  $N\not\equiv 6\pmod 8$, because there is a unique   $\alpha\in A_{\Lambda}$ of square $1$ (modulo $2\ZZ$) (see~\Ref{rmk}{unicadec}),  while if  $N\equiv 6\pmod 8$ all non-zero elements have square $1$  (modulo $2\ZZ$). In the latter case, write $N-2=4+8k$, and hence  $\Lambda_N$ can be identified with $II_{2,2+8k}\oplus D_4$. In this case, the   non-zero elements, are  given by $v^{*}$, where $v$ is one of
$$(0_{4+k},(1,1,1,1)),\quad (0_{4+k},(-1,1,1,1)), \quad (0_{4+k},(2,0,0,0)).$$
Acting by reflections  in these vectors, one checks that   $O(\Lambda)$ permutes the non-zero elements of $A_{\Lambda}$. 
\end{remark}
\subsection{Locally symmetric varieties of Type IV}\label{subsec:spaziperiodi}
\setcounter{equation}{0}
Suppose that $\Lambda$ is a lattice of \emph{signature} $(2,m)$. We let
\begin{equation*}
\cD_\Lambda:=\{[\sigma] \in\PP(\Lambda_{\CC})\mid \sigma^2=0,\ (\sigma+\ov{\sigma})^2>0\}.
\end{equation*}
Then $\cD_\Lambda$ is a complex manifold of dimension $m$, with two connected components  $\cD^{\pm}_\Lambda$, interchanged by complex conjugation (the \lq\lq real part\rq\rq\   projection $\Lambda_{\CC}\to\Lambda_{\RR}$ identifies  $\cD_\Lambda$ with the set of oriented positive definite $2$-dimensional subspaces of $\Lambda_{\RR}$).
Each of  $\cD^{\pm}_\Lambda$  is a bounded symmetric domain of Type IV. The orthogonal group $O(\Lambda)$ acts naturally on $\cD_\Lambda$ (left action). Let $O^{+}(\Lambda)<O(\Lambda)$ be the  subgroup of elements fixing each of $\cD^{\pm}_\Lambda$. By definition  $O^{+}(\Lambda)$ acts on $\cD^{+}_\Lambda$.  For  a finite-index subgroup $\Gamma< O^{+}(\Lambda)$ we let
\begin{equation*}
\cF_\Lambda(\Gamma):=\Gamma\backslash \cD^{+}_\Lambda.
\end{equation*}
Then $\cF_\Lambda(\Gamma)$ is naturally a complex space of dimension $n$. By a celebrated result of Baily and Borel, there exists  a projective variety  $\cF_\Lambda(\Gamma)^{*}$ (the \emph{Baily-Borel compactification} of  $\cF_\Lambda(\Gamma)$) containing    an open dense subset isomorphic to $\cF_\Lambda(\Gamma)$ as analytic space. In particular  $\cF_\Lambda(\Gamma)$  has a compatible structure of 
quasi-projective variety. 
\begin{definition}\label{dfn:decper}
Let  $(\Lambda,\xi)$ be a dimension-$N$  decorated $D$-lattice. Then $O(\Lambda)$ acts naturally on $A_\Lambda$, and hence we may define
$$\Gamma_\xi:=\{\phi\in O^+(\Lambda)\mid \phi(\xi)=\xi\}.$$
  We let $\cF_{\Lambda}(\Gamma_\xi)$ be the associated  locally symmetric variety (of dimension $N$) - this is a \emph{$D$ locally symmetric variety}.
 \end{definition} 
To  simplify notation we will write $\cF(\Lambda,\xi)$  for $\cF_{\Lambda}(\Gamma_\xi)$.

We will compare $\Gamma_{\xi}$ to more familiar subgroups of $O^+(\Lambda)$. First we recall that if $\Lambda$ is an even lattice, the \emph{stable orthogonal} subgroup  $\wt{O}(\Lambda)< O(\Lambda)$ is the kernel of the natural homomorphism  $\wt{O}(\Lambda)\to O(A_{\Lambda})$. We let $\wt{O}^{+}(\Lambda):=O^+(\Lambda)\cap \wt{O}(\Lambda)$. Notice that $\wt{O}^{+}(\Lambda)$ is of finite index in $O(\Lambda)$. 
\begin{example}\label{expl:riflessione}
 Let  $\Lambda$ be an even lattice, and $r\in\Lambda$  non isotropic. Let
\begin{equation}\label{riflessione}
\begin{matrix}
\Lambda_{\QQ} & \overset{\rho_r}{\lra} & \Lambda_{\QQ} \\
x & \mapsto & x-\frac{2(x,r)}{r^2}r
\end{matrix}
\end{equation}
be the reflection in the hyperplane $r^{\bot}$. If $r$ is primitive, then $\rho_r(\Lambda)\subset\Lambda$ if and only if $r^2| 2\divisore(r)$; if this is the case we use the same symbol $\rho_r$ for the corresponding element of $O(\Lambda)$. One checks easily the following results:
\begin{enumerate}
\item
$\rho_r\in O^{+}(\Lambda)$ if and only if $r^2<0$.
\item
If $r^2=\pm 2$, then $\rho_r\in \wt{O}(\Lambda)$.
\end{enumerate}
\end{example}
Let  $(\Lambda,\xi)$ be a  decorated $D$-lattice; then
\begin{equation}
\widetilde O^+(\Lambda)<\Gamma_\xi< O^+(\Lambda).
\end{equation}
 \begin{proposition}\label{prp:propgroup}
 Let $(\Lambda,\xi)$ be a decorated $D$-lattice of dimension $N$. Then
  \begin{enumerate}
 \item $\widetilde O^+(\Lambda)<\Gamma_\xi$ is of index $2$.
 \item $\Gamma_\xi= O^+(\Lambda)$ unless $N\equiv 6\pmod 8$, in which case $\Gamma_\xi< O^+(\Lambda)$ is of index $3$. 
 \item If $N$ is odd, $\cF_\Lambda(O^+(\Lambda))=\cF_{\Lambda}(\Gamma_\xi)= \cF_\Lambda(\widetilde O^+(\Lambda))$.
 \item If $N$ is even, the map $\cF_\Lambda(\widetilde O^+(\Lambda))\to \cF_{\Lambda}(\Gamma_\xi)$
 is a double cover. 
  \end{enumerate}
 \end{proposition}
 \begin{proof}
 The homomorphism $O(\Lambda)\to O(q_\Lambda)$ is surjective, see~\Ref{rmk}{solodim}.  Since the reflection $\rho_r$, for $r\in\Lambda$  of square $2$, acts trivially on $A_{\Lambda}$ and does not belong to the index-$2$ subgroup  $O^{+}(\Lambda)$, it follows that the homomorphism $O^{+}(\Lambda)\to O(q_\Lambda)$ is surjective as well. 
  Since $O(q_\Lambda)\cong \bZ/2$ for $N\not\equiv 6\pmod 8$, and $O(q_\Lambda)\cong \Sigma_3$ otherwise, Items~(1) and~(2) follow.   Notice that $-1 (=-\textrm{id}_\Lambda)\in \Gamma_\xi$  for all $N$, while $-1\in \widetilde O(\Lambda)$ if and only if $N$ is even.  It follows that  if $N$ is odd, then $-1$ generates $\Gamma_\xi/\widetilde O^+(\Lambda)\cong \bZ/2$. Since $-1$ acts trivially on $\cD_\Lambda$, Item~(3) follows. Item~(4) holds because $\widetilde O^+(\Lambda)<\Gamma_\xi$ is of index $2$, and $-1$ is the unique non trivial element of $O^{+}(\Lambda)$ acting trivially on  $\cD^{+}_{\Lambda}$.
 \end{proof}
\subsection{Heegner divisors on $D$ locally symmetric varieties}\label{sec:divisorclasses}
\setcounter{equation}{0}
\subsubsection{Divisor classes on locally symmetric variety  of Type IV}\label{subsubsec:divfour}
Let $X:=\cF_{\Lambda}(\Gamma)$ be   a locally symmetric variety  of Type IV. The \emph{Hodge bundle} (or automorphic bundle), denoted by $\cL(\Lambda,\Gamma)$,  is a fundamental fractional (orbifold) line-bundle on   $\cF_{\Lambda}(\Gamma)$;  it  is defined as the quotient of $\cO_{\cD^{+}_{\Lambda}}(-1)$ by $\Gamma$, where    $\cO_{\cD^{+}_{\Lambda}}(-1)$ is the restriction to $\cD^{+}_{\Lambda}$ of the tautological line-bundle on $\PP(\Lambda_{\CC})$. We recall that $\cL(\Lambda,\Gamma)$ extends to an ample fractional line bundle $\cL^{*}(\Lambda,\Gamma)$ on the Baily-Borel compactification $\cF_{\Lambda}(\Gamma)^{*}$, and that the sections of  $m\cL^{*}(\Lambda,\Gamma)$ are precisely the weight-$m$ $\Gamma$-automorphic forms. We let $\lambda(\Lambda,\Gamma):=c_1(\cL(\Lambda,\Gamma))$; thus $\lambda(\Lambda,\Gamma)$ is a $\QQ$-Cartier divisor class.

Any Weil divisor $D$ on the quasi-projective variety $X$ is  $\QQ$-Cartier. In fact  there exists a  finite index subgroup $\Gamma_0<\Gamma$ acting freely on $\cD^{+}_{\Lambda}$,  hence  $X_0:=\cF_{\Lambda}(\Gamma_0)$ is a smooth quasi-projective variety with a finite proper map  $\rho\colon X_0\to X$. Now, let us assume that $D$ is a prime divisor (as we may). Then $\rho^{-1}D$ is a closed subset of $X_0$, of pure codimension $1$, and hence a Cartier divisor because $X_0$ is smooth. Then $\rho_{*}(\rho^{-1}D)$ is an effective (non-zero)  Cartier divisor supported on $D$, and we are done. Thus we may identify $\CH^1(X)_{\QQ}$ and  $\Pic(X)_{\QQ}$.  

Heegner divisors on $\cF_\Lambda(\Gamma)$  are defined as follows. Let  $\pi\colon \cD^{+}_\Lambda\to \cF_\Lambda(\Gamma)$ be the quotient map. 
 Given a \emph{non-zero} $v\in \Lambda$,  we let
\begin{equation*}
\cH_{v,\Lambda}(\Gamma):=\bigcup_{g\in\Gamma} g(v)^{\bot}\cap \cD^{+}_\Lambda,\qquad  H_{v,\Lambda}(\Gamma):=\pi(\cH_{v,\Lambda}(\Gamma)).
\end{equation*}
Notice that $g(v)^{\bot}\cap \cD^{+}_\Lambda$ is empty if $v^2\ge 0$, hence we will always assume that $v^2<0$. Then $\cH_{v,\Lambda}(\Gamma)$ is a (particular) hyperplane arrangement 
(see~\cite{looijenga1,looijengacompact}): we call it the \emph{pre-Heegner divisor}  associated to  $v$. 
One checks  that $H_{v,\Lambda}(\Gamma)$ is a  prime divisor in the quasi-projective variety  $\cF_\Lambda(\Gamma)$. We say that $H_{v,\Lambda}(\Gamma)$ is the \emph{Heegner divisor} associated to  $v$.  Notice that $\cH_{v,\Lambda}(\Gamma)$ and $H_{v,\Lambda}(\Gamma)$  depend only on the $\Gamma$-orbit of $v$. We say that  $\cH_{v,\Lambda}(\Gamma)$ and  $H_{v,\Lambda}(\Gamma)$ are \emph{reflective} if the reflection $\rho_v$ (see~\Ref{expl}{riflessione}) belongs to $\Gamma$ (in particular $\rho_v(\Lambda)=\Lambda$). If this is the case we say that $v$ is a \emph{reflective} vector of $(\Lambda,\Gamma)$. If $(\Lambda,\xi)$ is a decorated $D$ lattice, and $\Gamma=\Gamma_{\xi}$, then we say that $v$ is a reflective vector of $(\Lambda,\xi)$.
\subsubsection{Relevant Heegner divisors for $D$ locally symmetric varieties} 
Let $(\Lambda,\xi)$ be a decorated $D$ lattice. The Heegner divisors which are relevant for the present work are associated to vectors $v\in\Lambda$ (of negative square) which minimize $|v^2|$ among vectors such that $v^{*}$ equals a given element of $A_{\Lambda}$. It will be convenient to write
\begin{equation}
A_{\Lambda}=\{0,\zeta,\xi,\zeta'\}.
\end{equation}
Thus
\begin{equation}
q_{\Lambda}(\zeta)=q_{\Lambda}(\zeta')\equiv -(N-2)/4\pmod{2\ZZ},
\end{equation}
always, and
\begin{eqnarray}
2\zeta=2\xi=2\zeta'=0 & \text{if $N$ is even,}\\
2\zeta=2\zeta'=\xi & \text{if $N$ is odd.}
\end{eqnarray}
\begin{remark}\label{rmk:scambio}
With notation as above, there exists $g\in\Gamma_{\xi}$ such that $g(\zeta)=\zeta'$. In fact it suffices to let $g=(\Id_{U^2},\rho)$, where $\rho\in O^{+}(D_{N-2})$ is the reflection in the vector $(2,0,\ldots,0)$, i.e.~$\rho(x_1,\ldots,x_{N-2})=(-x_1,x_2,\ldots,x_{N-2})$. 
\end{remark}
The result below will be used throughout the paper in order to classify $\Gamma_{\xi}$-orbits of vectors of $\Lambda$.   
\begin{proposition}[Eichler's Criterion,  Proposition~3.3 of~\cite{ghs-abelianisation}]\label{prp:criteich}
Let $\Lambda$ be an even lattice which contains $U\oplus U$. Let $v,w\in \Lambda$ be non-zero and primitive. There exists $g\in\wt{O}^{+}(\Lambda)$ such that $g (v)=w$ if and only if $v^2=w^2$ and $v^{*}=w^{*}$. 
\end{proposition}
\begin{proposition}\label{prp:minorm}
Let $(\Lambda,\xi)$ be a decorated D lattice, of dimension $N$. The following hold:
\begin{enumerate}
\item
There exists $v\in\Lambda$ such that $v^2=-2$ and $\divisore(v)=1$, and it is unique up to $\Gamma_{\xi}$. (Notice that $v^{*}=0$.)
\item
There exists $v\in\Lambda$ such that $v^2=-4$, $\divisore(v)=2$, and $v^{*}=\xi$. Such a $v$   is unique up to $\Gamma_{\xi}$.
\item
Let $a\in\{0,\dots,7\}$ be the residue mod $8$ of $N-2$. There exists  $v\in\Lambda$ such that 
\begin{enumerate}
\item[(3a)] $v^2=-4a$, $\divisore(v)=4$, and $v^{*}=\zeta$ (or  $v^{*}=\zeta'$) if $N$ is odd,
\item[(3b)]  $v^2=-a$ and $\divisore(v)=2$, and $v^{*}=\zeta$  (or  $v^{*}=\zeta'$) if $N$ is even.
\end{enumerate} 
 Such a $v$   is unique up to $\Gamma_{\xi}$ (recall that $\Gamma_{\xi}$ exchanges $\zeta$ and $\zeta'$).
\end{enumerate}
\end{proposition}
\begin{proof}
(1) and~(2): Existence is obvious. Such a $v$ is unique up to $\wt{O}^{+}(\Lambda)$ by Eichler's Criterion, and hence also up to  $\Gamma_{\xi}$, because $\wt{O}^{+}(\Lambda)<\Gamma_{\xi}$.  
(3): Let us prove existence. Assume first that  $N-2\not\equiv 0 \pmod 8$, and hence $a\in\{1,\dots,7\}$. Let $N-2=8k+a$; by~\Ref{rmk}{periodeight} we may identify 
$\Lambda$ with $II_{2,2+8k} \oplus D_a$.
 If  $N$ is odd, the vector $v:=(0_{4+8k},(2,\dots,2))\in \Lambda_N$   satisfies~(3a), and if
$N$ is even, the vector  $v:=(0_{4+8k},(1,\dots,1))\in \Lambda_N$ satisfies~(3b). 
 If $N-2\equiv 0\pmod 8$, let  $N=8k+2$. Then 
$\Lambda\cong II_{1,1+8k} \oplus U(2)$ by   Theorem 1.13.2 of~\cite{nikulin}. With this identification understood, Item~(3b) holds for  $v:=(0_{2+8k},e)\in\Lambda$, where  $e\in U(2)$ is a primitive isotropic vector. Lastly, we prove unicity of $v$ up to $\Gamma_{\xi}$. Let $v_1,v_2$ be two vectors such that~(3a) holds for both, or~(3b) holds for both. If $v_1^{*}=v_2^{*}$, then $v_1,v_2$  are $\wt{O}^{+}(\Lambda)$-equivalent by  Eichler's Criterion, and hence they are $\Gamma_\xi$-equivalent because 
$\wt{O}^{+}(\Lambda)<\Gamma_\xi$. If $v_1^{*}\not= v_2^{*}$ then $\{v_1^{*},v_2^{*}\}=\{\zeta,\zeta'\}$, and hence by~\Ref{rmk}{scambio} there exists $g\in\Gamma_\xi$ such that $g(v_2)^{*}=v_1^{*}$; thus $v_1,v_2$ are $\Gamma_\xi$-equivalent  by Eichler's Criterion.
\end{proof}
Below is a key definition for all that follows.
\begin{definition}\label{dfn:nodhypuni}
Let $(\Lambda,\xi)$ be a decorated $D$ lattice.  
\begin{enumerate}
\item
A vector $v\in\Lambda$ is \emph{nodal} if Item~(1) of~\Ref{prp}{minorm} holds. The \emph{nodal pre-Heegner divisor}  
  and the \emph{nodal Heegner divisor}    are  $\cH_{v,\Lambda}(\Gamma_{\xi})$ and $H_{v,\Lambda}(\Gamma_{\xi})$  respectively, where  $v\in\Lambda$ is a nodal vector. We will denote them by $\cH_n(\Lambda,\xi)$ and $H_n(\Lambda,\xi)$  respectively. 
\item
A vector $v\in\Lambda$ is \emph{hyperelliptic} if Item~(2) of~\Ref{prp}{minorm} holds. 
The \emph{hyperelliptic pre-Heegner divisor}  
  and the \emph{hyperelliptic Heegner divisor}    are  $\cH_{v,\Lambda}(\Gamma_{\xi})$ and $H_{v,\Lambda}(\Gamma_{\xi})$  respectively, where  $v\in\Lambda$ is a  hyperelliptic vector. We will denote them by $\cH_h(\Lambda,\xi)$ and $H_h(\Lambda,\xi)$  respectively. 
\item
A vector $v\in\Lambda$ is \emph{unigonal} if Item~(3) of~\Ref{prp}{minorm} holds. The \emph{unigonal pre-Heegner divisor}  
 and the \emph{unigonal Heegner divisor}   are  $\cH_{v,\Lambda}(\Gamma_{\xi})$ and 
$H_{v,\Lambda}(\Gamma_{\xi})$  respectively, where  $v\in\Lambda$ is a unigonal vector. 
We will denote them by $\cH_u(\Lambda,\xi)$ and $H_u(\Lambda,\xi)$  respectively. 
Notice that if $N\equiv 2\pmod{8}$, then 
$\cH_u(\Lambda,\xi)=0$ and 
$H_u(\Lambda,\xi)=0$ because $v^2=0$. 
\end{enumerate}
 (The definition makes sense because by~\Ref{prp}{minorm} there is a single  $\Gamma_{\xi}$-orbit of nodal, hyperelliptic or unigonal vectors.)  
\end{definition}
\subsubsection{Reflective Heegner divisors and the boundary of $D$ period spaces} 
\begin{proposition}\label{prp:propram}
Let $(\Lambda,\xi)$ be a dimension $N$ decorated $D$ lattice. Let $v\in\Lambda$ be primitive, of negative square, i.e.~$v^2<0$. The reflection $\rho_v$ (see~\eqref{riflessione}) belongs to $\Gamma_\xi$ (i.e.~$v$ is a  reflective vector of $(\Lambda,\xi)$) if and only if $v$ is either nodal,   hyperelliptic, or $N\equiv 3,4\pmod{8}$ and $v$ is unigonal. 
\end{proposition}
\begin{proof}
Assume that $\rho_v$  belongs to $\Gamma_\xi$. Since $\rho_v(\Lambda)=\Lambda$, it follows that $v^2|2\divisore(v)$ (see~\Ref{expl}{riflessione}). On the other hand $A_{\Lambda}$ is isomorphic to $\ZZ/(4)$ if $N$ is odd, and to the Klein group if $N$ is even.  Thus one of the following holds:
\begin{enumerate}
\item[(a)]
$v^2=-2$ and $\divisore(v)\in\{1,2,4\}$. 
\item[(b)]
$v^2=-4$ and $\divisore(v)\in\{2,4\}$. 
\item[(c)]
$v^2=-8$ and $\divisore(v)=4$.
\end{enumerate}
Suppose that Item~(a) holds. If $\divisore(v)=1$, then $v$ is nodal. Thus $\rho_v\in\wt{O}^{+}(\Lambda)$ by~\Ref{expl}{riflessione}, and  hence $\rho_v\in \Gamma_{\xi}$.  If $\divisore(v)=2$, then $q_{\Lambda}(v^{*})\equiv -1/2\pmod{2\ZZ}$. By~\eqref{rugby}, \Ref{clm}{disdin} and~\eqref{discdien}, it follows that $N\equiv 4\pmod{8}$, and $v$ is unigonal.  In particular $\Gamma_{\xi}=O^{+}(\Lambda)$ by~\Ref{prp}{propgroup}, and hence 
$\rho_v\in \Gamma_{\xi}$ because $\rho_v\in O^{+}(\Lambda)$.  If $\divisore(v)=4$, then $q_{\Lambda}(v^{*})\equiv -1/8\pmod{2\ZZ}$. On the other hand $N$ is odd because $\divisore(v)=4$, and hence by~\eqref{rugby}, \Ref{clm}{disdin} and~\eqref{discdien}, it follows that $-(N-2)/4\equiv -1/8\pmod{2\ZZ}$, which is absurd. Thus the case  $v^2=-2$ and  $\divisore(v)=4$ does not occur. 

Now suppose that Item~(b) holds. If $\divisore(v)=2$, then either $v$ is hyperelliptic, or $N\equiv 6\pmod{8}$.  If $v$ is  hyperelliptic,  then  $\rho_v\in O^{+}(\Lambda)$  by~\Ref{expl}{riflessione}. Since $\rho_v(v)=-v$, and $\xi=v^{*}$, the reflection $\rho_v$ fixes the $2$-torsion element $\xi$, and hence belongs to $\Gamma_{\xi}$. If, on the other hand,  $N\equiv 6\pmod{8}$ and $v$ is not hyperelliptic, we may assume that $v^{*}=\zeta$. Then  $\rho_v(\xi)=\zeta'$ (let $N=8k+6$, so that $\Lambda\cong II_{2,2+8k}\oplus D_4$, and compute), and hence $\rho_v\notin\Gamma_{\xi}$. If $\divisore(v)=4$, then 
 $q_{\Lambda}(v^{*})\equiv -1/4\pmod{2\ZZ}$. By~\eqref{rugby}, \Ref{clm}{disdin} and~\eqref{discdien}, it follows that $N\equiv 3\pmod{8}$, and $v$ is unigonal.  Then $\Gamma_{\xi}=O^{+}(\Lambda)$ by~\Ref{prp}{propgroup}, and hence 
$\rho_v\in \Gamma_{\xi}$ because $\rho_v\in O^{+}(\Lambda)$.

Lastly let us show that Item~(c) cannot hold. In fact $q_{\Lambda}(v^{*})\equiv -1/2\pmod{2\ZZ}$, and $N$ is odd because $\divisore(v)=4$; 
by ~\eqref{rugby}, \Ref{clm}{disdin} and~\eqref{discdien}, it follows that $-(N-2)/4\equiv -1/2\pmod{2\ZZ}$, which is absurd. 
\end{proof}
\begin{corollary}\label{crl:reflheeg}
Let $(\Lambda,\xi)$ be a decorated $D$ lattice, of dimension $N$.  The nodal and hyperelliptic (pre)Heegner divisors are reflective.  The unigonal  (pre)Heegner divisor is reflective if and only if $N\equiv 3,4\pmod{8}$. 
\end{corollary}

Below is a key definition.
\begin{definition}
Let $(\Lambda,\xi)$ be a decorated $D$ lattice, of dimension $N$.  Let $\Delta(\Lambda,\xi)$ be the $\QQ$-Cartier divisor on $\cF(\Lambda,\xi)$ defined as
\begin{equation}\label{deltaforce}
\Delta(\Lambda,\xi):=
\begin{cases}
H_h(\Lambda,\xi)/2 & \text{if $N\not\equiv 3,4\pmod{8}$,}\\
(H_h(\Lambda,\xi)+H_u(\Lambda,\xi))/2 & \text{if $N\equiv 3,4\pmod{8}$.}
\end{cases}
\end{equation}
\end{definition}
\begin{remark}
The main goal of the present paper is to predict the behaviour of $\Proj$ of the ring of sections  of $\lambda(\Lambda,\xi)+\beta\Delta(\Lambda,\xi)$  for $\beta\in[0,1]\cap\QQ$. The choice of boundary divisor $\Delta(\Lambda,\xi)$ is dictated by the geometry for $N\in\{18,19\}$, see~\Ref{sec}{whydelta}. Here we notice that 
\Ref{crl}{reflheeg} gives a property which distinguishes the Heegner divisors  appearing in~\eqref{deltaforce} from the other Heegner divisors that we have been considering - with one exception, the nodal Heegner divisor. The hyperelliptic and the reflective unigonal divisors  are further distinguished by the property of being normal, see~\Ref{prp}{manyhyp}, \Ref{prp}{unigcong3}, and~\Ref{prp}{unigcong4}.
\end{remark}
\subsubsection{Heegner divisors for the stable orthogonal group}\label{subsubsec:heegorto}
Let $(\Lambda,\xi)$ be a dimension $N$ decorated $D$ lattice. The main object of study in this paper is the geometry of the locally symmetric variety $\cF_\Lambda(\Gamma_\xi)$. (The reason for this is the inductive behavior described in~\Ref{subsec}{diserie}). It is more standard to discard the decoration $\xi$, and consider $\cF_\Lambda(\widetilde O^+(\Lambda))$, the variety associated to the stable orthogonal group. If $N$ is odd, $\cF_\Lambda(\widetilde O^+(\Lambda))\cong \cF_\Lambda(\Gamma_\xi)$ and there is nothing to be said. On the other hand, if $N$ is even then  $\Gamma_{\xi}$ is an index-$2$ subgroup of $\wt{O}^{+}(\Lambda)$ (see \Ref{prp}{propgroup}), and hence we have a double cover map 
\begin{equation}\label{stabledec}
\cF_\Lambda(\widetilde O^+(\Lambda))\overset{\rho}{\lra} \cF_\Lambda(\Gamma_\xi).
\end{equation}
We will describe the inverse image by $\rho$ of the Heegner divisors $H_n$, $H_h$ and $H_u$. 
\begin{definition}\label{dfn:vetmin}
Let $\Lambda$ be a dimension-$N$ $D$-lattice. A \emph{minimal norm vector} of $\Lambda$ is a $v\in\Lambda$ such that one of the following holds:
\begin{enumerate}
\item
 $v^2=-2$ and $\divisore(v)=1$, or
\item
$v^2=-4$, and $\divisore(v)=2$, or
\item
$(N-2)\not\equiv 0\pmod{8}$ and,  letting $a$ be  the residue of $(N-2)$ modulo $8$,
\begin{enumerate}
\item[(3a)] $v^2=-4a$, and $\divisore(v)=4$,  if $N$ is odd, or else
\item[(3b)]  $v^2=-a$, and $\divisore(v)=2$,  if $N$ is even.
\end{enumerate} 
\end{enumerate}
\end{definition}
\begin{remark}\label{rmk:minorbit}
Given $\eta\in A_{\Lambda}$, there exists a minimal norm vector $v\in\Lambda$ such that $v^{*}=\eta$, and by Eichler's Criterion (\Ref{prp}{criteich}) the set of such minimal norm vectors
 is a single $\wt{O}^{+}(\Lambda)$-orbit. If $u\in\Lambda$ is another vector such that $u^2<0$ and $u^{*}=\eta$, then $u^2\le v^2$; this is the reason for our choice of terminology.
\end{remark}
\begin{definition}\label{dfn:heegort}
Let $\Lambda$ be a dimension-$N$ $D$-lattice, and $\eta\in A_{\Lambda}$. We let
\begin{equation*}
\cH_{\eta}(\Lambda):=\cH_{v,\Lambda}(\wt{O}^{+}(\Lambda)),\qquad H_{\eta}(\Lambda):=H_{v,\Lambda}(\wt{O}^{+}(\Lambda)),
\end{equation*}
where $v\in\Lambda$ is any minimal norm vector  such that $v^{*}=\eta$. (The definition makes sense by~\Ref{rmk}{minorbit}.)
\end{definition}
\begin{remark}\label{rmk:dispari}
Let $\Lambda$ be a dimension-$N$ $D$-lattice, with $N$ odd. Choose a decoration $\xi$ of $\Lambda$, and let $A_{\Lambda}=\{\xi,\zeta,\zeta'\}$, as usual. 
Then, under the identification $\cF_\Lambda(\widetilde O^+(\Lambda))\cong \cF_\Lambda(\Gamma_\xi)$ , we have $H_0(\Lambda)=H_n(\Lambda,\xi)$, $H_{\xi}(\Lambda)=H_h(\Lambda,\xi)$, 
and $H_{\zeta}(\Lambda)=H_{\zeta'}(\Lambda)=H_u(\Lambda,\xi)$ (notice that $\zeta'=-\zeta$ because $N$ is odd).
\end{remark}
\begin{proposition}\label{prp:rhopull}
Let $(\Lambda,\xi)$ be a dimension-$N$, decorated $D$-lattice, and assume that $N$ is even. Let $\rho$ be the double covering in~\eqref{stabledec}. Then (notation as in~\Ref{rmk}{dispari})
\begin{equation}\label{rhopull}
\rho^{*}H_n(\Lambda,\xi)=H_0(\Lambda),\qquad \rho^{*}H_h(\Lambda,\xi)=2H_{\xi}(\Lambda),\qquad \rho^{*}H_u(\Lambda,\xi)=H_{\zeta}(\Lambda)+H_{\zeta'}(\Lambda).
\end{equation}
\end{proposition}
\begin{proof}
Let $v$ be a hyperelliptic vector of $(\Lambda,\xi)$, and let  $\rho_v$ be the associated reflection. Then $\rho_v\in \Gamma_\xi$, but  $\rho_v\not \in \widetilde O^+(\Lambda)$ (see~\Ref{rmk}{scambio}). Thus the class of $\rho_v$ in $\Gamma_\xi/\widetilde O^+(\Lambda)$ is the covering involution of $\rho$. One may choose a nodal vector $w$ of  $(\Lambda,\xi)$ which is orthogonal to $v$. Then $\rho_v$ acts non-trivially on $w^{\bot}\cap\cD^{+}_{\Lambda}$; it follows that  the covering involution of $\rho$ maps $H_0(\Lambda)$ to itself and is not the identity. The first equality of~\eqref{rhopull} follows from this. On the other hand, $\rho_v$ acts trivially on $v^{\bot}\cap\cD^{+}_{\Lambda}$, and the second equality of~\eqref{rhopull} follows. 
Lastly,   $\rho_v$ switches the vectors $\zeta$ and $\zeta'$ in $A_\Lambda$ (cf. \Ref{rmk}{scambio}), and this proves the third equality of~\eqref{rhopull}.
\end{proof}
Applying $\rho_{*}$ to the equalities in~\eqref{rhopull}, we get the following result.
\begin{corollary}\label{crl:otildegamma}
Keep assumptions as in~\Ref{prp}{rhopull}. Then
\begin{eqnarray*}
\rho_{*}H_0(\Lambda)=2H_n(\Lambda,\xi),\qquad \rho_{*}H_{\xi}(\Lambda)=H_h(\Lambda,\xi),\qquad \rho_{*}(H_{\zeta}(\Lambda)+H_{\zeta'}(\Lambda))=2H_u(\Lambda,\xi).
\end{eqnarray*}
\end{corollary}
Below is the last result of the present subsection.
\begin{claim}\label{clm:reflheeg2}
Let $\Lambda$ be a  $D$ lattice. Choose a decoration $\xi$ of $\Lambda$, and let $A_{\Lambda}=\{\xi,\zeta,\zeta'\}$, as usual. Then, with respect to the group $\wt{O}^{+}(\Lambda_N)$, the following hold:
\begin{enumerate}
\item
 $H_0(\Lambda)$ is a reflective Heegner divisor.  
\item
 $H_{\xi}(\Lambda)$  is reflective  if and only if $N$ is odd. 
\item
$H_{\zeta}(\Lambda)$ is reflective if and only if $N\equiv 3,4\pmod{8}$, and similarly for $H_{\zeta'}(\Lambda)$. 
\end{enumerate}
\end{claim}
\begin{proof}
If $N$ is odd the result follows from~\Ref{crl}{reflheeg}, because $\cF_\Lambda(\widetilde O^+(\Lambda))= \cF_\Lambda(\Gamma_\xi)$. Suppose that $N$ is even. Item~(1) holds because the reflection associated to $v\in\Lambda$ with $v^2=-2$ and $\divisore(v)=1$  belongs to $\widetilde O^+(\Lambda)$. On the other hand, as noted above, if $v$ is a hyperelliptic vector of $(\Lambda,\xi)$, then 
 $\pm \rho_v\not \in \widetilde O^+(\Lambda)$, and Item~(2) follows.  In oder to prove Item~(3), let $\rho$  be the double 
covering in~\eqref{stabledec}. If $H_{\zeta}(\Lambda)$ is reflective, then $\rho_{*}$ is reflective as well, and hence $N\equiv 3,4\pmod{8}$ by~\Ref{crl}{reflheeg}. On the other hand, one easily checks that if 
 $N\equiv 3,4\pmod{8}$, the reflection associated to  a unigonal vector is in $\widetilde O^+(\Lambda)$. 
 \end{proof}
\subsection{Hyperelliptic Heegner divisors}\label{subsec:hyperheeg}
\setcounter{equation}{0}
We will prove that the hyperelliptic Heegner divisor $H_h(\Lambda_N,\xi_N)$ is isomorphic to $\cF(\Lambda_{N-1},\xi_{N-1})$. 
\begin{lemma}\label{lmm:prophyp}
 Let $N\ge 4$, and let $(\Lambda,\xi)$ be a dimension-$N$ decorated $D$ lattice, and let $v\in \Lambda$ be a hyperelliptic vector. Then $v^{\bot}$ is a dimension-$(N-1)$  $D$ lattice. 
\end{lemma}
\begin{proof}
We may assume that $\Lambda=\Lambda_N=U^2\oplus D_{N-2}$.
By Eichler's Criterion, i.e.~\Ref{prp}{criteich}, any two vectors of $\Lambda_N$ of square $-4$ and divisibility $2$ are $O^{+}(\Lambda_N)$-equivalent.  Thus we may suppose  that $v=({\bf 0},(0,\ldots,0,2))$; it  is obvious that $v^{\bot}\cong \Lambda_{N-1}$. 
\end{proof}
\begin{remark}\label{rmk:inducedec}
Let  $N\ge 4$, let $(\Lambda,\xi)$ be a dimension-$N$ decorated  $D$ lattice, and let $v\in \Lambda$ be  hyperelliptic. The   $D$ lattice $v^{\bot}$ comes with a decoration. In fact, since $v^2=-4$ and $\divisore(v)=2$, the sublattice $\la v\ra\oplus v^{\bot}$ has index $2$ in $\Lambda$. Thus  there exists $w\in v^{\bot}$, well-determined modulo $2v^{\bot}$, such that $(v+w)/2$ is contained in $\Lambda$.  It follows that $(w/2,v^{\bot})\subset\ZZ$, and hence $w/2$ represents an element $\eta\in A_{v^{\bot}}$,  independent of the choice of $w$. Moreover  $q_{v^{\bot}}(\eta)\equiv 1\pmod{2}$ because
$$-1+(w/2)^2=(v/2)^2+(w/2)^2=((v+w)/2)^2\equiv 0 \pmod{2\ZZ}.$$
Thus $(v^{\bot},\eta)$ is a dimension-$(N-1)$  decorated $D$-lattice. 
\end{remark}
\begin{proposition}\label{prp:decest}
 Let $(\Lambda,\xi)$ be a dimension-$N$ decorated $D$ lattice, and $v\in\Lambda$ be a  hyperelliptic vector. Let $\eta$ be the decoration of the dimension-$(N-1)$  $D$ lattice $v^{\bot}$ defined above. If    $g\in\Gamma_{\eta}$, then there exists a unique  $\wt{g}\in\Gamma_{\xi}$ which fixes $v$ and restricts to $g$ on $v^{\bot}$.
\end{proposition}
\begin{proof}
Unicity of $\wt{g}$ is obvious, we must prove existence. Let $\wh{g}\in O(\Lambda_{\QQ})$ be the extension of $g$ which maps $v$ to itself. 
There exists $u\in v^{\bot}$ such that $\wh{g}(w/2)=w/2+u$  because  $g\in\Gamma_{\eta}$ (i.e.~$g[w/2]=[w/2]$). Hence
$\wh{g}((v+w)/2)=(v+w)/2+u\in\Lambda$, and this proves that  $\wh{g}(\Lambda)\subset\Lambda$.  We set $\wt{g}:=\wh{g}|_{\Lambda}$. Then $\wt{g}\in O(\Lambda)$, $\wt{g}(v)=v$, and $\wt{g}|_{v^{\bot}}=g$. Since $\xi=[v/2]$ the isometry $\wt{g}$ fixes $\xi$, and moreover $\wt{g}\in O^{+}(\Lambda)$ because $g\in O^{+}(v^{\bot})$; this proves that $\wt{g}\in\Gamma_{\xi}$. 
\end{proof}
By~\Ref{prp}{decest} we have an injection of groups
\begin{equation}\label{injgrp}
\begin{matrix}
\Gamma_{\eta} & \hra & \Gamma_{\xi} \\
g & \mapsto & \wt{g}
\end{matrix}
\end{equation}
\begin{claim}\label{clm:injgrp}
The injective homomorphism of~\eqref{injgrp} has image equal to the stabilizer $\Stab(v)< \Gamma_{\xi}$ of $v$.
\end{claim}
\begin{proof}
By construction the image is contained in $\Stab(v)$. Now suppose that  $h\in\Gamma_{\xi}$, and $h(v)=v$. Let $g\in O(v^{\bot})$ be the isometry obtained by restricting $h$ to $v^{\bot}$. Then $g\in O^{+}(v^{\bot})$ because  $h\in O^{+}(\Lambda)$. It remains to prove that $g(\eta)=\eta$, where $\eta\in A_{v^{\bot}}$ is as above. Let $w\in v$ be as above; thus $\eta=[w/2]$. Then $(v+w)/2$ belongs to $\Lambda$, and hence $h((v+w)/2)\in\Lambda$. Thus
$$\Lambda\ni h((v+w)/2)-(v+w)/2=h(w/2)-w/2.$$
Since $h(w)\in v^{\bot}$, the above equation shows that $g([w/2])=[w/2]$.
\end{proof}
Let $(\Lambda,\xi)$ be a decorated $D$ lattice, and $v\in\Lambda$ be a  hyperelliptic vector.  
Let $\Lambda':=v^{\bot}$ (a $D$ lattice), and $\xi'$ be the associated decoration of $\Lambda'$. We have defined an  injection $\Gamma_{\xi'}\hra\Gamma_{\xi}$, see~\eqref{injgrp}, and hence there is a well-defined regular map of quasi-projective varieties.
 \begin{equation}\label{mappaeffe}
 \begin{matrix}
\cF(\Lambda',\xi') & \overset{f}{\lra} & H_h(\Lambda,\xi) \\
\Gamma_{\xi'}[\sigma] & \mapsto & \Gamma_{\xi}[\sigma]
\end{matrix}
\end{equation}
Below is the main result of the present subsection.
\begin{proposition}\label{prp:hyplocsymm}
The map $f$ of~\eqref{mappaeffe} is an isomorphim onto the hyperelliptic Heegner divisor $H_h(\Lambda,\xi)$. Moreover the intersection of $H_h(\Lambda,\xi)$ and the singular locus of $\cF(\Lambda,\xi)$ has codimension at least two in  $H_h(\Lambda,\xi)$. 
\end{proposition}
We will prove~\Ref{prp}{hyplocsymm} at the end of the present subsubsection. First we will go through a series  of preliminary results.
\begin{proposition}\label{prp:liftrefl}
Let $(\Lambda,\xi)$ be a dimension-$N$ decorated $D$ lattice.  Let $v\in\Lambda$ be a hyperelliptic vector, and  let $\eta$ be the decoration of the dimension-$(N-1)$  $D$ lattice $v^{\bot}$ defined above. Suppose that $w$ is a reflective vector of $(v^{\bot},\eta)$. Then $w$ is a reflective vector of  $(\Lambda,\xi)$ as well. More precisely,
\begin{enumerate}
\item
If $w$ is a nodal  vector of $(v^{\bot},\eta)$, then it is a  nodal  vector of $(\Lambda,\xi)$.
\item
If $w$ is a hyperelliptic   vector of $(v^{\bot},\eta)$, then it is a   hyperelliptic  vector of $(\Lambda,\xi)$.
\item
If $N\equiv 4\pmod{8}$ and  $w$ is a unigonal (reflective)   vector of $(v^{\bot},\eta)$, then  it is a   hyperelliptic  vector of $(\Lambda,\xi)$.
\item
If $N\equiv 5\pmod{8}$ and  $w$ is a unigonal (reflective)   vector of $(v^{\bot},\eta)$, then  it is a nodal vector of $(\Lambda,\xi)$.
\end{enumerate}
\end{proposition}
\begin{proof}
(1): trivial. (2): We may assume that $\Lambda=U^2\oplus D_{N-2}$ and 
\begin{equation*}
v=(0_{4},(0,\ldots,0,2)).
\end{equation*}
 By~\Ref{prp}{decest} and~\Ref{prp}{minorm}, we may assume that $w=(0_{4},(0,\ldots,0,2,0))$, and~(2) follows. (3): Let $N=4+8k$, where $k\ge 0$. We may assume that $\Lambda=II_{2,2+8k}\oplus D_2$, and $v=(0_{4+8k},(0,2))$. By~\Ref{prp}{decest} and~\Ref{prp}{minorm}, we may assume that $w=(0_{4+8k},(2,0))$, and~(3) follows. (4): Let $N=5+8k$, where $k\ge 0$. We may assume that $\Lambda=II_{2,2+8k}\oplus D_3$, and $v=(0_{4+8k},(0,0,2))$. By~\Ref{prp}{decest}  and~\Ref{prp}{minorm}, we may assume that $w=(0_{4+8k},(1,-1,0))$, and~(4) follows.

\end{proof}

Let $\Lambda$ be a lattice (any latttice, not necessarily a $D$ lattice), $\Omega\subset\Lambda$ a subgroup. We let $\ov{\Omega}\subset\Lambda$ be the  \emph{saturation} of $\Omega$, i.e.~the subgroup  of vectors $v\in\Lambda$ such that $mv\in\Omega$ for some $0\not=m\in\ZZ$.
\begin{lemma}\label{lmm:manyhyp}
Let $(\Lambda,\xi)$ be a decorated $D$ lattice, and  $v_1,\ldots,v_k\in\Lambda$. Suppose that the following hold:
\begin{enumerate}
\item
$v_i$ is a  hyperelliptic vector for  $i\in\{1,\ldots,k\}$. 
\item
If $i\not=j\in\{1,\ldots,k\}$, then $v_i\not=\pm v_j$.
\item
The sublattice $\Omega\subset\Lambda$  generated  by $v_1,\ldots,v_k$ is negative definite. 
\end{enumerate}
Then  the following hold:
\begin{enumerate}
\item[(I)]
The formula
\begin{equation}\label{dividoperdue}
F(x_1,\ldots,x_k)=\frac{1}{2}\sum_{i=1}^k x_i v_i.
\end{equation}
defines an isomorphism of lattices $F\colon  D_k\overset{\sim}{\lra}\ov{\Omega}$. (In particular  $v_1,\ldots,v_k$ are pairwise orthogonal.)
\item[(II)]
A vector  $v\in\ov{\Omega}$ is  hyperelliptic if and only if $v=\pm v_i$ for some $i\in\{1,\ldots,k\}$.
\end{enumerate}
\end{lemma}
\begin{proof}
Let $u,w\in\ov{\Omega}$ be  hyperelliptic vectors, and suppose that $u\not=\pm w$. We will prove that $(u,w)=0$. In fact

 $$-4=u^2=(w+u-w,u)=(w,u)+2\left(\frac{u-w}{2},u\right),$$
 and since $(u-w)/2\in\Lambda$ and $\divisore(u)=2$, it follows that
  $(u,w)\in 4\ZZ$.  The vectors $u,w$ span a rank-$2$ negative definite sublattice of $\Lambda$ (the lattice has rank  $2$ because  $u\not=\pm w$, and it is negative definite by Item~(3)), and hence the determinant of the Gram matrix associated to the basis $\{u,w\}$ is strictly positive. Since $-4=u^2=w^2$ and $(u,w)\in 4\ZZ$, it follows that $(u,w)=0$. 
In particular  $(v_i,v_j)=0$  for  $i,j\in\{1,\ldots,k\}$, and hence $\Omega$ has rank $k$. The vector  $(\sum_{i=1}^k x_i v_i)/2$ belongs to $\ov{\Omega}$ whenever $\sum_{i=1}^k x_i$ is even, because $(v_i+v_j)/2\in\ov{\Omega}$.  It   follows that~\eqref{dividoperdue} defines an isomorphism between $D_k$ and the sublattice  $F(D_k)\subset\ov{\Omega}$. Suppose that $F(D_k)\not=\ov{\Omega}$; since $\det(D_k)=(-1)^k 4$, it follows that $\ov{\Omega}$ is unimodular, and that contradicts the hypothesis that each $v_i$ has divisibility $2$. Thus  $F(D_k)=\ov{\Omega}$, and this proves Item~(I).
 Let us prove  Item~(II).
Let $v\in\ov{\Omega}$   be a  hyperelliptic vector, and suppose that  $v\not=\pm v_i$ for all $i\in\{1,\ldots,k\}$. As proved above, it follows that $(v,v_i)=0$ for  $i\in\{1,\ldots,k\}$; that is a contradiction because the negative definite $\QQ$-vector space  $\Omega_{\QQ}$ is  generated (over $\QQ$)   by $v_1,\ldots,v_k$.
\end{proof}
\begin{proposition}\label{prp:manyhyp}
Let $(\Lambda,\xi)$ be a decorated $D$ lattice.
The hyperelliptic  Heegner divisor $H_h(\Lambda,\xi)$  is normal. 
\end{proposition}
\begin{proof}
By definition $H_h(\Lambda,\xi)=\Gamma_{\xi}\backslash \cH_h(\Lambda,\xi)$. Let $p\in H_h(\Lambda,\xi)$, and $[\sigma]\in\cH_h(\Lambda,\xi)$ be a representative of $p$. 
 Let  $\Stab([\sigma])<\Gamma_{\xi}$ be the stabilizer of the line $[\sigma]$. By construction we have the following isomorphisms of analytic germs:
 \begin{eqnarray}
(\cF(\Lambda,\xi),p) & \cong & (\Stab([\sigma])\backslash \cD^{+}_{\Lambda},\ov{[\sigma]}), \\
(H_h(\Lambda,\xi),p) & \cong & (\Stab([\sigma])\backslash \cH_h(\Lambda,\xi),\ov{[\sigma]}).
\end{eqnarray}
 Since $\sigma^{\bot}\cap\Lambda_{\RR}$ is negative definite, the set of  hyperelliptic vectors $v\in \sigma^{\bot}\cap\Lambda$  is finite (and non-empty). Let $\{v_1,\ldots,v_k\}$ be a maximal collection of such vectors with the property that $v_i\not=\pm v_j$ for $i\not=j$. Let  $\Omega_{\sigma}\subset\Lambda$ be the subgroup generated by $v_1,\ldots,v_k$. As noticed above,  the restriction of the quadratic form to  $\Omega_{\sigma}$ is negative definite. Thus the hypotheses of~\Ref{lmm}{manyhyp} are satisfied, and hence the saturation $\ov{\Omega}_{\sigma}$ is isomorphic (as lattice) to $D_k$. (Notice that by~\Ref{lmm}{manyhyp}, $\Omega_{\sigma}$ is independent of the choice of  a maximal collection as above.)
 Let $R'_{\sigma}$ be the set of $r\in\ov{\Omega}_{\sigma}$ such that $r^2=-2$, and let  $R''_{\sigma}$ be the set of  hyperelliptic vectors of $\ov{\Omega}_{\sigma}$.  
 Then
\begin{equation*}
\scriptstyle
R'_{\sigma}=\{((-1)^{m_i} v_i+(-1)^{m_j} v_j)/2\  \mid \  1\le i<j\le k\},\quad R''_{\sigma}=\{(-1)^{m_i}v_i \ \mid \ 1\le i\le k\}.
\end{equation*}
(The last equality holds by~\Ref{lmm}{manyhyp}.)  Let  $R_{\sigma}:=R'_{\sigma}\cup R''_{\sigma}$. Notice that  $\rho_r\in\Gamma_{\xi}$ for all $r\in R_{\sigma}$. 
    Let $W_{\sigma},W'_{\sigma},W''_{\sigma}<\Stab([\sigma])$ be the subgroups generated by the reflections $\rho_r$ for  $r\in R_{\sigma}$, $r\in R'_{\sigma}$, and 
    $r\in R''_{\sigma}$  respectively. Clearly 
 $W_{\sigma}$ is a normal subgroup of $\Stab([\sigma])$; let $G_{\sigma}:=\Stab([\sigma])/W_{\sigma}$. Then $G_{\sigma}$ acts on $W_{\sigma}\backslash\cD^{+}_{\Lambda}$, and on 
 $W_{\sigma}\backslash\cH_h(\Lambda,\xi)$. The analytic germ $(\Stab([\sigma])\backslash \cH_h(\Lambda,\xi),\ov{[\sigma]})$ is isomorphic to the germ at $\ov{[\sigma]}$ of   $W_{\sigma}\backslash\cH_h(\Lambda,\xi)$ modulo the action of $G_{\sigma}$.  Hence  it suffices to prove that  $W_{\sigma}\backslash\cH_h(\Lambda,\xi)$ is smooth   in a neighborhhod of $\ov{[\sigma]}$. 
 
We identify each $r\in R_{\sigma}$ with $F^{-1}(r)\in\CC^k$, where $F$ is the isometry of~\eqref{dividoperdue}, and we denote it by $r$. Thus $r\in  R'_{\sigma}$   $r\in R''_{\sigma}$  are of the form 
$$(0,\ldots,0,\pm 1,0,\ldots,0,\pm1,0,\ldots,0),\qquad (0,\ldots,0,\pm 2,0,\ldots,0)$$
respectively. The action of $W_{\sigma}$ is trivial on $\Omega_{\sigma}^{\bot}$.  It follows that there exist local analytic coordinates $({\bf x},{\bf t})=((x_1,\ldots,x_k),{\bf t})$ on $\cD^{+}_{\Lambda}$, centered at $[\sigma]$,  such that   $x_i=0$ is a local equation of $v_i^{\bot}\cap\cD^{+}_{\Lambda}$, and
 \begin{equation*}
\rho_r({\bf x},{\bf t})=({\bf x}-\frac{2({\bf x},r)}{r^2} r,{\bf t}),
\end{equation*}
where $({\bf x},r)$ is the opposite of the standard euclidean product of ${\bf x}$ and $r$.
 
  In order to describe  $W_{\sigma}\backslash\cD^{+}_{\Lambda}$ and  $W_{\sigma}\backslash\cH_h(\Lambda,\xi)$, we first take the quotient by the normal subgroup $W''_{\sigma}$, and then we act by $W'_{\sigma}/W''_{\sigma}$ on the quotient.  Local analytic coordinates on  $W''_{\sigma}\backslash\cD^{+}_{\Lambda}$ are 
  $(y_1,\ldots,y_k,{\bf t})$, where $y_i=x_i^2$. The action of $W'_{\sigma}/W''_{\sigma}$ on  $(y_1,\ldots,y_k)$ is the standard representation of the symmetric group $\cS_k$. Thus local analytic coordinates on  $W_{\sigma}\backslash\cD^{+}_{\Lambda}$ are $(\tau_1,\ldots,\tau_k,{\bf t})$, where $\tau_i$ is the degree-$i$ elementary symmetric function in $y_1,\ldots,y_k$. In particular  $W_{\sigma}\backslash\cD^{+}_{\Lambda}$ is smooth in a neighborhhod of $\ov{[\sigma]}$. Since a local equation of $\cH_h(\Lambda,\xi)$ is given by $x_1\cdot\ldots\cdot x_k=0$, a local equation of  $W_{\sigma}\backslash\cH_h(\Lambda,\xi)$ in  $W_{\sigma}\backslash\cD^{+}_{\Lambda}$  is $\tau_k=0$. We have proved that  $W_{\sigma}\backslash\cH_h(\Lambda,\xi)$ is smooth  in a neighborhhod of $\ov{[\sigma]}$, as claimed. 
\end{proof}
{\it Proof of~\Ref{prp}{hyplocsymm}.} 
 We adopt the notation introduced in the proof of~\Ref{prp}{manyhyp}. We start by noting that $f$ is surjective by definition. The composition of $f$ and the inclusion $H_h(\Lambda,\xi)\hra 
 \cF^{*}(\Lambda,\xi)$ 
  extends to a regular map
 \begin{equation*}
 \varphi\colon \cF^{*}(\Lambda',\xi')  \lra \cF^{*}(\Lambda,\xi),
\end{equation*}
 compatible with the Baily-Borel boundaries. Since the Baily-Borel compactifications are projective it follows that $f$ is a projective map. By~\Ref{clm}{injgrp}, the fiber of $f$ at the equivalence class represented by $\sigma\in\cH_h(\Lambda,\xi)$ is 
 identified with the set of  hyperelliptic vectors $v\in\Omega_{\sigma}$, modulo the action of $\Stab([\sigma])$ (notice that $-\Id_{\Lambda}\in \Stab([\sigma])$). Hence the proof of~\Ref{prp}{manyhyp} shows that the fiber of $f$ is a singleton, in particular $f$ is birational. Now $H_h(\Lambda,\xi)$ is 
 normal by~\Ref{prp}{manyhyp}; since $f$ is birational and projective, it follows that it is an isomorphism. 
 
It remains to prove that  the intersection of $H_h(\Lambda,\xi)$ and the singular locus of $\cF(\Lambda,\xi)$ has codimension at least $2$ in  $H_h(\Lambda,\xi)$. Suppose the contrary: we will reach a contradiction. Since  $\cF(\Lambda,\xi)$ is normal, there is an irreducible component $Z$ of $H_h(\Lambda,\xi)\cap\sing\cF(\Lambda,\xi)$ which has codimension $1$ in 
$H_h(\Lambda,\xi)$. Let $v\in\Lambda$ be a hyperelliptic vector, and $\eta$ the decoration of the $D$ lattice $v^{\bot}$ associated to $\xi$. Let $\pi\colon \cD_{v^{\bot}}^{+}\lra H_h(\Lambda,\xi)$ be the natural map; we have proved that $\pi$ is the quotient map for the action of $\Gamma_{\eta}$ on $\cD^{+}_{v^{\bot}}$. 
Now let $\wt{Z}\subset \cD^{+}_{v^{\bot}}$ be an irreducible component of $\pi^{-1}Z$. If $[\sigma]\in\wt{Z}$, then 
\begin{equation*}
\Stab([\sigma])\supsetneq \la\rho_v,-1_{\Lambda}\ra
\end{equation*}
because $\pi([\sigma])$ is a singular point of $\cF(\Lambda,\xi)$. We distinguish between the two cases:
\begin{enumerate}
\item
For very general $[\sigma]\in\wt{Z}$, the set of hyperelliptic vectors in $\sigma^{\bot}$ is $\{v,-v\}$.
\item
For very general $[\sigma]\in\wt{Z}$, the set of hyperelliptic vectors in $\sigma^{\bot}$ has cardinality strictly greater than two.
\end{enumerate}
Assume that (1) holds, and let  $[\sigma]\in\wt{Z}$ be  very general. Let $g\in (\Stab([\sigma])\setminus \la\rho_v,-1_{\Lambda}\ra)$. 
Then $g(v)=\pm v$, and hence multiplying $g$ by $\rho_v$ if necessary, we may assume that $g(v)=v$, i.e.~$g|_{v^{\bot}}\in\Gamma_{\eta}$. 
Now $g|_{v^{\bot}}$ fixes every point of $\wt{Z}$, which has codimension $1$ in $\cD^{+}_{v^{\bot}}$. It follows that  $g|_{v^{\bot}}=\pm \rho_{w}^{v^{\bot}}$, where $w$ is a reflective vector  of $(v^{\bot},\eta)$, and $\rho_{w}^{v^{\bot}}$ denotes the associated reflection of $v^{\bot}$. By~\Ref{prp}{liftrefl} the vector $w$ is a reflective vector  of $(\Lambda,\xi)$ as well; it follows that $g=\rho_w$. Thus $\wt{Z}=v^{\bot}\cap w^{\bot}\cap\cD_{\Lambda}^{+}$, and hence $\Stab([\sigma])=\la\rho_v,\rho_w,-1_{\Lambda}\ra$ for a very general $[\sigma]\in\wt{Z}$. It follows that if  $[\sigma]\in\wt{Z}$ is very general, then $\cF(\Lambda,\xi)$ is smooth at $\pi([\sigma])$, and that contradicts our assumption.

Lastly, assume that~(2) holds. Let $[\sigma]\in\wt{Z}$ be very general. Since $\wt{Z}$ has codimension two in $\cD_{\Lambda}^{+}$, the set of  hyperelliptic vectors in $\sigma^{\bot}$
 is equal to $\{\pm v_1,\pm v_2\}$, where $v_1,v_2$ are orthogonal hyperelliptic vectors, and moreover  $\Stab([\sigma])=\la\rho_{v_1},\rho_{v_2},-1_{\Lambda}\ra$. As shown in the proof of~\Ref{prp}{manyhyp}, it follows  that  $\cF(\Lambda,\xi)$ is smooth at $\pi([\sigma])$, and that contradicts our assumption.
 \qed
\subsection{Reflective unigonal divisors}
\setcounter{equation}{0}
\subsubsection{$N\equiv 3\pmod{8}$}\label{subsubsec:unigon3}
Let $(\Lambda,\xi)$ be a decorated $D$ lattice of dimension $N\equiv 3\pmod{8}$. Let $N=8k+3$, where $k\ge 0$. Let $v\in\Lambda$ be a unigonal vector, i.e.~$v^2=-4$ and $\divisore(v)=4$. Then $\Lambda=\la v\ra\oplus v^{\bot}$. Since $\det\Lambda=-4$, it follows that $v^{\bot}$ is unimodular. Thus
\begin{equation}\label{perpunig}
v^{\bot}\cong II_{2,2+8k}.
\end{equation}
Given $g\in O^{+}( II_{2,2+8k})$, let $\wt{g}\in O(\Lambda)$ be the isometry such that 
\begin{equation}
\wt{g}(v)=v,\qquad \wt{g}|_{v^{\bot}}=g.
\end{equation}
Then $\wt{g}\in O^{+}(\Lambda)$ because $g\in O^{+}( II_{2,2+8k})$, and $\wt{g}\in \wt{O}(\Lambda)$ because $A_{\Lambda}$ is generated by $v^{*}$. Thus we have an injection of groups
\begin{equation}
\begin{matrix}
 O^{+}( II_{2,2+8k}) & \hra & \Gamma_{\xi} \\
 g & \mapsto & \wt{g}
\end{matrix}
\end{equation}
It follows that the injection $\cD^{+}_{v^{\bot}}\hra \cD^{+}_{\Lambda}$ descends to a regular map
\begin{equation}\label{mappagi}
\cF_{II_{2,2+8k}}(O^{+}( II_{2,2+8k}))\overset{l}{\lra} H_u(\Lambda,\xi). 
\end{equation}
\begin{proposition}\label{prp:unigcong3}
Let $(\Lambda,\xi)$ be a decorated $D$ lattice of dimension $N\equiv 3\pmod{8}$, and  let $N=8k+3$, where $k\ge 0$. The map $l$ in~\eqref{mappagi} is an isomorphim onto the unigonal Heegner divisor $H_u(\Lambda,\xi)$. Moreover the intersection of $H_u(\Lambda,\xi)$ and the singular locus of $\cF(\Lambda,\xi)$ has codimension at least two in  $H_u(\Lambda,\xi)$. 
\end{proposition}
\begin{proof}
The map $l$ is finite (see the proof of~\Ref{prp}{hyplocsymm}), and it has degree $1$ because $-\Id_{\Lambda}\in\Gamma_{\xi}$. Thus, in order to prove that $l$  is an isomorphim,  it suffices to prove that $H_u(\Lambda,\xi)$ is normal. We claim that the pre-Heegner divisor  $\cH_u(\Lambda,\xi)$ is smooth, i.e.~that if $v_1, v_2$ are non-proportional unigonal vectors, then $v_1^{\bot}\cap v_2^{\bot}\cap\cD^{+}_{\Lambda}=\es$. In fact, suppose the contrary. Then $v_1,v_2$ span a negative-definite rank-$2$ sublattice of $\Lambda$; since $v_i^2=-4$ and $(v_1,v_2)\in 4\ZZ$, it follows that $v_1\bot v_2$. On the other hand $v_1^{\bot}\cong II_{2,2+8k}$ by~\eqref{perpunig}, and hence $v_1^{\bot}$ does not contain a primitive vector of divisibility greater than $1$. This proves that   $\cH_u(\Lambda,\xi)$ is smooth, and hence $H_u(\Lambda,\xi)=\Gamma_{\xi}\backslash \cH_u(\Lambda,\xi)$ is normal. 

The proof that  $H_u(\Lambda,\xi)\cap\sing\cF(\Lambda,\xi)$ has codimension at least two in  $H_u(\Lambda,\xi)$ is similar to the analogous statement for $H_h(\Lambda,\xi)$, see~\Ref{prp}{hyplocsymm}. We leave details to the reader. 
\end{proof}
In order to simplify notation, from now on we let
\begin{equation}
\cF( II_{2,2+8k}):=\cF_{II_{2,2+8k}}(O^{+}( II_{2,2+8k}))). 
\end{equation}
A  vector $v\in II_{2,2+8k}$ is \emph{nodal} if it has square $-2$.  By Eichler's Criterion, i.e.~\Ref{prp}{criteich}, any two nodal vectors of $II_{2,2+8k}$  are  $O^{+}( II_{2,2+8k})$-equivalent. We let $H_n( II_{2,2+8k})$ be the Heegner divisor of $\cF( II_{2,2+8k})$ corresponding to a nodal $v\in II_{2,2+8k}$.
\subsubsection{$N\equiv 4\pmod{8}$}\label{subsubsec:unigon4}
Let $(\Lambda,\xi)$ be a decorated $D$ lattice of dimension $N\equiv 4\pmod{8}$. Let $N=8k+4$, where $k\ge 0$. Let $v\in\Lambda$ be a unigonal vector, i.e.~$v^2=-2$ and $\divisore(v)=2$. Then $\Lambda=\la v\ra\oplus v^{\bot}$. Moreover 
\begin{equation}\label{traffico}
v^{\bot}\cong   II_{2,2+8k}\oplus A_1.
\end{equation}
In fact $A_{\Lambda}=\ZZ/(2)\oplus A_{v^{\bot}}$, where  a generator of the first summand has square $-1/2$ modulo $2\ZZ$.  It follows that the discriminat group of $v^{\bot}$ is $\ZZ/(2)$, and a generator has square $-1/2$ modulo $2\ZZ$. Thus  $v^{\bot}$ and  $U^2\oplus E_8^k\oplus A_1$ have  the same signature and isomorphic discriminant groups; thus~\eqref{traffico} follows from  Theorem 1.13.2 of~\cite{nikulin}. 
Given $g\in O^{+}(v^{\bot})$, let $\wt{g}\in O(\Lambda)$ be the isometry such that 
\begin{equation}
\wt{g}(v)=v,\qquad \wt{g}|_{v^{\bot}}=g.
\end{equation}
Then $\wt{g}\in O^{+}(\Lambda)$ because $g\in O^{+}(v^{\bot})$, and $\wt{g}\in \wt{O}(\Lambda)$ because $O(v^{\bot})=\wt{O}(v^{\bot})$, and  $A_{\Lambda}=A_{v^{\bot}}\oplus \la v^{*}\ra$. Thus we have an injection of groups
\begin{equation}
\begin{matrix}
 O^{+}( II_{2,2+8k}\oplus A_1) & \hra & \Gamma_{\xi} \\
 g & \mapsto & \wt{g}
\end{matrix}
\end{equation}
It follows that the injection $\cD^{+}_{v^{\bot}}\hra \cD^{+}_{\Lambda}$ descends to a regular map
\begin{equation}\label{mappaelle}
\cF_{ II_{2,2+8k}\oplus A_1}(O^{+}( II_{2,2+8k}\oplus A_1)\overset{m}{\lra} H_u(\Lambda,\xi). 
\end{equation}
\begin{proposition}\label{prp:unigcong4}
Let $(\Lambda,\xi)$ be a decorated $D$ lattice of dimension $N\equiv 4\pmod{8}$, and  let $N=8k+4$, where $k\ge 0$. 
 The map $m$ in~\eqref{mappaelle} is an isomorphim onto the unigonal Heegner divisor $H_u(\Lambda,\xi)$. Moreover the intersection of $H_u(\Lambda,\xi)$ and the singular locus of $\cF(\Lambda,\xi)$ has codimension at least two in  $H_u(\Lambda,\xi)$. 
\end{proposition}
\begin{proof}
Arguing as in the proof of~\Ref{prp}{unigcong3}, we conclude that, in order to prove that $m$  is an isomorphim,   it suffices to show that   $H_u(\Lambda,\xi)$ is normal. Let $[\sigma]\in\cH_u(\Lambda,\xi)$; we will prove that $H_u(\Lambda,\xi)$ is normal at the point representing the $\Gamma_{\xi}$-orbit of $[\sigma]$. Let 
$v_1,\ldots,v_n\in\sigma^{\bot}$ be a set of pairwise non-proportional unigonal vectors;  we claim that $n\le 2$. In fact, let $i\not= j\in\{1,\ldots,n\}$. Then $\la v_i,v_j\ra$ is a rank-$2$ negative-definite sublattice of $\Lambda$; since $v_i^2=v_j^2=-2$ and $(v_i,v_j)\in 2\ZZ$, it follows that $v_i\bot v_j$. Thus distinct vectors $v_i$ are orthogonal. Now suppose that $n\ge 3$. It follows from~\eqref{traffico} that $\la v_1,v_2\ra^{\bot}\cong U^2\oplus E_8^k$, and hence $\la v_1,v_2\ra$ can not contain a primitive vector of divisibility greater than $1$. This proves that $n\le 2$. It follows that either 
$\cH_u(\Lambda,\xi)$ is smooth at $[\sigma]$, i.e., up to $\pm$, there is a unique  unigonal vector perpendicular to $\sigma$, or $\cH_u(\Lambda,\xi)$  in a neighborhood of $[\sigma]$ is a normal crossings divisor with two  components, i.e., up to $\pm$, there are two vectors  perpendicular to $\sigma$. 
 Suppose that the former holds; then  $H_u(\Lambda,\xi)$ has a quotient singularity  at the point representing  $\Gamma_{\xi}[\sigma]$, and hence is normal at that point.   Suppose that the latter holds, and let $v_1,v_2\in\sigma^{\bot}$ be a maximal set of  non-proportional unigonal vectors, i.e.~such that in a neighborhood of $[\sigma]$, the pre-Heegner divisor $\cH_u(\Lambda,\xi)$ equals $(v_1^{\bot}\cap\cD^{+}_{\Lambda})+ (v_2^{\bot}\cap\cD^{+}_{\Lambda})$.  The vectors $\pm v_1\pm v_2$ are hyperelliptic, and any hyperelliptic vector in $\ZZ v_1+\ZZ v_2$ is of this kind.  
 Let  $\Stab([\sigma])<\Gamma_{\xi}$ be  the stabilizer of $[\sigma]$, and 
 $G\vartriangleleft\Stab([\sigma])$ be the normal subgroup generated by the reflections $\rho_v$ for  $v$  a hyperelliptic vector  of $\ZZ v_1+\ZZ v_2$.  The quotient  $G\backslash\cH_u(\Lambda,\xi)$ is smooth at the orbit of $[\sigma]$ because $\rho_{v_1\pm v_2}(v_i)=\mp v_{3-i}$ for $i=1,2$. 
 The stabilizer $\Stab([\sigma])$ acts on   $G\backslash\cH_u(\Lambda,\xi)$, and the analytic germ of  $H_u(\Lambda,\xi)$ a  $\Gamma_{\xi}[\sigma]$ is isomorphic to the germ  of  $G\backslash\cH_u(\Lambda,\xi)$ modulo  $\Stab([\sigma])/G$  at the orbit of $G[\sigma]$; thus  $H_u(\Lambda,\xi)$ has a quotient singularity at $\Gamma_{\xi}[\sigma]$, and hence is normal at that point.
 
The proof that  $H_u(\Lambda,\xi)\cap\sing\cF(\Lambda,\xi)$ has codimension at least two in  $H_u(\Lambda,\xi)$ is similar to the analogous statement for $H_h(\Lambda,\xi)$, see~\Ref{prp}{hyplocsymm}. We leave details to the reader. 
\end{proof}
In order to simplify notation, from now on we let
\begin{equation}
\cF( II_{2,2+8k}\oplus A_1):=\cF_{II_{2,2+8k}\oplus A_1}(O^{+}( II_{2,2+8k}\oplus A_1))). 
\end{equation}
A  vector $v\in II_{2,2+8k}\oplus A_1$ is \emph{nodal} if it has square $-2$ and divisibility $1$, it is \emph{unigonal}  if it has square $-2$ and divisibility $2$.  By Eichler's Criterion, i.e.~\Ref{prp}{criteich}, any two nodal vectors of $II_{2,2+8k}\oplus A_1$  are  $O^{+}( II_{2,2+8k}\oplus  A_1)$-equivalent, and similarly for any two unigonal vectors. 
We let $H_n( II_{2,2+8k}\oplus A_1)$ and  $H_u( II_{2,2+8k}\oplus A_1)$ be the Heegner divisor of $\cF( II_{2,2+8k}\oplus A_1)$ corresponding to a nodal or unigonal $v\in II_{2,2+8k}\oplus A_1$  respectively.

One describes  the Heegner divisor  $H_u( II_{2,2+8k}\oplus A_1)$ proceeding as in~\Ref{subsubsec}{unigon3}. In fact, let $v\in(II_{2,2+8k}\oplus A_1)$ be a unigonal vector. Then 
$( II_{2,2+8k}\oplus A_1)=\la v\ra\oplus v^{\bot}$, and $v^{\bot}\cong II_{2,2+8k}$ is unimodular. Given $g\in O^{+}( II_{2,2+8k})$, let $\wt{g}\in O(II_{2,2+8k}\oplus A_1)$ be the isometry which fixes $v$, and restricts to $g$ on $v^{\bot}$. 
Then $\wt{g}\in O^{+}(II_{2,2+8k}\oplus A_1)$, and we have an injection of groups
\begin{equation}
\begin{matrix}
 O^{+}( II_{2,2+8k}) & \hra & O^{+}( II_{2,2+8k}\oplus A_1)\\
 g & \mapsto & \wt{g}
\end{matrix}
\end{equation}
It follows that the injection $\cD^{+}_{v^{\bot}}\hra \cD^{+}_{\Lambda}$ descends to a regular map
\begin{equation}\label{mappapi}
\cF( II_{2,2+8k})\overset{p}{\lra} H_u(II_{2,2+8k}\oplus A_1). 
\end{equation}
The proof of the  result below is similar to the proof of~\Ref{prp}{unigcong3}; we leave details to the reader.
\begin{proposition}\label{prp:unigunig}
Let $k\ge 0$. The map $p$ of~\eqref{mappapi} is an isomorphim onto the unigonal Heegner divisor $H_u(II_{2,2+8k}\oplus A_1)$. Moreover the intersection of 
$H_u(II_{2,2+8k}\oplus A_1)$ and the singular locus of $\cF(II_{2,2+8k}\oplus A_1)$ has codimension at least two in  $H_u(II_{2,2+8k}\oplus A_1)$. 
\end{proposition}
\subsection{The (non-reflective) unigonal divisors for $N\equiv 5\pmod{8}$}
\setcounter{equation}{0}
Let $(\Lambda,\xi)$ be a decorated $D$ lattice of dimension $N\equiv 5\pmod{8}$. Let $N=8k+5$, where $k\ge 0$; thus $\Lambda\cong \II_{2,2+8k}\oplus D_3$. Let $v\in\Lambda$ be a unigonal vector, i.e.~$v^2=-12$ and $\divisore(v)=4$. Then  
\begin{equation}\label{schwarzwald}
v^{\bot}\cong   II_{2,2+8k}\oplus A_2.
\end{equation}
In fact, the above isomorphism is clearly verified if $v=(0_{4+8k},(2,2,2))$, and 
since the set of unigonal vectors is a single $O(\Lambda)$-orbit, it holds for any unigonal $v$.  
Given $g\in \wt{O}^{+}(v^{\bot})$, let $\wt{g}\in O(\Lambda)$ be the isometry such that 
$\wt{g}(v)=v$, and $\wt{g}|_{v^{\bot}}=g$.
Then $\wt{g}\in O^{+}(\Lambda)$ because $g\in O^{+}(v^{\bot})$, and $\wt{g}\in \wt{O}(\Lambda)$ because  $g\in \wt{O}(v^{\bot})$. 
Thus we have an injection of groups
\begin{equation}
\begin{matrix}
 O^{+}( II_{2,2+8k}\oplus A_2) & \hra & \wt{O}^{+}(\Lambda) \\
 g & \mapsto & \wt{g}
\end{matrix}
\end{equation}
It follows that the injection $\cD^{+}_{v^{\bot}}\hra \cD^{+}_{\Lambda}$ descends to a regular map
\begin{equation}\label{mappaqu}
\cF_{ II_{2,2+8k}\oplus A_2}(\wt{O}^{+}( II_{2,2+8k}\oplus A_2))\overset{q}{\lra} H_u(\Lambda,\xi). 
\end{equation}
\begin{proposition}\label{prp:unigcong5}
Let $(\Lambda,\xi)$ be a decorated $D$ lattice of dimension $N\equiv 5\pmod{8}$, and  let $N=8k+5$, where $k\ge 0$.  The map $q$ in~\eqref{mappaqu} is an isomorphism onto the unigonal Heegner divisor $H_u(\Lambda,\xi)$. Moreover the intersection of $H_u(\Lambda,\xi)$ and the singular locus of $\cF(\Lambda,\xi)$ has codimension at least two in  $H_u(\Lambda,\xi)$. 
\end{proposition}
Before proving~\Ref{prp}{unigcong5}, we prove the following result.
\begin{lemma}\label{lmm:unigcong5}
Let $(\Lambda,\xi)$ be a decorated $D$ lattice of dimension $N\equiv 5\pmod{8}$, and  let $N=8k+5$, where $k\ge 0$. 
Let $v_1,\ldots,v_n\in\Lambda$ be unigonal vectors, and suppose that the span $\la v_1,\ldots,v_n\ra$ is negative definite. Then there exists a decomposition $\Lambda=L\oplus L^{\bot}$, such that 
\begin{equation}\label{elleperp}
L\cong D_3,\qquad L^{\bot}\cong \II_{2,2+8k},
\end{equation}
and $v_1,\ldots,v_n\in L$. 
\end{lemma}
\begin{proof}
Eliminating some of the $v_i$'s, if necessary, we may assume that $v_1,\ldots,v_n$ are linearly independent. We may also assume that $[v_i/4]=[v_j/4]$ for $i,j\in\{1,\ldots,n\}$, by 
 multiplying some of the $v_i$'s by $-1$. If $n=1$, the lemma holds because the set of unigonal vectors is a single $O(\Lambda)$-orbit (see the proof of~\eqref{schwarzwald}). 
 
 Next suppose that $n=2$. The rank-$2$ lattice  $\la v_1,v_2\ra$ is  negative definite, hence $|(v_1,v_2)|< 12$, and since $(v_1,v_2)\equiv 0\pmod{8}$, it follows that 
 $(v_1,v_2)\in\{0,\pm 4,\pm 8\}$. On the other hand  $v_1-v_2=4u$, where $u\in\Lambda$, and hence $(v_1,v_2)=4$. Thus $u^2=-2$. Let $w:=v_2+2u=(v_1+v_2)/2$. Then $w^2=-4$, and $\divisore(w)=2$, i.e.~$w$ is a hyperelliptic vector. Thus 
 \begin{equation}\label{alligator}
w^{\bot}\cong \Lambda_{N-1}\cong D_2\oplus \II_{2,2+8k}. 
\end{equation}
The equality $w:=v_2+2u$ shows that $\divisore_{w^{\bot}}(u)=2$, and hence we may assume that $u\in D_2$ (with respect to the decomposition in~\eqref{alligator}). Thus  $v_1$ and $v_2$ belong to  the saturation of the span of $w$ and the $D_2$ summand. As is easily checked, $\Sat\la w,D_2\ra$ 
 is isomorphic to $D_3$.    This finishes the proof for $n=2$.  

Now suppose that $n=3$. As proved above (the case $n=2$), $v_i+v_j=2w_h$ for any permutation $i,j,h$ of $1,2,3$, and each $w_h$ is a hyperelliptic vector. Since $(w_h,w_l)=0$ for $h\not=l$, the saturation  $L:=\Sat\la w_1,w_2,w_3\ra$ is isomorphic to $D_3$, by~\Ref{lmm}{manyhyp}. It remains to show that $L$ is a direct summand of $\Lambda$. 
Let $u$ be as in the case $n=2$, i.e.~such that $v_1-v_2=4u$. 
It suffices to exhibit a vector $u'\in \{w,u\}^{\bot}\cap L$ of square $-2$ and divisibility $2$ in $w^{\bot}$; in fact this will prove that $\la u,u'\ra$ is a direct summand of $w^{\bot}$, and hence   $L$ is a direct summand of $\Lambda$. 
Since for $i,j\in\{1,2,3\}$ we have $[v_i/4]=[v_j/4]$ in the group $A_{\Lambda}\cong\ZZ/(4)$, we have  $[v_1/4]+[v_2/4]+2[v_3/4]=0$, i.e.~$v_1+v_2+2v_3=4u'$, where $u'\in\Lambda$. As is easily checked, $u'$ has the required properties. 
This finishes the proof for $n=3$.  
 
Lastly, suppose that $n>3$. Then $L:=\Sat\la v_1,v_2,v_3\ra\cong D_3$, and $\Lambda=L\oplus L^{\bot}$. We may assume that 
\begin{equation}\label{listina}
v_1=(2,-2,-2),  \quad v_2=(-2,2,-2), \quad v_3=(-2,-2,2).
\end{equation}
  Since $(v_i,v_4)=4$, it follows that $v_4=((2,2,2),4u)$, where $u\in L^{\bot}$. Thus $u^2=0$, because $v_4^2=-12$, and that is a contradiction since $u\in\la v_1,v_2,v_3,v_4\ra$, and by hypothesis $\la v_1,v_2,v_3,v_4\ra$ is negative dedinite.
\end{proof}
\begin{proof}[Proof of~\Ref{prp}{unigcong5}]
We will be brief because the proof is analogous to proofs (of analogous results) that we have already given. In order to prove that $q$  is an isomorphim,   it suffices to show that   $H_u(\Lambda,\xi)$ is normal. Let $[\sigma]\in\cH_u(\Lambda,\xi)$. Let 
$\{u_1,\ldots,u_n\}$  be the set of unigonal vectors in $\sigma^{\bot}$ (the set is finite because $(,)$ is negative definite on $\sigma^{\bot}$). By~\Ref{lmm}{unigcong5}, the dimension 
$r:=\dim\Span\la v_1,\ldots,v_n\ra$ is at most $3$.  

If $r=1$, then  
$\cH_u(\Lambda,\xi)$ is smooth at $[\sigma]$, and hence   $H_u(\Lambda,\xi)$ has a quotient singularity  at the point representing  $\Gamma_{\xi}[\sigma]$, in particular it is normal at that point.  

If $r=2$, then there exist unigonal vectors $v_1,v_2$ such that $[v_1/4]=[v_2/4]$, and $\{u_1,\ldots,u_n\}=\{\pm v_1,\pm v_2\}$. We have $v_1-v_2=2u$ and $v_1+v_2=2w$, where $u^2=-2$, $(u,w)=0$, and $w$ is a hyperelliptic vector (see the proof of~\Ref{lmm}{unigcong5}).
Now, $\cH_u(\Lambda,\xi)$  in a neighborhood of $[\sigma]$ is a normal crossings divisor with two  components, namely $v_1^{\bot}\cap\cD^{+}_{\Lambda}$ and 
$v_1^{\bot}\cap\cD^{+}_{\Lambda}$, but these components are interchanged by the reflection in $u$ (or by that in $w$); it follows easily that  $H_u(\Lambda,\xi)$ has a quotient singularity  at the point   $\Gamma_{\xi}[\sigma]$. 

If $r=3$, then  by~\Ref{lmm}{unigcong5} we may write $\{u_1,\ldots,u_n\}=
\{\pm v_1,\pm v_2,\pm v_3,\pm(v_1+v_2+v_3)\}$, where $v_i,v_2,v_3$ are given by~\eqref{listina}. Let $L:=\Sat\la v_1,v_2,v_3\ra$; thus $L\cong D_3$ 
by~\Ref{lmm}{unigcong5}.  Now, $\cH_u(\Lambda,\xi)$  in a neighborhood of $[\sigma]$ has four   components, namely 
$v_i^{\bot}\cap\cD^{+}_{\Lambda}$, for $i\in\{1,2,3\}$, and $(v_1+v_2+v_3)^{\bot}\cap\cD^{+}_{\Lambda}$. The group  $O(L)$, which is a normal subgroup of $\Stab([\sigma])$, acts transitively on these  hyperplanes, and the quotient $O(L)\backslash\cH_u(\Lambda,\xi)$ is smooth at the point  $O(L)[\sigma]$.  It follows  that  $H_u(\Lambda,\xi)$ has a quotient singularity  at the point   $\Gamma_{\xi}[\sigma]$. 
 
The proof that  $H_u(\Lambda,\xi)\cap\sing\cF(\Lambda,\xi)$ has codimension at least two in  $H_u(\Lambda,\xi)$ is similar to the analogous statement for $H_h(\Lambda,\xi)$, see~\Ref{prp}{hyplocsymm}. 
\end{proof}
In order to simplify notation, from now on we let
\begin{equation}\label{atwosymm}
\cF( II_{2,2+8k}\oplus A_2):=\cF_{II_{2,2+8k}\oplus A_2}(\wt{O}^{+}( II_{2,2+8k}\oplus A_2)). 
\end{equation}
A  vector $v\in II_{2,2+8k}\oplus A_2$ is \emph{nodal} if it has square $-2$ and divisibility $1$, it is \emph{unigonal}  if it has square $-12$ and divisibility $3$.  Notice that if $v$ is unigonal, then $[v/3]$ is a generator of the discriminant group $ A_{II_{2,2+8k}\oplus A_2}\cong \ZZ/(3)$. 
By Eichler's Criterion, i.e.~\Ref{prp}{criteich}, the set of nodal vectors of $II_{2,2+8k}\oplus A_2$  is a single  $\wt{O}^{+}( II_{2,2+8k}\oplus A_2)$-orbit, and similarly the set of unigonal vectors \emph{up to $\pm 1$} is a single $\wt{O}^{+}( II_{2,2+8k}\oplus A_2)$-orbit. 
We let $H_n( II_{2,2+8k}\oplus A_2)$ and  $H_u( II_{2,2+8k}\oplus A_2)$ be the Heegner divisor of $\cF( II_{2,2+8k}\oplus A_2)$ corresponding to a nodal or a unigonal $v\in II_{2,2+8k}\oplus A_2$  respectively ($H_u( II_{2,2+8k}\oplus A_2)$ is an irreducible divisor by the observation above).

One describes  the Heegner divisor  $H_u( II_{2,2+8k}\oplus A_2)$ proceeding as in~\Ref{subsubsec}{unigon4}. In fact, let $v\in(II_{2,2+8k}\oplus A_2)$ be a unigonal vector. Then 
$v^{\bot} \cong II_{2,2+8k}\oplus  A_1$. Given $g\in O^{+}( II_{2,2+8k}\oplus A_1)$, let $\wt{g}\in O(II_{2,2+8k}\oplus A_2)$ be the isometry which fixes $v$, and restricts to $g$ on $v^{\bot}$ (such a $g$ exists because $O^{+}( II_{2,2+8k}\oplus A_1)=\wt{O}^{+}( II_{2,2+8k}\oplus A_1)$). 
Then $\wt{g}\in \wt{O}^{+}(II_{2,2+8k}\oplus A_2)$, and we have an injection of groups
\begin{equation}
\begin{matrix}
 O^{+}( II_{2,2+8k}\oplus  A_1) & \hra & \wt{O}^{+}( II_{2,2+8k}\oplus A_2)\\
 g & \mapsto & \wt{g}
\end{matrix}
\end{equation}
It follows that we get a regular map
\begin{equation}\label{mappaerre}
\cF( II_{2,2+8k}\oplus  A_1)\overset{r}{\lra} H_u(II_{2,2+8k}\oplus A_2). 
\end{equation}
The proof of the  result below is similar to the proof of~\Ref{prp}{unigcong4}; we leave details to the reader.
\begin{proposition}\label{prp:unigatwo}
Let $k\ge 0$. The map $r$ of~\eqref{mappaerre} is an isomorphim onto the unigonal Heegner divisor $H_u(II_{2,2+8k}\oplus A_2)$. Moreover the intersection of 
$H_u(II_{2,2+8k}\oplus A_2)$ and the singular locus of $\cF(II_{2,2+8k}\oplus A_2)$ has codimension at least two in  $H_u(II_{2,2+8k}\oplus A_2)$. 
\end{proposition}
\subsection{Nested  locally symmetric varieties}\label{subsec:diserie}
\setcounter{equation}{0}
\subsubsection{The infinite tower of $D$ locally symmetric varieties}\label{subsubsec:torre}
\setcounter{equation}{0}
For $3\le M$, we let $\xi_M$ be a decoration of $\Lambda_M$. Choose $N\ge 3$. 
By~\Ref{prp}{hyplocsymm} we have a sequence of inclusions of $D$ period spaces:
\begin{equation}\label{catena}
\scriptstyle
\cF(\Lambda_3,\xi_3)\overset{f_4}{\hra} \cF(\Lambda_4,\xi_4)\overset{f_5}{\hra}\ \ \ldots \overset{f_{N-1}}{\hra} \cF(\Lambda_{N-1},\xi_{N-1})\overset{f_N}{\hra}\cF(\Lambda_N,\xi_N).
\end{equation}
Thus $\im f_M=H_h(\Lambda_M,\xi_M)$ for $4\le M\le N$.  There is a unique continuation of the above sequence. In fact, let $w\in\Lambda_N$ be a vector of divisibility $2$, and such that $[v/2]=\xi_N$ (e.g.~a hyperelliptic vector). Let $L\subset (\Lambda_N\oplus (-4))_{\QQ}$ be the sublattice generated by $(\Lambda_N\oplus (-4))$ and 
$(w/2,1/2)$. Then $L$ and $\Lambda_{N+1}$ are even lattices of signature $(2,N+1)$, and their discriminant groups (equipped with the discriminant quadratic forms) are isomorphic. By Theorem 1.13.2 of~\cite{nikulin} it follows that $L$ is isomorphic to $\Lambda_{N+1}$; we choose an identification of $L$ with $\Lambda_{N+1}$. Let $\xi_{N+1}$ be the decoration of $\Lambda_{N+1}$ (i.e.~$L$) defined by $(0,1/2)$. Then 
$v:=(0_{\Lambda_N},1)$ is a hyperelliptic vector of $\Lambda_{N+1}$, and $v^{\bot}=\Lambda_N$. Furthermore $\xi_N$ is the decoration of $v^{\bot}$ associated to $\xi_{N+1}$. In conclusion, there is an infinite prolongation of~\eqref{catena}, unique up to isomorphism:
\begin{equation}\label{inficatena}
\scriptstyle
\cF(\Lambda_3,\xi_3)\overset{f_4}{\hra} \ \ldots \overset{f_N}{\hra}\cF(\Lambda_N,\xi_N)\overset{f_{N+1}}{\hra} \cF(\Lambda_{N+1},\xi_{N+1})\overset{f_{N+2}}{\hra} \ \ldots.
\end{equation}
For $3\le M<N$ we let $f_{M,N}:=f_N\circ f_{N-1}\circ \ldots \circ f_{M+1}$. Thus
\begin{equation}\label{emmenne}
f_{M,N}\colon \cF(\Lambda_M,\xi_M)\hra \cF(\Lambda_N,\xi_N). 
\end{equation}
\subsubsection{Other building blocks  of the $D$ tower}\label{subsubsec:diriver}
\setcounter{equation}{0}
Let 
\begin{eqnarray}
\cF( II_{2,2+8k}) & \overset{l_{8k+3}}{\lra} & H_u(\Lambda_{8k+3},\xi_{8k+3}), \label{ellekappa} \\
\cF( II_{2,2+8k}\oplus A_1) & \overset{m_{8k+4}}{\lra} & H_u(\Lambda_{8k+4},\xi_{8k+4}), \label{emmekappa} \\
 \cF( II_{2,2+8k})) & \overset{p_{8k+3}}{\lra} & H_u(II_{2,2+8k}\oplus A_1), \label{pikappa} \\ 
\cF( II_{2,2+8k}\oplus A_2) & \overset{q_{8k+5}}{\lra} & H_u(\Lambda_{8k+5},\xi_{8k+5}), \label{qukappa} \\ 
\cF( II_{2,2+8k}\oplus A_1) & \overset{r_{8k+4}}{\lra} & H_u(II_{2,2+8k}\oplus A_2), \label{errekappa} 
\end{eqnarray}
be  the isomorphisms in~\eqref{mappagi}, \eqref{mappaelle}, \eqref{mappapi}, \eqref{mappaqu} and~\eqref{mappaerre} respectively - the convention is that the subscript denotes the dimension of the period space containing the codomain as Heegner divisor (as for the maps $f_N$). 
\begin{claim}\label{clm:triangolo}
Keeping notation as above, $f_{8k+4}\circ l_{8k+3}=m_{8k+4}\circ p_{8k+3}$, and $f_{8k+5}\circ m_{8k+4}=q_{8k+5}\circ r_{8k+4}$.
\end{claim}
\begin{proof}
Choose an isomorphism $\Lambda_{8k+4}\cong II_{2,2+8k}\oplus A_1^2$, and let $e_1,e_2$ be generators of the addend $ A_1^2$ such that $-2=e_1^2=e_2^2$, $(e_1,e_2)=0$. Let $\alpha,\beta,\gamma,\delta$ be the obvious inclusions in the  diagram 
\begin{equation}\label{rombo}
\xymatrix{
&   II_{2,2+8k}\oplus\la e_1+e_2\ra \ar@{^{(}->}[dr]^{\beta}  & \\
II_{2,2+8k}\ar@{^{(}->}[ur]^{\alpha} \ar@{^{(}->}[dr]^{\gamma} &   & II_{2,2+8k}\oplus\la e_1\ra\oplus \la e_2\ra \\
& II_{2,2+8k}\oplus\la e_1\ra \ar@{^{(}->}[ur]^{\delta} &
}
\end{equation}
The map $f_{8k+4}\circ l_{8k+3}$ is induced by the composition $\beta\circ\alpha$, and the map $m_{8k+4}\circ p_{8k+3}$ is induced by the composition 
$\delta\circ\gamma$. Thus $f_{8k+4}\circ l_{8k+3}=m_{8k+4}\circ p_{8k+3}$  because $\beta\circ\alpha=\delta\circ\gamma$. A similar proof shows that 
 $f_{8k+5}\circ m_{8k+4}=q_{8k+5}\circ r_{8k+4}$. 
\end{proof}
 The picture of the $D$ tower  is periodic of  period $8$:
\begin{equation}\label{diriver}
 \centering
\xy
\POS(-20,0) = "z",
\POS(0,0) *\cir<3pt>{}="a",
\POS(20,0) *\cir<3pt>{}="b",
\POS(40,0) *\cir<3pt>{}="c",
\POS(60,0) *\cir<3pt>{}="d",
\POS(20,10) *\cir<3pt>{}="e",
\POS(40,10) *\cir<3pt>{}="x",
\POS(0,20) *\cir<3pt>{}="r",
\POS(80,0) ="y",
%
\POS "z" \ar@{.>}_{f_{8k+2}} "a",
\POS "a" \ar@{->}_{f_{8k+3}} "b",
\POS "b" \ar@{->}_{f_{8k+4}} "c",
\POS "c" \ar@{->}_{f_{8k+5}} "d",
\POS "e" \ar@{->}^{m_{8k+4}} "c",
\POS "r" \ar@{->}_{l_{8k+3}} "b",
\POS "r" \ar@{->}^{p_{8k+3}} "e",
\POS "d" \ar@{.>}_{f_{8k+6}} "y",
\POS "e" \ar@{->}^{r_{8k+4}} "x",
\POS "x" \ar@{->}^{q_{8k+5}} "d",
\endxy 
\end{equation}
\subsubsection{Stratification of the support of the boundary divisor}\label{subsubsec:arrstrata}
Let $(\Lambda,\xi)$ be a decorated $D$-lattice, of dimension $N$. We let $\Delta^{(1)}(\Lambda,\xi)\subset\cF(\Lambda,\xi)$ be the \emph{support} of the boundary divisor, i.e.
\begin{equation}\label{deltasupport}
\Delta^{(1)}(\Lambda,\xi):=
\begin{cases}
H_h(\Lambda,\xi) & \text{if $N\not\equiv 3,4\pmod{8}$,}\\
H_h(\Lambda,\xi)\cup H_u(\Lambda,\xi) & \text{if $N\equiv 3,4\pmod{8}$.}
\end{cases}
\end{equation}
We let $\wt{\Delta}^{(1)}(\Lambda,\xi)\subset \cD^{+}_{\Lambda}$ be the inverse image of $\Delta^{(1)}(\Lambda,\xi)$ for the quotient map 
\begin{equation}\label{diverio}
\pi\colon \cD^{+}_{\Lambda}\to \cF(\Lambda,\xi).
\end{equation}
 Thus $\wt{\Delta}^{(1)}(\Lambda,\xi)$ is a linearized arrangement, in the terminology of Looijenga~\cite{looijenga1}, and it is naturally stratified. The $k$-th stratum   is defined to be
\begin{equation}\label{shyperplanes}
\scriptstyle
\wt{\Delta}^{(k)}(\Lambda,\xi):=\{[\sigma]\in \cD^{+}_{\Lambda} \mid \ \text{$\exists$ linearly independent non-nodal $\Gamma_{\xi}$-reflective $v_1,\ldots,v_k\in\sigma^{\bot}\cap\Lambda$}\},
\end{equation}
i.e.~the set of points belonging to $k$ (at least) independent \lq\lq hyperplanes\rq\rq\footnote{The present $k$ has \emph{no} relation to the $k$ appearing in~\Ref{subsubsec}{diriver}, which is  $\lfloor{N/8}\rfloor$  for $N\equiv 3,4\pmod{8}$}. 
Let
\begin{equation}\label{stratumes}
\Delta^{(k)}(\Lambda,\xi):=\pi(\wt{\Delta}^{(k)}(\Lambda,\xi)),
\end{equation}
where $\pi$ is the quotient map~\eqref{diverio}. 
The strata $\wt{\Delta}^{(k)}(\Lambda,\xi)$ and $\Delta^{(k)}(\Lambda,\xi)$ play a key r\^ole in Looijenga's semi-toric compactification of the complement of $\Delta^{(1)}(\Lambda,\xi)$ in $\cF(\Lambda,\xi)$. 
We will show that  the subvarieties of $\cF(\Lambda,\xi)$ appearing in~\eqref{diriver} are exactly the irreducible components of $\Delta^{(k)}(\Lambda,\xi)$. 
In order to state our results, let
\begin{eqnarray}\label{hyperandunig}
\scriptstyle \cH_h^{(k)}(\Lambda,\xi) & \scriptstyle := & 
\scriptstyle \{[\sigma]\in \cD^{+}_{\Lambda} \mid \ \text{$\exists$  linearly independent  hyperelliptic $v_1,\ldots,v_k\in\sigma^{\bot}\cap\Lambda$}\}, \\
\scriptstyle \cH_u^{(k)}(\Lambda,\xi) & \scriptstyle := & 
\scriptstyle \{[\sigma]\in \cD^{+}_{\Lambda} \mid \ \text{$\exists$  linearly independent   reflective unigonal $v_1,\ldots,v_k\in\sigma^{\bot}\cap\Lambda$}\}.
\end{eqnarray}
We let $H_h^{(k)}(\Lambda,\xi)\subset\cF(\Lambda,\xi)$ and  $H_u^{(k)}(\Lambda,\xi)\subset\cF(\Lambda,\xi)$ be the images via $\pi$ of $\cH_h^{(k)}(\Lambda,\xi)$ and $\cH_u^{(k)}(\Lambda,\xi)$  respectively. 
In order to simplify notation, we let  
\begin{equation}
\Delta^{(k)}(N)=\Delta^{(k)}(\Lambda_N,\xi_N),\quad H^{(k)}_h(N)=H^{(k)}_h(\Lambda_N,\xi_N)\quad H^{(k)}_u(N)=H^{(k)}_u(\Lambda_N,\xi_N),
\end{equation}
where  $\xi_N$ is  a decoration of  $\Lambda_N$. Below is the main result of the present subsubsection. 
\begin{proposition}\label{prp:deltakappa}
Let $N\ge 3$, and keep notation as above. Then
\begin{equation}\label{onlyhyper}
\scriptstyle \Delta^{(k)}(N):=
\begin{cases}
\scriptstyle  H^{(k)}_h(N) & \scriptstyle \text{if $N\not\equiv 3,4\pmod{8}$, or  $k\ge 2$,}\\
\scriptstyle H^{(k)}_h(N)\cup H^{(k)}_u(N) & \scriptstyle \text{if $N\equiv 3,4\pmod{8}$ and $k=1$.}
\end{cases}
\end{equation}
If, in addition, $N\ge 4$, and  $1\le k\le N-3$, then
\begin{equation}\label{onksheets}
H^{(k)}_h(N) =
\begin{cases}
\im f_{N-k,N} & \text{if $k\not\equiv N-2\pmod{8}$,} \\
\im f_{N-k,N}\cup \im(f_{N-k+1,N}\circ l_{N-k+1}) & \text{if $k\equiv N-2\pmod{8}$.}
\end{cases}
\end{equation}
\end{proposition}
We will prove~\Ref{prp}{deltakappa} at the end of the present subsubsection.
\begin{lemma}\label{lmm:inthypunig}
Let  $(\Lambda,\xi)$ be a    $D$-lattice, of dimension $N\equiv 3\pmod{8}$. Then 
$$\es=H_h(\Lambda,\xi)\cap H_u(\Lambda,\xi)=H^{(2)}_u(\Lambda,\xi).$$  
\end{lemma}
\begin{proof}
Suppose the contrary. Then there exists $[\sigma]\in\cD_{\Lambda_N}^{+}$ such that $\sigma^{\bot}$ contains a unigonal vector $w\in\Lambda_N$ (i.e.~$w^2=-4$ and $\divisore(w)=4$), and a vector $v$ which is either  hyperelliptic or unigonal, and moreover $v\not=\pm w$. It follows that $\la v,w\ra$ is a rank-$2$ negative definite lattice ($\sigma^{\bot}\cap\Lambda_{\RR}$ is negative definite), and hence the determinant of $\la v,w\ra$ is strictly positive. This in turn implies that  $v\bot w$ (recall that $(v,w)\in 4\ZZ$). 
 On the other hand, $\Lambda=\ZZ w\oplus w^{\bot}$, because $w^2=-4$ and $\divisore(w)=4$. Since $\Lambda$ has determinant $-4$, it follows that the lattice $w^{\bot}$ is unimodular. Now, $v$ belongs to the unimodular lattice $w^{\bot}$, and $\divisore_{\Lambda}(v)\in\{2,4\}$; that is a contradiction.
\end{proof}
\begin{lemma}\label{lmm:twounig}
Let  $(\Lambda,\xi)$ be a    $D$-lattice, of dimension $N\equiv 4\pmod{8}$. If $w_1,w_2$ are non-proportional unigonal vectors of $(\Lambda,\xi)$, spanning a negative definite sublattice, then   $\la w_1,w_2\ra$  is isomorphic to $D_2$, and  
\begin{equation}\label{rabbia}
\Lambda=\la w_1,w_2\ra\oplus \la w_1,w_2\ra^{\bot},\qquad \la w_1,w_2\ra^{\bot}\cong \II_{2,N-2}.
\end{equation}
In particular $\{w_1^{*},w_2^{*}\}=\{\zeta,\zeta'\}$, where, as usual $A_{\Lambda}=\{0,\xi,\zeta,\zeta'\}$.
\end{lemma}
\begin{proof}
Computing the determinant of the quadratic form on $\la w_1,w_2\ra$ (which is positive by hypothesis), we get that $w_1\bot w_2$. Since $(\zeta,\zeta)=(\zeta',\zeta')\equiv -1/2\pmod{2\ZZ}$, it follows that $w_1^{*}\not= w_2^{*}$. The first equality 
of~\eqref{rabbia} follows at once, and since  $\la w_1,w_2\ra^{\bot}$ is unimodular, the isomorphism $\la w_1,w_2\ra^{\bot}\cong \II_{2,N-2}$ follows from the classification of unimodular indefinite even lattices.
\end{proof}
\begin{lemma}\label{lmm:unigandhyp}
Let  $(\Lambda,\xi)$ be a    $D$-lattice, of dimension $N\equiv 4\pmod{8}$. Let $v,w$ be a hyperelliptic and a unigonal vector of $(\Lambda,\xi)$ respectively, spanning a negative definite sublattice. Then  $\la v,w\ra$  is isomorphic to $D_2$, and  
\begin{equation}\label{ditwoperp}
\Lambda=\la v,w\ra\oplus \la v,w\ra^{\bot},\qquad \la v,w\ra^{\bot}\cong \II_{2,N-2}.
\end{equation}
\end{lemma}
\begin{proof}
Computing the determinant of the quadratic form on $\la v,w\ra$ (which is positive by hypothesis), we get that $|(v,w)|\in\{0,2\}$. Since $w^{\bot}$ is unimodular (see the proof of~\Ref{lmm}{twounig}), we must have 
$|(v,w)|=2$, and hence we may assume that $(v,w)=2$. It follows that $\la v,w\ra$  is isomorphic to $D_2$; in fact  $\{v+w,w\}$ is a standard  basis, i.e.~$-2=(v+w)^2=w^2$ and $(v+w)\bot w$. Then~\eqref{ditwoperp} follows as in the proof of~\Ref{lmm}{twounig}.
\end{proof}
\begin{lemma}\label{lmm:hyperort}
Let $(\Lambda,\xi)$ be a decorated $D$-lattice, of dimension $N\ge 4$. Let $v$ be a hyperelliptic vector of $(\Lambda,\xi)$, and $\eta$ be the decoration of the $D$-lattice $v^{\bot}$ defined in~\Ref{rmk}{inducedec}.  Suppose that $w$ is a hyperelliptic vector of $(\Lambda,\xi)$, \emph{orthogonal} to $v$. Then one of the following holds:
\begin{enumerate}
\item
$w$ has divisibility $2$ in $v^{\bot}$, and it is a hyperelliptic vector of $(v^{\bot},\eta)$. 
\item
$w$ has divisibility $4$ in $v^{\bot}$, $\{v,w\}^{\bot}\cong II_{2,N-2}$, and $N\equiv 4\pmod{8}$.
\end{enumerate}
\end{lemma}
\begin{proof}
By hypothesis the divisibility of $w$ in $\Lambda$ is $2$, hence  the divisibility of $w$ in $v^{\bot}$ is a multiple of $2$. Since $v^{\bot}$ is a $D$-lattice, $w^{*}\in A_{v^{\bot}}$ has order  $2$ or $4$; it follows that   the divisibility of $w$ in $v^{\bot}$ is either $2$ or $4$. 

Suppose that $\divisore_{v^{\bot}}(w)=2$. We must prove  that $w^{*}=\eta$.  Since $v,w$ are  hyperelliptic vectors of $(\Lambda,\xi)$, there exists $y\in\Lambda$ such that $w=v+2y$. 
There exists a hyperelliptic vector $u$ of $(v^{\bot},\eta)$ such that $u^{*}=\eta$. By definition of $\eta$, the vector $z:=(v-u)/2$ belongs to $\Lambda$. Thus  $v=u+2z$, and  hence $w=u+2(y+z)$, and $(y+z)\in v^{\bot}$; it follows that $w^{*}=u^{*}=\eta$. 

Now suppose that $\divisore_{v^{\bot}}(w)=4$.  Then    $v^{\bot}=\la w\ra \oplus \{v,w\}^{\bot}$ because $w^2=-4$. On the other hand $v^{\bot}\cong \Lambda_{N-1}$ by~\Ref{lmm}{prophyp}.   Since     $|\det  \Lambda_{N-1}|=4$, it follows that 
$\{v,w\}^{\bot}$ is unimodular, and hence  isomorphic to $II_{2,N-2}$. The equality  $N\equiv 4\pmod{4}$ holds by the classification of unimodular even lattices.
\end{proof}
\begin{proposition}\label{prp:pairort}
Let  $N\ge 4$ and $1\le k\le N-3$. Let  $(\Lambda,\xi)$ be a dimension-$N$   $D$-lattice, and    $v_1,\ldots,v_k$ be pairwise orthogonal hyperelliptic vectors of  
$(\Lambda,\xi)$. Then   for any  $1\le i\le k-1$, 
\begin{equation}
\{v_1,\ldots,v_i\}^{\bot}\cong \Lambda_{N-i},
\end{equation}
 $\{v_1,\ldots,v_i\}^{\bot}$ carries a decoration $\xi(i)$ such that $(\{v_1,\ldots,v_{i-1}\}^{\bot},\xi(i-1))$ induces  $(\{v_1,\ldots,v_i\}^{\bot},\xi(i))$ as
 in~\Ref{rmk}{inducedec} (for $i=1$, this means that $(\Lambda,\xi)$ induces $(v_1^{\bot},\xi(1))$), and $v_i$ is a hyperelliptic vector of 
  $(\{v_1,\ldots,v_i\}^{\bot},\xi(i))$.
Moreover one of the following holds:
 \begin{enumerate}
\item
 $v_k$ is a hyperelliptic vector of 
  $(\{v_1,\ldots,v_{k-1}\}^{\bot},\xi(k-1))$.
\item
$k\ge 2$,  $N\equiv k+2\pmod{8}$,   $v_k$ has divisibility $4$ in $\{v_1,\ldots,v_{k-1}\}^{\bot}$, and   
$\{v_1,\ldots,v_k\}^{\bot}\cong \II_{2,N-k}$.  
\end{enumerate}
\end{proposition}
\begin{proof}
By induction on $k$. If $k=1$, the proposition holds by~\Ref{lmm}{prophyp}, if $k=2$ it holds by~\Ref{lmm}{hyperort}.  Now let $k\ge 3$. By the inductive hypothesis, 
$\{v_1,\ldots,v_{k-2}\}^{\bot}$, $\{v_1,\ldots,v_{k-1}\}^{\bot}$,   and $\{v_1,\ldots,v_{k-2},v_k\}^{\bot}$ are either $D$-lattices or unimodular. Since $v_{k-1}$ and $v_k$ have divisibility $2$ in $\Lambda$ it follows that 
$\{v_1,\ldots,v_{k-2}\}^{\bot}$, $\{v_1,\ldots,v_{k-1}\}^{\bot}$,   and $\{v_1,\ldots,v_{k-2},v_k\}^{\bot}$ are
 not unimodular, and  
hence  they are $D$-lattices. In particular, by~\Ref{lmm}{hyperort}, $v_{k-1}$ and $v_k$  are  hyperelliptic vectors of  $(\{v_1,\ldots,v_{k-2}\}^{\bot},\xi(k-2))$. The proposition follows by applying~\Ref{lmm}{hyperort} to the decorated $D$-lattice 
 $(\{v_1,\ldots,v_{k-2}\}^{\bot},\xi(k-2))$ and the vectors $v=v_{k-1}$, $w=v_k$.
\end{proof}
{\it Proof of~\Ref{prp}{deltakappa}.\/}
Let us prove~\eqref{onlyhyper}. It is obvious that $\Delta^{(k)}(N)$ contains the set on the right-hand side. Next, we check that the left-hand side of~\eqref{onlyhyper} is contained in the  right-hand side. If $N\not\equiv 3,4\pmod{8}$, the containment is obvious, because there are no reflective unigonal vectors. If $N\equiv 3\pmod{8}$ (and $k\ge 2$), the containment follows from~\Ref{lmm}{inthypunig}.  It remains to prove that if $N\equiv 4\pmod{8}$, then  
$\Delta^{(2)}(N)\subset H_h^{(2)}(N)$; 
this is an easy consequence of~\Ref{lmm}{twounig}. 

Lastly, we prove~\eqref{onksheets}. Suppose that $\pi([\sigma])\in H_h^{(k)}(N)$, i.e.~ there exist linearly independent  hyperelliptic vectors $v_1,\ldots,v_k$ of $(\Lambda_N,\xi_N)$ such that 
  $\sigma\in\{v_1,\ldots,v_k\}^{\bot}$. The latter condition implies that the restriction of $(,)_{\Lambda}$ to the span of  $v_1,\ldots,v_k$ is negative definite, and hence the  vectors $v_1,\ldots,v_k$ are pairwise orthogonal by~\Ref{lmm}{manyhyp}. Thus we may apply~\Ref{prp}{pairort}; if Item~(1) holds, then $\pi([\sigma])\in\im f_{N-k,N}$, if Item~(2) holds, then $N\equiv k+2\pmod{8}$ and $\pi([\sigma])\in \im(f_{N-k+1,N}\circ l_{N-k+1})$. This proves that the left-hand side of~\eqref{onksheets} is contained in the right-hand side. 
The converse is obvious.
\qed
\begin{remark}
Let $N=8s+4$ where $s\ge 0$. We have proved that
\begin{equation}\label{manyunig}
H_u^{(2)}(N)=\im(f_{8s+4}\circ l_{8s+3})=\im(m_{8s+4}\circ p_{8s+3}).
\end{equation}
In particular, it follows from~\Ref{lmm}{inthypunig} that $\im f_{N-2,N}\cap H_u^{(2)}(N)=\es$.
\end{remark}

\section{Locally symmetric spaces $\cF(N)$  as period spaces for $N\in\{18,19,20,21\}$}\label{sec:gitper}
We will denote by $\cF(N)$ the locally symmetric variety $\cF(\Lambda_N,\xi_N)$, where  $(\Lambda_N,\xi_N)$ is a decorated dimension-$N$ $D$-lattice.    As is well-known, the period space for quartic $K3$ surfaces is $\cF_{\Lambda_{19}}(\widetilde O^+(\Lambda_{19}))$, and the latter is equal to $\cF(19)$   by~\Ref{prp}{propgroup}.  
 We will prove that   $\cF(18)$  is the period space of hyperelliptic quartic $K3$ surfaces, see~\Ref{subsec}{hyperper}, and that $\cF(20)$ is the period space of desingularized EPW-sextics (a quotient of the period space of double EPW sextics by the natural duality involution, see~\cite{epwperiods}).  
   Lastly, in~\Ref{subsec}{epwper}  we will showand we will establish a relation between $\cF(21)$ and the period  space of hyperk\"ahler $10$-folds of Type OG10 equipped with a certain polarization.

\bigskip

\begin{notation}\label{shortnotation} 
We will denote by $H_h(N),H_u(N)$ the hyperelliptic and unigonal Heegner divisors in $\cF(N)$.
\end{notation}
\subsection{Quartic $K3$ surfaces and their periods}\label{subsec:quartper}
\setcounter{equation}{0}
  For us a \emph{$K3$ surface} is a complex projective surface $X$ with DuVal singularities,  trivial dualizing sheaf  $\omega_X$, and $H^1(\cO_X)=0$. 
A  \emph{polarization} of degree $d$ on a $K3$ surface $X$ is an ample line bundle $L$ on $X$, such that $c_1(L)$ is primitive, and $c_1(L)\cdot c_1(L)=d$. The degree $d$ of a polarization is strictly positive and even, hence we may write $d=2g-2$, with $g\ge 2$; then $g$ is the arithmetic genus of curves in $|L|$. A \emph{polarized} $K3$ surface is a couple  $(X,L)$, where $X$ is a $K3$ surface and $L$ is a polarization of $X$; the degree and genus of $(X,L)$ are defined to be the degree and genus of $L$. We recall the following fundamental result.
\begin{theorem}[Mayer~\cite{mayer}]\label{thm:mayerthm}
Let $(X,L)$ be a polarized $K3$ of genus $g$. Then one of the following holds:
\begin{enumerate}
\item
The map $\varphi_L\colon S\dra |L|^{\vee}\cong\PP^g$ defined by $L$ is regular, and is an isomorphism onto its image (a surface whose generic hyperplane section is a canonical curve of genus $g$).
\item
The map $\varphi_L\colon S\dra |L|^{\vee}\cong\PP^g$ is regular, and is a double cover of its image (a rational surface of degree $g-1$).
\item
The linear system $|L|$ has a fixed component, a smooth rational curve $R$, and $L(-R)\cong\cO_X(gE)$, where $E$ is an elliptic curve.
\end{enumerate}
\end{theorem}
Actually Mayer considered  big and nef divisors on a smooth $K3$ surface; one gets the result for a singular $K3$ surface $X$ upon replacing $X$ by its minimal desingularization $\wt{X}$, and $L$ by the pull-back $\wt{L}$ of $L$ to $\wt{X}$. 
 \begin{definition}
 A polarized K3 surface  $(X,L)$ of degree $4$   is \emph{hyperelliptic} if Item~(2) of~\Ref{thm}{mayerthm}  holds, and it is \emph{unigonal} if Item~(3) 
  of~\Ref{thm}{mayerthm} holds. 
 \end{definition}
\begin{remark}\label{rmk:eccohyp}
One constructs hyperelliptic  quartics as follows. Let $Q\subset\PP^3$  be an irreducible quadric, i.e.~either a smooth quadric, or a quadric cone over a smooth conic. Let   $B\in |\omega_Q^{-2}|$, and suppose that $B$ has simple singularities. Let $\varphi\colon X\to Q$ be the double cover ramified over $B$, and  let $L:=\varphi^{*}\cO_Q(1)$. Then  $(X,L)$  is a hyperelliptic quartic $K3$ surface, and the image of $\varphi_L\colon X\to |L|^{\vee}$ is identified with $Q$. Conversely, every hyperelliptic quartic $K3$ surface is obtained by this procedure.
\end{remark}
 \begin{remark}\label{rmk:eccounig}
One constructs unigonal quartics as follows.   Let $\FF_4:=\PP(\cO_{\PP^1}(4)\oplus\cO_{\PP^1})$. Let  $\rho\colon \FF_4\to\PP^1$ be the structure map, let  $F$ be a fiber of $\rho$, and let
  $A:=\PP(\cO_{\PP^1}(4))\subset \FF_4$ be the negative section of $\rho$.  Let  $B\in|3A+12F|$ be reduced with simple singularities, and disjoint from $A$ (the generic divisor in $|3A+12F|$ has these properties, because $B\cdot A=0$). Let $\pi\colon X\to \FF_4$ be the double cover branched over $A+B$. Then $X$ is a $K3$ surface, because $(A+B)\in |-2K_{\FF_4}|$.  Since 
   $\pi$ is simply ramified over $A$, we have $\pi^{*}A=2R$, where $R$ is a smooth rational curve. Let $E:=\pi^{*}F$; thus $E$ is an elliptic curve, moving in the elliptic fibration $\rho\circ\pi\colon X\to\PP^1$. 
 Then  $(X,\cO_X(R+3E))$ is a unigonal quartic $K3$ surface, and conversely every  unigonal quartic is obtained by this procedure.  
 \end{remark}
   If $(X,L)$  is a quartic $K3$ surface, the period point $\Pi(X,L)\in \cF(19)$ is defined as follows. Let $\wt{X}$ be the minimal desingularization of $X$ and $\wt{L}$ be the pull back of $L$ to $\wt{X}$. The lattices $c_1(\wt{L})^{\bot}_{\ZZ}:=c_1(\wt{L})^{\bot}\cap H^2(\wt{X};\ZZ)$ and $\Lambda_{19}$ are isomorphic; let 
\begin{equation}\label{primisom}
\psi\colon     c_1(\wt{L})^{\bot}_{\ZZ}\overset{\sim}{\lra} \Lambda_{19} 
\end{equation}
be an  isomorphism, and $\psi_{\CC}\colon     c_1(\wt{L})^{\bot}\overset{\sim}{\lra} \Lambda_{19,\CC}$ be the $\CC$-linear extension of $\psi$. The line $\psi_{\CC}(H^{2,0}(\wt{X}))$ belongs to $\cD_{\Lambda_{19}}$, and, if necessary, we may  precompose $\psi$ with an element of $\wt{O}(c_1(\wt{L})^{\bot})$ so that 
  $\psi_{\CC}(H^{2,0}(\wt{X}))$  belongs to  $\cD^{+}_{\Lambda_{19}}$. Such a point is well determined up to 
  $\wt{O}^{+}(\Lambda_{19})$, hence it determines a well-defined 
\begin{equation}\label{periodok3}
\Pi(X,L)\in \wt{O}^{+}(\Lambda_{19})\backslash\cD^{+}_{\Lambda_{19}}= \cF(19).
\end{equation}
This is the \emph{period point} of $(X,L)$. 

Now, let 
\begin{equation}\label{emme19}
\gM(19)  :=  |\cO_{\PP^3}(4)|\gquot\PGL(4)  
\end{equation}
be the GIT moduli space of quartic surfaces in $\PP^3$ (see~\cite{shah4} for many results about $\gM(19)$). 
\begin{definition}\label{dfn:u19}
Let   $\cU(19)\subset |\cO_{\PP^3}(4)|$ be the open subset  parametrizing quartics with ADE singularities. 
\end{definition}
Then $\cU(19)$ is contained in the stable locus of 
$|\cO_{\PP^3}(4)|$  for the natural action of $\PGL(4)$ by~\cite{shah4}. 
By Global Torelli for $K3$ surfaces, and Mayer's~\Ref{thm}{mayerthm}, the period map restricted to $\cU(19)\gquot\PGL(4)$  defines an isomorphism 
\begin{equation}\label{juventus}
\cU(19)\gquot\PGL(4) \overset{\sim}{\lra} (\cF(19)\setminus H_h(19)\setminus H_u(19)).
\end{equation}
Thus we have a birational map
\begin{equation}\label{per19}
\gp_{19}\colon\gM(19)\dra \cF(19)^{*},
\end{equation}
and one of the main goals of the present paper is to propose a conjectural decomposition of $\gp_{19}^{-1}$ as a composition of the $\QQ$-factoralization of $\cF(19)^{*}$, and a series of flops of 
subloci described by the $D$ tower of~\Ref{subsec}{diserie}.   
\subsection{Periods  of hyperelliptic and unigonal quartics}\label{subsec:hyperper}
\setcounter{equation}{0}
Below is the main result of the present subsection.
\begin{proposition}\label{prp:perhyp}
Let  $(X,L)$ be a quartic $K3$ surface. Then
\begin{enumerate}
\item
$(X,L)$  is hyperelliptic if and only if $\Pi(X,L)\in H_h(19)$, and 
\item
$(X,L)$  is unigonal if and only if $\Pi(X,L)\in H_u(19)$.
\end{enumerate}
\end{proposition}
\Ref{prp}{perhyp} will be proved at the end of the present subsection.  
\begin{lemma}\label{lmm:estensione}
Let  $(X,L)$ be a hyperelliptic quartic $K3$ surface. Let  $\wt{X}$ be the minimal desingularization of $X$, and $\wt{L}$ the pull-back of $L$ to $\wt{X}$. 
Then there exists $\epsilon_L\in H^2(\wt{X};\ZZ)$ perpendicular to $c_1(\wt{L})$  such that 
\begin{equation}\label{divquattro}
\frac{c_1(\wt{L})+\epsilon_L}{4}\in H^2(\wt{X};\ZZ),
\end{equation}
and such that $H^2(\wt{X};\ZZ)$ is generated (over $\ZZ$) by $\ZZ c_1(\wt{L})\oplus c_1(\wt{L})^{\bot}$ and the element in~\eqref{divquattro}. Moreover 
\begin{equation*}
\divisore_{c_1(\wt{L})^{\bot}}(\epsilon_L)=4,\quad q_{c_1(\wt{L})^{\bot}}(\epsilon_L^{*})\equiv -1/4\pmod{2\ZZ}.
\end{equation*}

\end{lemma}
\begin{proof}
Existence of $\epsilon_L\in c_1(\wt{L})^{\bot}$ such that~\eqref{divquattro} holds follows from unimodularity of the intersection form on  $H^2(\wt{X};\ZZ)$. The remaining statements follow at once from~\eqref{divquattro}.
\end{proof}
\begin{lemma}\label{lmm:pianoiper}
Let  $(X,L)$ be a hyperelliptic quartic $K3$ surface. Let  $\wt{X}$ be the minimal desingularization of $X$, and $\wt{L}$ the pull-back of $L$ to $\wt{X}$. 
Then 
\begin{enumerate}
\item
$\Pi(X,L)\in H_h(19)$   if and only if  there exists a primitive sublattice $U(2)\subset H^{1,1}(\wt{X};\ZZ)$ containing $c_1(\wt{L})$, and 
\item
$\Pi(X,L)\in H_u(19)$   if and only if  there exists a  sublattice $U\subset H^{1,1}(\wt{X};\ZZ)$  (necessarily primitive) containing $c_1(\wt{L})$.
\end{enumerate}
\end{lemma}
\begin{proof}
(1): First, note that $\Pi(X,L)\in H_h(19)$   if and only if  there exists  $\alpha\in H^{1,1}(\wt{X};\ZZ)$ such that 
\begin{equation}\label{spalletti}
c_1(\wt{L})\bot \alpha,\qquad \alpha^2=-4,\qquad \divisore_{c_1(\wt{L})^{\bot}}(\alpha)=2.
\end{equation}
Now suppose that $\Pi(X,L)\in H_h(19)$, and let $\alpha$ be as above. Then $A_{c_1(\wt{L})^{\bot}}\cong\ZZ/(4)$, and hence $\alpha^{*}=2\epsilon_L^{*}$, where 
$\epsilon_L$ is as in~\Ref{lmm}{estensione}. It follows that $H^{1,1}(\wt{X};\ZZ)\ni (c_1(\wt{L})\pm \alpha)/2$. Then
\begin{equation}
U(2)\cong \la (c_1(\wt{L})+\alpha)/2,(c_1(\wt{L})- \alpha)/2\ra\subset H^{1,1}(\wt{X};\ZZ).
\end{equation}
The above lattice is saturated because $ \divisore_{c_1(\wt{L})}(\alpha)=2$. To prove the converse, suppose that there exists a primitive sublattice $U(2)\subset H^{1,1}(\wt{X};\ZZ)$ containing $c_1(\wt{L})$. Let $\alpha$ be a generator $c_1(\wt{L})^{\bot}\cap  U(2)$; it suffices to show that~\eqref{spalletti} holds.  The equality $\alpha^2=-4$ is clear, and  $2$ divides $\divisore_{c_1(\wt{L})^{\bot}}(\alpha)$ because there exists $\beta\in U(2)$ such that $c_1(\wt{L})+\alpha=2\beta$. It remains to prove that $\divisore_{c_1(\wt{L})^{\bot}}(\alpha)=4$ can not hold. If it holds, then $\alpha^{*}=\epsilon_L^{*}$, where $\epsilon_L$ is as
 in~\Ref{lmm}{estensione}, and hence $H^{1,1}(\wt{X};\ZZ)\ni (c_1(\wt{L})+\alpha)/4$; then $U(2)$ is not saturated, against the hypothesis. This proves Item~(1). The proof of Item~(2) is similar, we leave details to the reader. 
\end{proof}
\n
{\it Proof of~\Ref{prp}{perhyp}.\/} First we prove Item~(2). Suppose that $(X,L)$ is unigonal, and hence $X$ (and $L$) are as in~\Ref{rmk}{eccounig}. Let $\wt{X}\to X$ be the minimal desingularization of $X$, and $\wt{R},\wt{E}$ be the pull-backs to $\wt{X}$ of the divisor classes $R,E$ on $X$ defined in~\Ref{rmk}{eccounig}.
Then $\wt{R},\wt{E}$ span a hyperbolic lattice in $H^{1,1}(\wt{X};\ZZ)$ containing $c_1(\wt{L})$, and hence $\Pi(X,L)\in H_u(19)$  by~\Ref{lmm}{pianoiper}. Now assume that  $\Pi(X,L)\in H_u(19)$. By~\Ref{lmm}{pianoiper} there exists a hyperbolic plane  $U\subset H^{1,1}(\wt{X};\ZZ)$ containing $c_1(\wt{L})$. It follows that there exist cohomology classes  $r,e\in U$ such that $c_1(\wt{L})=r+3e$, and $r^2=-2$, $r\cdot e=1$, and $e^2=0$. By Riemann-Roch, both $r$ and $e$ are effective; it follows that  $r$ is the class of a rational curve which is the base-locus of $\varphi_L$, and $e$ is the class of the fiber of an elliptic 
fibration.  Thus $(X,L)$ is unigonal. 

Next, let us  prove Item~(1). Suppose that $(X,L)$ is hyperelliptic, and hence is as in~\Ref{rmk}{eccohyp}.  Let $Q:=\varphi_L(X)\subset|L|^{\vee}\cong\PP^3$. Let $\wt{X}\to X$ be the minimal desingularization, and let 
 $\wt{\varphi}\colon\wt{X}\to Q$ be the composition of  $\wt{X}\to X$ and $\varphi_L$. If $Q$ is a  quadric cone, with vertex $v$, let $Q_0\cong\FF_2=\PP(\cO_{\PP^1}(2)\oplus\cO_{\PP^1})$ be the blow-up of $Q$ at $v$; the map $\wt{\varphi}\colon\wt{X}\to Q$ lifts to a map $\wt{\varphi}_0\colon\wt{X}\to Q_0$. Let 
\begin{equation}\label{eliogermano}
M:=
\begin{cases}
\wt\varphi^{*}H^2(Q;\ZZ) & \text{if $Q$ is a smooth quadric,} \\
\wt\varphi_0^{*}H^2(Q_0;\ZZ) & \text{if $Q$ is a quadric cone.}
\end{cases}
\end{equation}
Then $M\cong U(2)\subset H^{1,1}(\wt{X};\ZZ)$, and it contains  $c_1(\wt{L})$. By~\Ref{lmm}{pianoiper}, in order to prove that $\Pi(X,L)\in H_h(19)$ it suffices to check  that  $M$ is primitive.
  
Suppose first that $Q$ is  smooth.  Since the family of $K3$ surfaces which are double covers of $\PP^1\times\PP^1$ is irreducible it suffices to prove that  $\wt{\varphi}^{*}H^2(Q;\ZZ)$ is primitive for the double cover of $Q=\PP^1\times\PP^1$ ramified over a   particular smooth $B\in |\cO_{\PP^1}(4)\boxtimes \cO_{\PP^1}(4)|$; it suffices to take $B$ which is bitangent to lines $L_1,L_2\subset Q$ belonging to distinct rulings. 

Now   suppose  that $Q$ is  a quadric cone.
As in the previous case,  it will suffice to prove   that  $M$ is primitive for a particular double cover of $Q$. Let 
$B_0\in |\omega^{-2}_{Q_0}|$ be a  smooth curve  which is bitangent to a fiber of the structure map $Q_0\to\PP^1$, and is otherwise generic. The double cover $\wt{X}\to Q_0$ branched over $B_0$ is the minimal desingularization of a hyperelliptic quartic $K3$, and  $M=\wt{\varphi}_0^{*}H^2(Q_0;\ZZ)$ is primitive. This  proves that if $(X,L)$ is a hyperelliptic quartic, then   $\Pi(X,L)\in H_h(19)$. 

Lastly, suppose that $\Pi(X,L)\in H_h(19)$, and let us prove that $(X,L)$ is a hyperelliptic quartic. By~\Ref{lmm}{pianoiper}, there exists  a primitive $U(2)\subset H^{1,1}(\wt{X};\ZZ)$ containing  $c_1(\wt{L})$. 
 There exists $\alpha\in U(2)$ such that $\alpha^2=0$ and $\alpha\cdot c_1(\wt{L})=2$. Thus $\alpha$ is represented by an  effective divisor $D$ such that $D\cdot D=0$ and $\wt{L}\cdot D=2$. Hence  the divisor $D$ moves by Riemann-Roch; it follows that $\varphi_L(X)$ is not a quartic surface with DuVal singularities, and hence $(X,L)$ is either hyperelliptic or unigonal. If $(X,L)$ is unigonal, then $\Pi(X,L)\in H_u(19)$ (we have just proved it), and we get a contradiction because $H_h(19)\cap H_u(19)=\es$ by~\Ref{lmm}{inthypunig}. Thus  $(X,L)$ is hyperelliptic.
\qed
\begin{remark}\label{rmk:perhyp}
 By~\Ref{prp}{hyplocsymm}, there is a natural isomorphism between $\cF(18)$ and the hyperelliptic divisor $H_h(19)$. By~\Ref{prp}{perhyp}, it follows that we may identify $\cF(18)$ with the period space for hyperelliptic quartic $K3$ surfaces, and this is what we will do for the rest of the paper. 
\end{remark}
There is a GIT moduli space which is in a relation to $\cF(18)$ that is similar to the relation between $\gM(19)$ and $\cF(19)$. 
In fact, let
\begin{eqnarray}\label{emme18}
\gM(18)  :=  |\cO_{\PP^1}(4)\boxtimes\cO_{\PP^1}(4)|\gquot\Aut(\PP^1\times\PP^1).
\end{eqnarray}
\begin{definition}\label{dfn:u18}
Let  $\cU(18)\subset|\cO_{\PP^1\times\PP^1}(4,4)|$ be the open subset parametrizing curves with ADE singularities.
\end{definition}
 If $D\in\cU(18)$, the double cover $\pi\colon X\to\PP^1\times\PP^1$ branched over
  $D$ is a $K3$ surface, and $(X,\pi^{*}(\cO_{\PP^1}(1)\boxtimes \cO_{\PP^1}(1))$ is a hyperelliptic quartic $K3$ surface. 
 Now,   $\cU(18)$ is contained in the stable locus of 
$|\cO_{\PP^\times\PP^1}(4,4)|$ by~\cite[Theorem 4.8]{shah4}, and by associating to the orbit of $D\in\cU(18)$ 
the period of $(X,\pi^{*}(\cO_{\PP^1}(1)\boxtimes \cO_{\PP^1}(1))$ (notation as above), one gets an isomorphism 
\begin{equation}\label{bayern}
\cU(18)\gquot\Aut(\PP^1\times\PP^1) \overset{\sim}{\lra} (\cF(18)\setminus H_h(18)).
\end{equation}
Thus we have a birational period map 
\begin{equation}\label{periodi18}
\gp_{18}\colon\gM(18)\dra \cF(18)^{*},
\end{equation}
whose behaviour is very similar to that of $\gp_{19}$.

\subsection{Periods of certain higher-dimensional hyperk\"ahler varieties}\label{subsec:epwper}
\setcounter{equation}{0}
\subsubsection{Double EPW sextics}
Let $(X,L)$ be a polarized   HK variety of Type $K3^{[2]}$, where $q(c_1(L))=2$ ($q$ is the Beauville-Bogomolov quadratic formon $H^2(X)$).
 Then 
\begin{equation*}
c_1(L)^{\bot}\cong  II_{2,18}\oplus  A_1\oplus  A_1\cong  \Lambda_{20}.
\end{equation*}
In the above equation, perpendicularity is with respect to the Beauville-Bogomolov quadratic form, and the first isomorphism is found in~\cite{epwperiods}. The period space for degree-$2$ polarized   HK varieties of Type $K3^{[2]}$ is   
$\cF_{\Lambda_{20}}(\widetilde O^+(\Lambda_{20}))$;  in fact the moduli space of such (polarized) varieties is identified with an open dense subset of 
$\cF_{\Lambda_{20}}(\widetilde O^+(\Lambda_{20}))$  by Verbitsky's Global Torelli Theorem. (One should introduce the analogue of DuVal singularities, in order to identify $\cF_{\Lambda_{20}}(\widetilde O^+(\Lambda_{20}))$  with the moduli space of polarized varieties with mild singularities.) On the other hand we have a natural degree-$2$ covering map 
\begin{equation}\label{cyper}
\cF_{\Lambda_{20}}(\wt{O}^{+}(\Lambda_{20}))\overset{\rho}{\lra}\cF(20),
\end{equation}
because  $\Gamma_{\xi_{20}}$ is an index-$2$ subgroup of $\widetilde O^+(\Lambda_{20})$. (See Item~(4) of~\Ref{prp}{propgroup}.) 

There is an analogue of the GIT moduli spaces $\gM(19)$ and $\gM(18)$ in this case as well. 
In fact, let  $(X,L)$ be a generic  degree-$2$ polarized   HK varieties of Type $K3^{[2]}$.
Then the map $\varphi_L\colon X\dra |L|^{\vee}\cong\PP^5$ is regular, finite, of degree $2$ onto a special sextic hypersurface, an \emph{EPW sextic}. Conversely, given an EPW sextic $Y\subset\PP^5$, there is a canonical double cover  $f\colon X\to Y$, and if $Y$ is generic then $(X,f^{*}\cO_Y(1))$ is a degree-$2$ polarized   HK variety of Type $K3^{[2]}$.  (For the definition and properties of double EPW sextics we refer to~\cite{epwduke}  and~\cite{epw}.)  The parameter space for EPW sextics is (an open dense subset of) $\lagr$. Let 
\begin{equation}\label{emme20}
\gM(20)=\lagr\gquot\PGL_6(\CC).
\end{equation}
We have a birational period map
\begin{equation}\label{perepw}
\wt{\gp}_{20}\colon\gM(20)\dra \cF_{\Lambda_{20}}(\wt{O}^{+}(\Lambda_{20}))^{*}.
\end{equation}
The dual of an EPW sextic is another EPW sextic, see~\cite{dualepw}, and hence we have a (regular) duality involution $\delta\colon\gM(20)\to\gM(20)$. 
This is the geometric version of the covering involution of the double cover $\rho$ in~\eqref{cyper}. 
 The upshot is that we have a birational period map
\begin{equation}\label{per20}
\gp_{20}\colon\gM(20)/\la\delta\ra\dra \cF(20)^{*}.
\end{equation}
Thus we may view $\cF(20)$ as the period space for double EPW sextics up to duality. Alternatively, $\cF(20)$ is the period space for the Calabi-Yau fourfolds obtained by blowing up a (generic) EPW sextic $Y$ along it singular locus (a smooth surface). 
 
 The inverse images by $\rho$ of the reflective Heegner divisors of $\cF(20)$ have appeared in~\cite{epwperiods}.  
 In fact
\begin{equation*}
\rho^{-1}H_n(20)={\mathbb S}_2^{\star},\quad \rho^{-1}H_h(20)={\mathbb S}_4,
\quad \rho^{-1}H_u(20)={\mathbb S}'_2\cup {\mathbb S}''_2.
\end{equation*}
For a geometric interpretation of periods in 
${\mathbb S}_2^{\star}$, ${\mathbb S}_4$, 
${\mathbb S}'_2$, and ${\mathbb S}''_2$, see Section~5 of~\cite{epwperiods}.  
\subsubsection{EPW cubes}
Let $(X,L)$ be a polarized   HK variety of Type $K3^{[3]}$, where $q(c_1(L))=4$ and $\divisore(c_1(L))=2$.
 Then 
\begin{equation*}
c_1(L)^{\bot}\cong  II_{2,18}\oplus  A_1\oplus  A_1\cong  \Lambda_{20}.
\end{equation*}
Moreover, one shows that the period space for polarized   HK varieties of Type $K3^{[3]}$ of this kind is   
$\cF(20)$. (We thank M.~Kapustka and G.~Mongardi for bringing this to our attention.)  A.~Iliev, G.~Kapustka, M.~Kapustka and K.~Ranetsad~\cite{ikkr} have proved that the generic such polarized HK is isomorphic to an EPW cube, a double cover of a codimension-$3$ degeneracy locus in $\Gr(3,\CC^6)$. The parameter space of EPW cubes is the same as that fordouble EPW sextics (but the EPW cubes parametrized by $A\in\lagr$ and the dual lagrangian $A^{\bot}\in\lagrdual$ are isomorphic). 
\subsubsection{Hyperk\"ahler varieties of Type OG10}
Let $X$ be  a   HK manifold of Type OG10. Then $H^2(X;\ZZ)$ equipped with the Beauville-Bogomolov quadratic form is isomorphic to $II_{3,19}\oplus A_2$, see~\cite{rapagnetta}. Let $a,b\in A_2$ be standard generators ($a^2=b^2=(a+b)^2=-2$), and $v\in II_{3,19}$ of square $2$.  Let 
$h:=(3v+a-b)$; notice that $h^2=12$, and $(h,H^2(X;\ZZ))=3\ZZ$. 
The  discriminant group and quadratic form of  $h^{\bot}$ 
are isomorphic to the discriminant group and quadratic form of $\Lambda_{21}$ respectively, and hence $h^{\bot}\cong \Lambda_{21}$ by Theorem 1.13.2 of~\cite{nikulin}. Thus $\cD_{\Lambda_{21}}$ is the classifying space for the corresponding 
$10$-dimensional polarized O'Grady HK manifolds. By Verbitsky's Global Torelli Thorem~\cite{verbitskytor}, the moduli space  for such  varieties is isomorphic to $\cF_{\Lambda_{21}}(\Gamma)$, where $\Gamma<O^{+}(\Lambda_{21})$ is the relevant monodromy group (a subgroup of finite index, see~\cite{verbitskytor}).
See~\cite{mongardimon} for a result regarding the  monodromy group of O'Grady's  $10$-dimensional   HK manifolds.

\section{The Picard group of the $D$ tower}\label{sec:picdtower}

\subsection{Set up, and statement of the main results}
\setcounter{equation}{0}
Let $N\ge 3$. The Picard group of $\cF_{\Lambda_N}(\wt{O}(\Lambda_N)^+)$ is finitely generated, see~\cite{ghs-abelianisation}; 
  in the present section we will compute its rank  by applying results of Bergeron et al.~\cite{pick3nl}, and  of Bruinier~\cite{bruinier}.  
\begin{theorem}\label{thm:thmrankpic}
The rank of $\Pic(\cF_{\Lambda_N}(\wt{O}(\Lambda_N)^+))$ is equal to
\begin{equation*}
\rho_N:=\left\{
\renewcommand*{\arraystretch}{1.4}
\begin{array}{lcl}
\left\lfloor \frac{N+6}{8}\right\rfloor&\textrm{if}&N\equiv 1 \pmod 2\\
\left\lfloor \frac{N+2}{6}\right\rfloor&\textrm{if}&N\equiv 0,  6 \pmod 8\\
\left\lfloor \frac{N-4}{6}\right\rfloor&\textrm{if}&N\equiv 2 \pmod 8\\
\left\lfloor \frac{N+8}{6}\right\rfloor&\textrm{if}&N\equiv 4 \pmod 8
\end{array}
\right.
\end{equation*}
Explicitly,  for $3\le N\le 20$, we obtain:

\medskip

\begin{tabular}{c||c|c|c|c|c|c|c|c|c|c}
$N$&3&4&5-10&11&12&13-15&16&17-18&19&20  \\
\hline
$\rho_N$&1&2&1&2&3&2&3&2&3&4
\end{tabular}
\end{theorem}
\Ref{thm}{thmrankpic} will be proved in~\Ref{subsec}{eccopic}.
 If $N$ is odd, the natural map
\begin{equation}\label{crazyeddie}
\cF_{\Lambda_N}(\wt{O}(\Lambda_N)^+)\to \cF_{\Lambda_N}(\Gamma_{\xi_N})=\cF(N)
\end{equation}
is an isomorphism, and hence $\rho_N$ is also the  Picard number of $\cF(N)$. On the other hand, if $N$ is even the map in~\eqref{crazyeddie} 
  is a double cover (ramified along the hyperelliptic divisor, see~\Ref{crl}{otildegamma}), and hence $\Pic(\cF(N))_{\QQ}$ is identified with the  subspace  of 
  $\Pic(\cF_{\Lambda_N}(\wt{O}(\Lambda_N)^+))_{\QQ}$ invariant under the action of the covering involution of  the map in~\eqref{crazyeddie}. In particular, the 
  Picard number of  $\cF(N)$ is at most $\rho_N$. In fact, for most $N$ (with at most 1-2 exceptions), one can see that  the Picard number of $\cF(N)$ is equal to $\left\lfloor \frac{N+5}{8}\right\rfloor$.
Let  $N\in\{18,19,20\}$;  in~\Ref{sec}{whydelta} we will show that a basis of the $\QQ$-Picard group  of $\cF_{\Lambda_N}(\wt{O}(\Lambda_N)^+)$  is provided by the Heegner divisors defined in~\Ref{subsubsec}{heegorto}, and that the corresponding Heegner divisors of $\cF(N)$ provide a basis of $\cF(N)_{\QQ}$. In particular the ranks of $\Pic(\cF(18))$ and $\Pic(\cF(19))$ are equal to $\rho_{18}$ and $\rho_{19}$ respectively, while the rank of  $\Pic(\cF(20))$ is $3=(\rho_{20}-1)$.  
  
The fact that the Picard number of  $\cF(10)$ is equal to $1$ is 
particularly relevant for us. It implies that the $\QQ$-Picard group of $\cF(10)$ is generated by $\lambda(10)$, and thus $H_h(10)$ is  proportional to $\lambda(10)$. In fact Gritsenko~\cite[Thm. 3.2]{gritsenko} proved that  $H_h(10)=8 \lambda(10)$.  That is a key relation for what follows. More generally, we will need to know what are the relations between $\lambda(N)$, $H_n(N)$, $H_h(N)$ and $H_u(N)$. We will derive such 
 relations  in the range $N\le 25$ by considering suitable quasi pull-backs of Borcherd's celebrated reflective modular form for the orthogonal group $O(II_{2,26})$.  
 
 In order to state our results,  we let $\mu(N)$, for $3\le N\le 25$, be defined by Table~\eqref{tablemuen}.
  \begin{table}[htb!]
\renewcommand{\arraystretch}{1.60}
\begin{tabular}{c|c|c|c|c|c|c|c|c|c|c|c|c|c}
N &  3 &  4 & 5--10&11  & 12  & 13--18   &  19  &  20   & 21 &  22  &  23  &    24  &  25    \\
\hline
$\mu(N)$ & 46 & 1 & 0 &30& 1  & 0 & 78  & 33 & 16   &  8   &  4   &  2  &  1  \\\end{tabular}
\vspace{0.2cm}
\caption{Definition of $\mu(N)$}\label{tablemuen}
\end{table}
\begin{theorem}[First Borcherds' relation]\label{thm:thmborcherds1}
Let  $3\le N\le 25$. Then
\begin{equation}\label{borcherds1}
%
2(12+(26-N)(25-N))\lambda(N)=  H_n(N)+2(26-N) H_h(N)+\tau(N)\mu(N) H_u(N), 
\end{equation}
where $\tau(N)$ is   equal to $1$ if $N\equiv 3,4\pmod{8}$, and is equal to $2$ otherwise.
\end{theorem}
\begin{theorem}[Second Borcherds' relation]\label{thm:thmborcherds2}
 Let $3\le N\le 17$. Then
\begin{equation}\label{borcherds2}
%
2(132+(18-N)(17-N))\lambda(N)=  H_n(N)+2(18-N) H_h(N)+\tau(N)\mu(N+8) H_u(N), 
\end{equation}
where $\tau(N)$ is  as in~\Ref{thm}{thmborcherds1}.
\end{theorem}
\Ref{thm}{thmborcherds1} and~\Ref{thm}{thmborcherds2} will be proved in~\Ref{subsec}{sectborcherds}. Below we prove a corollary which is extremly important for our work. 
\begin{corollary}[Gritsenko's relation]\label{crl:keyrelations}
 Let $3\le N\le 17$. Then
 \begin{equation}\label{borcherds3}
32(14-N)\lambda(N)=16 H_h(N)+\tau(N)(\mu(N)-\mu(N+8))H_u(N).
\end{equation}
\end{corollary}
\begin{proof}
Taking the difference between the identities \eqref{borcherds1} and \eqref{borcherds2}, we see that the terms $H_n(N)$ cancel, giving the relationship above between $\lambda(N)$, $H_h(N)$ and $H_u(N)$. 
\end{proof}
\begin{corollary}\label{cormoving}
 The following hold:
\begin{enumerate}
\item $H_h(14)$ is linearly equivalent to $H_u(14)$.
\item If $4\le N<14$, then $H_h(N)$ is a big divisor.
\item If $4\le N\le 10$, then $H_h(N)$ is an ample divisor.
\end{enumerate}
\end{corollary}
\begin{proof}
This is a direct consequence of~\eqref{borcherds3}. Plug $N=14$ into~\eqref{borcherds3} to get Item~(1). Next, for $ N< 14$, we get
\begin{equation}\label{gritenne}
H_h(N)=2(14-N)\lambda(N)+\frac{\tau(N)}{16}(\mu(N+8)-\mu(N))H_u(N).
\end{equation}
Since $(14-N)>0$, $\lambda(N)$ is ample, and $\mu(N)\le \mu(N+8)$ (for $4\le N<14$), we get that $H_h(N)$ is big. For $4\le N\le 10$, the coefficient of $H_u(N)$ is zero, and hence  $H_h(N)$ is ample.
\end{proof}
\begin{remark}\label{remgritsenko}
Let $N\in\{4,\dots, 10\}$. Then Gritsenko \cite[Thm. 3.2]{gritsenko} constructed a strongly reflective automorphic form for the group $\widetilde O^+(\Lambda_N)$, of weight $(14-N)$, 
and with associated divisor  the hyperelliptic divisor $H_{\xi}(\Lambda_N)$, see~\Ref{dfn}{heegort}.
 Interpreted in our context, this reads
\begin{equation}\label{eqgritsenko}
H_h(N)=2(14-N)\lambda(N),
\end{equation}
which is exactly \Ref{crl}{keyrelations} for $N\in\{4,\dots, 10\}$.
\end{remark}
\subsection{The Picard rank of $\cF_{\Lambda_N}(\wt{O}(\Lambda_N)^+)$}\label{subsec:eccopic}
\setcounter{equation}{0}
The work of Borcherds and the refinements of Bruinier \cite{bruinier,bruinierbook} give a recipe for computing the rank of the Picard group for modular varieties of Type IV, by relating this to a dimension computation for vector valued modular forms. 
\medskip

\n
{\it Proof of~\Ref{thm}{thmrankpic}.\/} 
Let $L$ be an even lattice of signature $(2,N)$. Let $k=1+\frac{N}{2}\in \frac{1}{2}\bZ$. Borcherds has defined a homorphism
$$S_{k,L}\to \Pic(\sF_L(\widetilde O^+(L)))_\bC/\langle \lambda\rangle$$
from the space of cusp forms of weight $k$ with values in $L$. This morphism is injective if $L$ contains $2$ hyperbolic summands (cf. Bruinier). Recent work of Bergeron et al. \cite{pick3nl} establishes that this is in fact an isomorphism. We are interested in the case $L= \Lambda_N$ is a $D$ lattice. Thus~\Ref{thm}{thmrankpic} will follow from the computation of $\dim S_{k,\Lambda_N}$ carried out  below (see 
 \eqref{dimpicodd} and \eqref{dimpiceven}).

Bruinier \cite[p. 52, Eqs. (6) and (7)]{bruinier} gives the following formula for the space of cusp forms:
\begin{equation}\label{dimskl}
\dim(S_{k,\Lambda_N})=d+\frac{dk}{12}-\alpha_1-\alpha_2-\alpha_3-\alpha_4,
\end{equation}
where $k=1+\frac{N}{2}$ is as above, 
$$d:=\left|A_{\Lambda_N}/\{\pm 1\}\right|=\left\{ \begin{matrix}4&\textrm{if}& N\equiv 0 \pmod 2\\
 3&\textrm{if}& N\equiv 1 \pmod 2,
 \end{matrix}\right. $$
 and the $\alpha_i$'s will be computed  below.

In fact \cite[Eq. (7)]{bruinier}, gives
\begin{equation}\label{computea4}
\alpha_4:=\left|\left \{\gamma\in A_{\Lambda_N}/\{\pm 1\} \mid \frac{1}{2}q(\gamma)\in \bZ\right\}\right|=
\left\{ \begin{matrix}1&\textrm{if}& N\not\equiv 2 \pmod 8\\
 3&\textrm{if}& N\equiv 2 \pmod 8
  \end{matrix}\right.
  \end{equation}
 (Warning: Bruinier uses a different convention on the quadratic form $q$: for us, $q(x)=(x,x)\pmod{2\ZZ}$, while for Bruinier  $q(x)=(x,x)/2\pmod{\ZZ}$.)
  
Similarly, \cite[p. 55]{bruinier} gives
$$\alpha_3=\sum_{\gamma\in A_{\Lambda_N}/\{\pm1\}}\left\{-\frac{1}{2}q(\gamma)\right\},$$
where $\{z\}=z-\lfloor z\rfloor\in[0,1)$ denotes the fractional part of $z\in\bR$. It follows that
\begin{equation}\label{computea3}
\alpha_3=\left\{ 
\renewcommand*{\arraystretch}{1.4}
\begin{array}{lcl}\frac{1}{2}+\left\{\frac{N-2}{8}\right\}&\textrm{if}& N\equiv 1 \pmod 2\\
\frac{1}{2}+2\left\{\frac{N-2}{8}\right\}&\textrm{if}& N\equiv 0 \pmod 2.
  \end{array}\right. 
  \end{equation}
The computation of  $\alpha_1$ and $\alpha_2$ involves the Gauss sum
$$G(n,\Lambda_N):=\sum_{\gamma\in A_{\Lambda_N}} \exp\left(\frac{nq(\gamma)}{2}\right),$$
where $\exp(z)=e^{2\pi i z}$. The two sums needed later are 
\begin{equation}\label{eqg2}
\mathrm{Re}\left( G(2,\Lambda_N)\right)=\left\{\begin{matrix}2&\textrm{if}& N\equiv 1 \pmod 2\\
0&\textrm{if}& N\equiv 0 \pmod 4\\
4&\textrm{if}& N\equiv 2 \pmod 4,
\end{matrix}\right.
\end{equation}
and 
\begin{equation}\label{eqg13}
G(1,\Lambda_N)+G(-3,\Lambda_N)=\left\{\begin{matrix} 0 &\textrm{if}& N\equiv 1 \pmod 2\\
 4i e^{\frac{3N}{4} \pi i}&\textrm{if}& N\equiv 0 \pmod 2.
 \end{matrix}
 \right.
 \end{equation}
For $\alpha_1$, we have the following formula (cf. \cite[Eq. (12)]{bruinier}):
$$\alpha_1=\frac{d}{4}-\frac{1}{4\sqrt{|A_{\Lambda_N}|}}\cdot \exp\left(\frac{2k+2-N}{8}\right)\mathrm{Re}\left(G(2,\Lambda_N)\right)$$
Clealry, $|A_{\Lambda_N}|=4$,  and $2k+2-N=4$. Using \eqref{eqg2}, we get
\begin{equation}\label{computea1}
\alpha_1=\frac{d}{4}+\frac{1}{8}\mathrm{Re}(G(2,\Lambda_N))=\left\{
\begin{matrix}
1&\textrm{if}&N\equiv 1 \pmod 2\\
1&\textrm{if}&N\equiv 0 \pmod 4\\
\frac{3}{2}&\textrm{if}&N\equiv 2 \pmod 4.
\end{matrix}
\right. 
\end{equation}
Similarly, for $\alpha_2$, the following holds (cf. \cite[Eq. (13)]{bruinier}):
$$\alpha_2=\frac{d}{3}+\frac{1}{3\sqrt{3 |A_{\Lambda_N}|}} \mathrm{Re}\left(\exp\left(\frac{4k-4-3N}{24}\right)\cdot
\left(G(1,\Lambda_N)+G(-3,\Lambda_N)\right)\right).$$
Notice that $4k-4-3N=-N$. Applying \eqref{eqg13}, we obtain
\begin{equation}\label{computea2}
\alpha_2=\left\{
\renewcommand*{\arraystretch}{1.3}
\begin{matrix}
1&\textrm{if}&N\equiv 1 \pmod 2\\
\frac{4}{3}&\textrm{if}&N\equiv 0 \pmod 6\\
\frac{5}{3}&\textrm{if}&N\equiv 2 \pmod 6\\
1&\textrm{if}&N\equiv 4 \pmod 6
\end{matrix}
\right. .
\end{equation}
Equivalently, for $N$ odd $\alpha_2=1$, and for $N$ even 
\begin{equation}\label{computea2even}
\alpha_2=1+\left\{\frac{N+2}{6}\right\}.
\end{equation}

Using~\eqref{computea4}, \eqref{computea3}, \eqref{computea1}, and~\eqref{computea2}, we conclude the computation of $\dim S_{k,\Lambda}$ as follows. If $N$ is odd, then $d=3$, $\alpha_1=\alpha_2=\alpha_4=1$, and $\alpha_2=\frac{1}{2}+\left\{\frac{N-2}{8}\right\}$. Thus, for $N$ odd~\eqref{dimskl} reads
\begin{equation}\label{dimpicodd}
\dim S_{k,\Lambda_N}=\frac{N+2}{8}-\left(\frac{1}{2}+\left\{\frac{N-2}{8}\right\}\right)=\left\lfloor \frac{N-2}{8}\right\rfloor. 
\end{equation}
For $N$ even, $d=4$, and thus \eqref{dimskl} reads
$$\dim S_{k,\Lambda_N}=4+\frac{N+2}{6}-\alpha_1-\alpha_2-\alpha_3-\alpha_4.$$
From \eqref{computea2even}, we get
$$\frac{N+2}{6}-\alpha_2=\frac{N+2}{6}- \left(1+\left\{\frac{N+2}{6}\right\}\right)=
\left\lfloor \frac{N-4}{6}\right\rfloor.$$
Next, we note that $\alpha_4\in\bZ$, and $\alpha_1,\alpha_3\in \frac{1}{2}\bZ$, but their sum is an integer. Specifically, for even $N$ we have (cf. \eqref{computea1}, \eqref{computea3}): 
$$\alpha_1=\left\{
\begin{matrix}
1&\textrm{if}&N\equiv 0 \pmod 4\\
\frac{3}{2}&\textrm{if}&N\equiv 2 \pmod 4,
\end{matrix}
\right.$$
and 
$$\alpha_3=\frac{1}{2}+2\left\{\frac{N-2}{8}\right\}.$$
Using also \eqref{computea4}, we get
$$\alpha_1+\alpha_3+\alpha_4=\left\{\begin{matrix}
3&\textrm{if}&N\equiv 4 \pmod 8 \\
4&\textrm{if}&N\equiv 0,6 \pmod 8\\
5&\textrm{if}&N\equiv 2 \pmod 8.
\end{matrix}
\right.$$
In conclusion, 
\begin{equation}\label{dimpiceven}
\dim S_{k,\Lambda_N}=\left\lfloor \frac{N-4}{6}\right\rfloor+\left\{\begin{matrix}
-1&\textrm{if}&N\equiv 2 \pmod 8
0&\textrm{if}&N\equiv 0,6 \pmod 8\\
1&\textrm{if}&N\equiv 4 \pmod 8
\end{matrix}
\right.
\end{equation}
\qed
\subsection{Proof of Borcherds' relations}\label{subsec:sectborcherds}
\setcounter{equation}{0}
\subsubsection{Borcherds' automorphic form} 
We recall that Borcherds \cite{borcherds} constructed an automorphic form $\Phi_{12}$ on $\cD^{+}_{\II_{2,26}}$ for the orthogonal group $O^+(\II_{2,26})$, of weight $12$, whose zero-locus is the union 
 the nodal hyperplanes (the intersections $\delta^{\bot}\cap \cD^{+}_{\II_{2,26}}$, for $\delta$ a root of $\II_{2,26}$). Actually $\Phi_{12}$ vanishes with order one on each nodal hyperplane, i.e.
 \begin{equation}\label{scanzi}
\divisore(\Phi_{12})=\sum_{\pm\delta\in R(\II_{2,26})}\delta^{\bot}\cap \cD^{+}_{\II_{2,26}}=  \cH_{\delta_0,\II_{2,26}}(O^+(\II_{2,26})),
\end{equation}
where (as usual) $R(\II_{2,26})$ is the set of roots of $\II_{2,26}$, and $\delta_0$ is a chosen root of $\II_{2,26}$. In other words the divisor of $\Phi_{12}$ is the pre-Heegner divisor associated to a root of $\II_{2,26}$. 
\begin{remark}\label{rmk:automrelation}
Let $\Lambda$ be a lattice of signature $(2,m)$, let $\Gamma< O^{+}(\Lambda)$ be a subgroup of finite index, and let $\Phi$ be a $\Gamma$-automorphic form on $\cD^{+}_{\Lambda}$ of weight $w$. Then $\Phi$ descends  to a regular section of $\cL(\Lambda,\Gamma)^{\otimes w}$, see~\Ref{subsubsec}{divfour}. Thus~\eqref{scanzi} gives that, on the locally symmetric variety $O^+(\II_{2,26})\backslash\cD^{+}_{\II_{2,26}}$, one has the relation $12\lambda(\II_{2,26})=H_n(\II_{2,26})$, where $\lambda(\II_{2,26})$ is the first Chern class of the automorphic (or Hodge) line-bundle, and $H_n(\II_{2,26})$ is the nodal Heegner divisor.  
\end{remark}
Now suppose that  $\Lambda$ is a lattice of signature $(2,n)$ and we are given a saturated embedding $\Lambda\subset\II_{2,26}$.  The \emph{quasi-pullback} of $\Phi_{12}$ is defined as 
\begin{equation*}
\Phi_\Lambda:=\frac{\Phi_{12} }{\prod\limits_{\pm \delta\in R(\Lambda^{\perp})} \ell_\delta}\Biggr\rvert_{\cD^{+}_{\Lambda}},
\end{equation*}
where $\ell_\delta$ is the restriction to $\cD^{+}_{\Lambda}$ of the linear form $\sigma\mapsto (\delta,\sigma)$. Then $\Phi_{\Lambda}$ is an automorphic form on $\cD^{+}_{\Lambda}$ for the \emph{stable} orthogonal group $\wt{O}^+(\Lambda)$, see~\cite{bkpsb} and~\cite[\S8]{ghs}.
Notice that our notation is somewhat imprecise: the  automorphic form $\Phi_{\Lambda}$ depends on the embedding of $\Lambda$ into $\II_{2,26}$, we will see  instances of this later on. The  weight  and divisor of $\Phi_{\Lambda}$ are computed as follows. 
\begin{recipe}\label{rcp:pesodiv}
The  weight of $\Phi_{\Lambda}$ is equal to $12+w$, where $w=|R(\Lambda^{\bot})|/2$ is half the number of roots of $\II_{2,26}$ perpendicular to $\Lambda$.  The divisor of $\Phi_{\Lambda}$ is supported on the union of 
the intersections $\delta^{\bot}\cap\cD^{+}_{\Lambda}$, for $\delta\in R(\II_{2,26})\setminus R(\Lambda^{\bot})$. More precisely, 
\begin{equation}\label{divogiulio}
\divisore(\Phi_{\Lambda})=\sum_{\pm\delta\in R(\II_{2,26})\setminus R(\Lambda^{\bot})}(\delta^{\bot}\cap\cD^{+}_{\Lambda}).
\end{equation}
\end{recipe}
\begin{remark}\label{rmk:pesomolt}
Let $\delta\in R(\II_{2,26})\setminus R(\Lambda^{\bot})$; then $\delta^{\bot}\cap\cD^{+}_{\Lambda}$ is non empty if and only if $\la\delta,\Lambda^{\bot}\ra$ is negative definite. Given such a $\delta$, let $\nu(\delta)$ be a generator of $(\QQ\delta\oplus\QQ\Lambda^{\bot})\cap\Lambda$ (thus $\nu(\delta)$ is determined up to multiplication by $\pm 1$). Then 
\begin{equation*}
\delta^{\bot}\cap\cD^{+}_{\Lambda}=\nu(\delta)^{\bot}\cap\cD^{+}_{\Lambda}.
\end{equation*}
Let $\Sat\la\delta,\Lambda^{\bot}\ra$
be the \emph{saturation} of $\la\delta,\Lambda^{\bot}\ra$ in $\II_{2,26}$.
If $\delta'\in \Sat\la\delta,\Lambda^{\bot}\ra$  is another root which does not belong to $\Lambda^{\bot}$, then $\nu(\delta')=\pm\nu(\delta)$. The upshot is that we may rewrite the right hand side of~\eqref{divogiulio} as a finite sum of pre-Heegner divisors $\cH_{\nu(\delta_i)}$, where the coefficient of $\cH_{\nu(\delta_i)}$ 
 is equal to half  the  number of the roots of $\Sat\la\delta_i,\Lambda^{\bot}\ra$ which do not belong to $\Lambda^{\bot}$ (call this number $m(\delta_i)$):
 \begin{equation}\label{sommafinita}
\divisore(\Phi_{\Lambda})=m(\delta_1)\cH_{\nu(\delta_i),\Lambda}(\wt{O}^{+}(\Lambda))+\ldots+m(\delta_s)\cH_{\nu(\delta_s),\Lambda}(\wt{O}^{+}(\Lambda)).
\end{equation}
Lastly, Equation~\eqref{sommafinita} descends to a relation between the Hodge bundle on $\wt{O}^{+}(\Lambda)\backslash\cD^{+}_{\Lambda}$ and the Heegner divisors corresponding to the vectors 
$\nu(\delta_s)$, see~\Ref{rmk}{automrelation}.
\end{remark}
The plan is the following. We will choose embeddings of $\Lambda_N$ (for $3\le N\le 25$) into $\II_{2,26}$ such that the pre-Heegner divisors appearing on the right hand side of~\eqref{sommafinita} are associated to minimal norm vectors (see~\Ref{dfn}{vetmin}). For a given embedding, Equation~\eqref{sommafinita} descends to a relation between the Hodge bundle and the  Heegner divisors of $\cF_{\Lambda_N}(\wt{O}^{+}(\Lambda_N))$  associated to minimal norm vectors. For $N$ odd $\cF(N)=\cF_{\Lambda_N}(\wt{O}^{+}(\Lambda_N))$, and the relations that we will get are those of~\Ref{thm}{thmborcherds1} and~\Ref{thm}{thmborcherds2}. If $N$ is even, we will get the relations in~\Ref{thm}{thmborcherds1} and~\Ref{thm}{thmborcherds2} by pushing forward via the double covering $\cF_{\Lambda_N}(\wt{O}^{+}(\Lambda_N))\to\cF(N)$.
\subsubsection{Embeddings of $\Lambda_N$ into $\II_{2,26}$}
\begin{lemma}\label{lmm:lemextendd}
If $3\le N\le 25$, the lattice $\Lambda_N$ has a saturated embedding into the unimodular lattice $\II_{2,26}$ with orthogonal complement isomorphic to $D_{26-N}$, and if $3\le N\le 17$ it also  has a saturated embedding  with orthogonal complement isomorphic to $E_8\oplus D_{18-N}$. Conversely, any saturated embedding of $D_{26-N}$, or of $E_8\oplus D_{18-N}$, into $\II_{2,26}$ has orthogonal complement isomorphic to $\Lambda_N$. 
\end{lemma}
\begin{proof} 
If $3\le N\le 25$, then by~\eqref{discdien} and~\Ref{clm}{disdin} there exists an isomorphism of groups $\varphi\colon A_{\Lambda_N}\overset{\sim}{\lra} D_{26-N}$ which multiplies  the discriminant quadratic form by $-1$, i.e.
\begin{equation}\label{antisometria}
q_{D_{26-N}}(\varphi(\eta))=-q_{\Lambda_N}(\eta).
\end{equation}
Let $L\subset(\Lambda_N\oplus D_{26-N})_{\QQ}$ be the overlattice of $\Lambda_N\oplus D_{26-N}$ generated by vectors $(v,w)$ such $[v]\in A_{\Lambda_N}$, 
 $[w]\in A_{D_{26-N}}$, and $[w]=\varphi([v])$. Then, because of~\eqref{antisometria}, $L$ is even, unimodular, of signature $(2,26)$, and hence is isomorphic to $\II_{2,26}$. By construction, 
 $\Lambda_N$ is a saturated  sublattice of $L$, and its orthogonal complement is isomorphic  to $D_{26-N}$. 
 If $3\le N\le 17$, there exists an isomorphism of groups 
 $\psi\colon A_{\Lambda_N}\overset{\sim}{\lra} (E_8\oplus D_{18-N})$
 which multiplies  the discriminant quadratic forms by $-1$, i.e.
\begin{equation*}
 q_{E_8\oplus D_{18-N}}(\psi(\eta))=-q_{\Lambda_N}(\eta).
\end{equation*}
Proceeding as in the previous case, one constructs an overlattice of $\Lambda_N\oplus (E_8\oplus D_{18-N})$ which is isomorphic to $\II_{2,26}$, in which  
 $\Lambda_N$ is saturated, with orthogonal complement isomorphic to $E_8\oplus D_{18-N}$. 

Let us prove the last statement of the lemma. Suppose that  $D_{26-N}\subset\II_{2,26}$, or  $(E_8\oplus D_{18-N})\subset\II_{2,26}$, is saturated. 
Let $M:=D^{\bot}_{26-N}$ (respectively $M=(E_8\oplus D_{18-N})^{\bot}$)  The overlattice $\II_{2,26}\supset(D_{26-N}\oplus M)$ (respectively 
  $\II_{2,26}\supset((E_8\oplus D_{18-N})\oplus M)$ induces an isomorphism of groups
\begin{equation}\label{isogi}
g\colon A_{D_{26-N}}\overset{\sim}{\lra} A_M\quad\text{(respectively 
$g\colon A_{E_8\oplus D_{18-N}}\overset{\sim}{\lra} A_M$).}
\end{equation}
such that $q_{M}(g(\eta))=-q_{D_{26-N}}(\eta)$ for all $\eta\in  A_{D_{26-N}}$ (respectively, $q_{M}(g(\eta))=-q_{E_8\oplus D_{18-N}}(\eta)$ for all 
$\eta\in  A_{E_8\oplus D_{18-N}}$). 
Thus $M$ has the same signature, parity and discriminant quadratic form as $\Lambda_N$, and hence   is isomorphic to $\Lambda_N$  by Theorem 1.13.2 of~\cite{nikulin}. 
\end{proof}
\begin{remark}\label{rmk:gluevectors}
Let $3\le N\le 25$. Suppose that $\Lambda_N\subset\II_{2,26}$ is a saturated embedding, and that $\Lambda_N^{\bot}$ is isomorphic either to $D_{26-N}$, or to $E_8\oplus D_{18-N}$ (in this case $N\le 17$). Let $(N-2)=8k+a$, where $k\ge 0$ and $a\in\{1,\ldots,8\}$. Then $\Lambda_N\cong \II_{2,2+8k}\oplus D_a$, and we  identify the two lattices throughout. 
\begin{enumerate}
\item
If $N\le 25$, and $\Lambda_N^{\bot}\cong D_{26-N}$, then
 $\II_{2,26}$ is identified with the subattice of 
\begin{equation*}
(\II_{2,2+8k}\oplus D_a\oplus D_{26-N})_{\QQ}
\end{equation*}
   generated by $\II_{2,2+8k}\oplus D_a\oplus D_{26-N}$ together with the following $3$ vectors:
\begin{equation*}
\scriptstyle
u_1:=(0_{4+8k},(\underbrace{\scriptstyle 1/2,\ldots,1/2}_{a}),(\underbrace{\scriptstyle 1/2,\ldots,1/2}_{26-N})),\quad
u_2:=(0_{4+8k},(\underbrace{\scriptstyle -1/2,1/2,\ldots,1/2}_{a}),(\underbrace{\scriptstyle -1/2,1/2,\ldots,1/2}_{26-N})),\quad 
u_3:=(0_{4+8k},(\underbrace{\scriptstyle 1,0,\ldots,0}_{a}),(\underbrace{\scriptstyle 1,0,\ldots,0}_{26-N})).
\end{equation*}
\item
If $N\le 17$, and $\Lambda_N^{\bot}\cong E_8\oplus D_{18-N}$, then
 $\II_{2,26}$ is identified with the sublattice of 
\begin{equation*}
(\II_{2,2+8k}\oplus D_a\oplus D_{18-N}\oplus E_8)_{\QQ}
\end{equation*}
  generated by $\II_{2,2+8k}\oplus D_a \oplus D_{18-N}\oplus E_8$ together with the following $3$ vectors:
\begin{equation*}
\scriptstyle
w_1:=(0_{4+8k},(\underbrace{\scriptstyle 1/2,\ldots,1/2}_{a}),(\underbrace{\scriptstyle 1/2,\ldots,1/2}_{18-N}),0_8)),\ 
w_2:=(0_{4+8k},(\underbrace{\scriptstyle -1/2,1/2,\ldots,1/2}_{a}),(\underbrace{\scriptstyle -1/2,1/2,\ldots,1/2}_{18-N}),0_8)),\  
w_3:=(0_{4+8k},(\underbrace{\scriptstyle 1,0,\ldots,0}_{a}),(\underbrace{\scriptstyle 1,0,\ldots,0}_{18-N}),0_8)).
\end{equation*}
\end{enumerate}
\end{remark}
\begin{remark}
Suppose that  $\Lambda_N\subset \II_{2,26}$ is a saturated embedding.  
Then the orthogonal complement of $\Lambda_N$  in $\II_{2,26}$ is a lattice in the genus of $D_{26-N}$. We note the following:
\begin{itemize}
\item \emph{if $1\le k\le 8$, there is only one isometry class of lattices in the genus of $D_k$, namely $D_k$},
\item \emph{if $9\le k\le 16$, there are at least two distinct classes, namely $D_{k}$ and $D_{k-8}\oplus E_8$}.
\end{itemize}
In other words, there is one Borcherds relation for $N\ge 18$, and at least two such relations for $10\le N\le 17$. For $N\le 17$, the relations relevant for us are those associated to $D_{k}$ and $D_{k-8}\oplus E_8$ as those involve only Heegner divisors associatde to minimal norm vectors. 
\end{remark}
\begin{remark}\label{rmk:emmequattro}
Let $m\ge 1$. The non-zero elements of $ A_{D_m}$ are $\xi,\zeta,\zeta'$, where
\begin{equation*}
\xi:=[\underbrace{(1,0,\ldots,0)}_{m}],\qquad \zeta:=[\underbrace{(1/2,\ldots,1/2)}_{m}],\qquad \zeta':=[\underbrace{(-1/2,1/2,\ldots,1/2)}_{m}].
\end{equation*}
Let $v_m:=(2,0,\ldots,0)\in D_m$. Then the divisibility of $v_m$ is even (equal to $2$, unless $m=1$), and $[v_m/2]=\xi$. Moreover if $u\in D_m$ 
has even divisibility, and $[u/2]=\xi$, then $u^2\le(-4)=v_m^2$. Similarly, let
\begin{equation*}
w_m:=
\begin{cases}
(1,\ldots,1)\in D_m & \text{if $m$ is even,} \\
(2,\ldots,2)\in D_m & \text{if $m$ is odd.}
\end{cases}
\end{equation*}
Then $w^{*}_m=\zeta$, and if $u\in D_m$ is such that $u^{*}=\zeta$, then 
\begin{equation*}
u^2\le w_m^2= 
\begin{cases}
-m=w_m^2 & \text{if $m$ is even,} \\
-4m=w_m^2  & \text{if $m$ is odd.}
\end{cases}
\end{equation*}
Since the map $(x_1,\ldots,x_m)\mapsto (-x_1,x_2,\ldots,x_m)$ defines an automorphism of $D_m$ interchanging $\zeta$ with $\zeta'$, an analogous result holds for the minimal absolute value of $u^2$ for vectors $u\in D_m$ such that $u^{*}=\zeta'$. 
 This fact has the following interesting consequence. Suppose that 
  $\Lambda\subset\II_{2,26}$ is a saturated sublattice, and that one of the following holds:
\begin{enumerate}
\item
$\Lambda^{\bot}\cong D_{26-N}$ and $N\in\{6,14\}$, or
\item
$\Lambda^{\bot}\cong (E_8\oplus D_{18-N})$ and $N=6$.
\end{enumerate}
Although there is no preferred decoration of the abstract dimension $N$ $D$-lattice $\Lambda$ (because  $q_{\Lambda}(\eta)\equiv 1\pmod{2\ZZ}$ for all non-zero $\zeta\in A_{\Lambda}$),  
 there is a preferred  decoration determined by the embedding $\Lambda\subset\II_{2,26}$, namely  $\eta:=g(\xi)$, where  $g$ is the isomorphism in~\eqref{isogi},
 and  $\xi\in A_{\Lambda^{\bot}}$ is the 
 unique class for which there exists $v\in\Lambda^{\bot}$ of square $-4$ and even divisibility such that $[v/2]=\xi$.
\end{remark}
\begin{definition}\label{dfn:decoammiss}
Let  $\Lambda\subset\II_{2,26}$ be a saturated sublattice such that $\Lambda^{\bot}$ is isomorphic 
to $D_{26-N}$ or to $(E_8\oplus D_{18-N})$ (hence $\Lambda$ is a dimension $N$ $D$-lattice by~\Ref{lmm}{lemextendd}).
A decoration $\eta$ of $\Lambda$ is \emph{admissible} if $\eta=g(\xi)$,  where  $g$ is the isomorphism in~\eqref{isogi},
 and  $\xi\in A_{\Lambda^{\bot}}$ is a class for which there exists $v\in\Lambda^{\bot}$ of square $-4$ and even divisibility such that $[v/2]=\xi$.
\end{definition}
\begin{remark}
If $N\not\equiv 6\pmod{8}$, the unique decoration of $\Lambda$ is admissible, if $N=22$ all three decorations of $\Lambda$ are admissible (any non-zero element of $D_4$ 
is equal to $[v/2]$ for a suitable $v\in\Lambda$ of square $-4$ and even divisibility), 
and if  $N\in\{6,14\}$, then only one of the  three decorations is admissible, see~\Ref{rmk}{emmequattro}.
\end{remark}
\subsubsection{The pre-Heegner divisors associated to the quasi pull-backs of $\Phi_{12}$}
The following result will allow us to determine the vectors $\nu(\delta)$  which appear  in~\Ref{rmk}{pesomolt}, for the two embeddings of 
$\Lambda_N$ into $\II_{2,26}$  given by~\Ref{lmm}{lemextendd}. 
\begin{proposition}\label{prp:numinorm}
Let $3\le N\le 25$. Assume that $\Lambda\subset \II_{2,26}$ is a saturated sub lattice, and that $\Lambda^{\bot}$  is isomorphic to $D_{26-N}$ or to  $E_8\oplus D_{18-N}$ (in the latter case, $N\le 17$), and hence  $\Lambda$ is a dimension $N$ $D$-lattice by~\Ref{lmm}{lemextendd}. Suppose that $\delta\in R(\II_{2,26}\setminus\Lambda^{\bot})$ is such that 
$\la\Lambda^{\bot},\delta\ra$ is negative definite. Then $\nu(\delta)$ (notation as in~\Ref{rmk}{pesomolt}) is a minimal norm vector of $\Lambda$ (see~\Ref{dfn}{vetmin}).
Moreover, one of the following holds:
\begin{enumerate}
\item[(a)] 
$\nu(\delta)^2=-2$.
\item[(b)]
$\nu(\delta)^2=-4$, $\divisore_{\Lambda}(\nu(\delta))=2$, and $\nu(\delta)^{*}$  is an admissible decoration  of $\Lambda$. 
\item[(c)]
$\nu(\delta)^2=-4$, and $\divisore_{\Lambda}(\nu(\delta))=4$. 
\item[(d)]
$\Lambda^{\bot}\cong D_{26-N}$, and $19\le N$.
\item[(e)]
$\Lambda^{\bot}\cong (E_8\oplus D_{18-N})$, and $11\le N$.
\end{enumerate}
\end{proposition}
\begin{proof}
Let $m$ be the minimum strictly positive integer such that $m\delta\in(\Lambda\oplus\Lambda^{\bot})$.    Thus 
\begin{equation*}
m\delta=v+w,\qquad 0\not=v\in\Lambda,\quad w\in\Lambda^{\bot},
\end{equation*}
and $m\in\{1,2,4\}$,  see~\Ref{rmk}{gluevectors}. Notice that
 $v\in\la \nu(\delta)\ra$ and $\nu(\delta)=\pm v$ if and only if $v$ is primitive.
 Let us prove that one of the following Items holds.
 \begin{enumerate}
\item
$m=1$, $v^2=-2$,  and $\nu(\delta)=\pm v$.
\item
$m=2$, $v^2\in\{-2,-6\}$, $\divisore_{\Lambda}(v)=2$, and $\nu(\delta)=\pm v$.
\item
$m=2$, $v^2=-4$, $\divisore_{\Lambda}(v)=2$,  $\nu(\delta)=\pm v$, and $\nu(\delta)^{*}$  is an admissible decoration  of $\Lambda$.
\item
$m=2$, $v^2=-4$, $\divisore_{\Lambda}(v)=4$, and $\nu(\delta)=\pm v$
\item
$m=4$,  $N$ is odd, $v^2=-4a$, where 
$a$ is the residue modulo $8$ of $(N-2)$, $\divisore_{\Lambda}(v)=4$, and $\nu(\delta)=\pm v$.
\end{enumerate}
  We have $\nu(\delta)^2<0$, because
 $\la\Lambda_N^{\bot},\delta\ra$   is negative definite, and hence $v^2<0$. Since $-2m^2=(m\delta)^2=v^2+w^2$, one of the following holds:
\begin{enumerate}
\item[(i)]
$m=1$, $v^2=-2$,  and $\nu(\delta)=\pm v$.
\item[(ii)]
$m=2$, $v^2\in\{-2,-4,-6\}$,  $\nu(\delta)=\pm v$, and $\divisore_{\Lambda}(v)$ is either $2$ or $4$. 
\item[(iii)]
$m=4$, $v^2\in\{-2,-4,-6,-8,\ldots,-30\}$, $\nu(\delta)=\pm v$, and $\divisore_{\Lambda}(v)=4$.
\item[(iv)]
$m=4$, $v^2\in\{-8,-16,-18,-24\}$, and $v$ is \emph{not} primitive.
\end{enumerate}
If~(i) holds, then Item~(1) holds. 

Suppose that~(ii) holds. If $v^2=-2$,  then  Item~(2) holds. 
 If $v^2=-6$, then by a discriminant quadratic form computation (see~\eqref{discdien} and~\Ref{clm}{disdin}) 
we get that $\divisore_{\Lambda_N}(v)=2$, and hence Item~(2) holds. 
If $v^2=-4$ and  $\divisore_{\Lambda_N}(v)=4$,  then  Item~(4) holds. Thus we are left with the case $v^2=-4$ and  $\divisore_{\Lambda_N}(v)=2$. Since $m=2$, the divisibility of $w$ in $\Lambda^{\bot}$ is even, and hence $[w/2]\in A_{\Lambda^{\bot}}$. Since $g([w/2])=[v/2]$, where $g$ is the isomorphism in~\eqref{isogi}, and $w^2=-4$, it follows that $[v/2]=\nu(\delta)^{*}$ is an admissible decoration. Thus Item~(3) holds. This finishes the proof that if~(ii) holds, then one of Items~(2), (3), (4) holds.

If~(iii) holds,  then, by looking at the discriminant quadratic form of $\Lambda_N$, we get that $v^2=-4a$, where 
$a\in\{1,3,5,7\}$, and $a\equiv(N-2)\pmod{8}$. Thus Item~(5) holds.

Lastly, suppose that~(iv) holds. We will arrive at a contradiction. First, let us show  that $w$ is primitive. If $v^2=-18$, this is clear because $w^2=-14$. If $v^2\in\{-8,-16,-24\}$, then   $v=2u$
where $u\in\Lambda^{\bot}$ (because by assumption $v$ is not primitive), and if  $w=r z$ with $r\ge 2$  
and  $z\in\Lambda^{\bot}$, then $r=2$  
because $w^2\in\{-24,-16,-8\}$. Then $2\delta=u+z$, contradicting our hypothesis.  This proves that  $w$ is primitive. Since the divisibility of $w$ in $\Lambda^{\bot}$  is a multiple of $4$, it follows that $\divisore_{\Lambda^{\bot}}(w)=4$, and hence $N$ is odd. Thus $w^2/16=q_{\Lambda^{\bot}}(w^{*})\equiv -(N-2)/4\pmod{2\ZZ}$, and hence $w^2=-4a$, where $a$ is odd. This is a contradiction, because $w^2\in\{-24,-16,-14,-8\}$. 

Now we finish the proof of the proposition. First, if either one of Items~(1)-(5) holds, $\nu(\delta)$ is a minimal norm vector. 

It remains to prove that if Item~(2) holds with $v^2=-6$, or if Item~(5) holds, then Item~(d) or Item~(e) of the Proposition holds.  We will assume that 
$\Lambda^{\bot}\cong D_{26-N}$, if $\Lambda^{\bot}\cong (E_8\oplus D_{18-N})$ is analogous.

Suppose that  Item~(2) holds with $v^2=-6$.
  Since $q_{\Lambda_N}(v^{*})\equiv -3/2\pmod{2\ZZ}$, we have $N\equiv 0\pmod{8}$; it follows that , $\Lambda^{\bot}\cong D_{26-N}=D_{2+8h}$ for 
  $h\in\{0,1,2\}$.
The divisibility of $w$ as element of $\Lambda^{\bot}$ is even, because $m=2$, and hence it is equal to $2$ (e.g.~because $w^2=-2$). 
Now, $q_{\Lambda_N^{\bot}}(w^{*})\equiv -1/2\pmod{2\ZZ}$, and hence $w^{*}\in\{\zeta,\zeta'\}$, where notation is as in~\Ref{rmk}{emmequattro}.  Since $w^2=-2$, it follows from~\Ref{rmk}{emmequattro} that $\Lambda^{\bot}\cong D_2$, i.e.~$N=24$.

 Lastly, suppose that Item~(5) holds. Then $w^2=-4(8-a)$, and  $\divisore_{\Lambda^{\bot}}(w)=4$. Thus $q_{\Lambda^{\bot}}(w^{*})\equiv -(8-a)/4\pmod{2\ZZ}$, and hence $w^{*}\in\{\zeta,\zeta'\}$,  notation  as in~\Ref{rmk}{emmequattro}. On the other hand, $\Lambda^{\bot}\cong D_{8-a+h}$ 
where $h\ge 0$.  Since $w^2=-(8-a)$, it follows from~\Ref{rmk}{emmequattro} that $h=0$,  
 i.e.~$N\ge 19$. 
\end{proof}
\subsubsection{Borcherds' automorphic forms for $\wt{O}^{+}(\Lambda_N)$}
\begin{theorem}\label{thm:automone}
Let  $3\le N\le 25$. Let $\Lambda_N\subset\II_{2,26}$ be a saturated embedding with orthogonal complement isomorphic to $D_{26-N}$. Let  $\xi\in A_{\Lambda_N}$ be an admissible decoration of $\Lambda_N$
 (see~\Ref{dfn}{decoammiss}), and let $\zeta,\zeta'$ be the remaining non-zero elements of $A_{\Lambda_N}$. 
Let $\Psi_N$ be the quasi-pullback of Borcherds' automorphic form $\Phi_{12}$. Then $\Psi_N$ has weight $(12+(26-N)(25-N))$, and 
\begin{equation}\label{automone}
\divisore(\Psi_N)=  \cH_0(\Lambda_N)+2(26-N) \cH_\xi(\Lambda_N)+\mu(N) (\cH_{\zeta}(\Lambda_N) +\cH_{\zeta'}(\Lambda_N)),
\end{equation}
where $\mu(N)$ is as in Table~\eqref{tablemuen}.
\end{theorem}
\begin{theorem}\label{thm:automdue}
 Let $3\le N\le 17$. Let $\Lambda_N\subset\II_{2,26}$ be a saturated embedding with orthogonal complement isomorphic to $E_8\oplus D_{18-N}$. Let  $\xi\in A_{\Lambda_N}$ be an admissible decoration of $\Lambda_N$
 (see~\Ref{dfn}{decoammiss}), and let $\zeta,\zeta'$ be the remaining non-zero elements of $A_{\Lambda_N}$. 
 Let  $\Xi_N$ be the quasi-pullback of Borcherds' automorphic form $\Phi_{12}$. Then $\Xi_N$ has weight $(132+(18-N)(17-N))$, and 
\begin{equation}\label{automdue}
\divisore(\Xi_N)=  \cH_0(\Lambda_N)+2(18-N) \cH_\xi(\Lambda_N)+\mu(N+8) (\cH_{\zeta}(\Lambda_N) +\cH_{\zeta'}(\Lambda_N)).
\end{equation}
\end{theorem}
Before proving the above results, we introduce some notation.
\begin{notation}[Extended Exceptional Series $E_r$]
It is standard (e.g. in the context of del Pezzo surfaces, cf. \cite{fmdelpezzo}) to extend the $E_r$ series for $r< 6$ as follows. Let $(1)\oplus(-1)^r$ be $\ZZ^{r+1}$ with the quadratic form $q(x,y_1,\ldots,y_r):=x^2-\sum_{i=1}^r y_i^2$. Then $E_r$ is the sublattice of $(1)\oplus(-1)^r$ defined by
\begin{equation}\label{delpezzoprim}
E_r:=(3,\underbrace{1,\ldots,1}_{r})^{\bot}\subset (1)\oplus(-1)^r.
\end{equation}
Of course this is the usual $E_r$ for $r\in\{6,7,8\}$, and
 $E_5=D_5$, $E_4=A_4$, $E_3=A_2\oplus A_1$. The lattice $E_2$ has Gram matrix $\left(\begin{matrix}-4&1\\1&-2\end{matrix} \right)$, and is no longer a root lattice. We record for later use the cardinalities of the sets of roots of the $E_r$ lattices:
 \begin{equation}\label{rootsofe}
\scriptstyle
|R(E_2)|=2,\quad |R(E_3)|=8,\quad |R(E_4)|=20,\quad |R(E_5)|=40,\quad |R(E_6)|=72,\quad |R(E_7)|=126,\quad |R(E_8)|=240.
\end{equation}
 \end{notation}
\noindent{\it Proof of~\Ref{thm}{automone}.\/}
The weight is equal to $12+|R(\Lambda_N^{\bot})|/2=12+(26-N)(25-N)$ by~\Ref{rcp}{pesodiv} (recall that $|R(D_m)|=2m(m-1)$). Next, choose minimal norm vectors $v_0,v_{\xi},v_{\zeta},v_{\zeta'}\in\Lambda_N$ representing $0,\xi,\zeta,\zeta'\in A_{\Lambda_N}$ respectively (i.e.~$v_0^{*}=0$, $v_{\xi}^{*}=\xi$, etc.). For $\eta\in 
A_{\Lambda_N}$, let
\begin{equation*}
%
a_{\eta}(N):=\frac{1}{2}(|R(\Sat\la v_{\eta},\Lambda_N^{\bot}\ra)|-| R(\Lambda_N^{\bot})|)
=\frac{1}{2}|R(\Sat\la v_{\eta},\Lambda_N^{\bot}\ra)|-\frac{1}{2}(26-N)(25-N). 
\end{equation*}
By~\Ref{rcp}{pesodiv}, Equation~\eqref{sommafinita}, and~\Ref{prp}{numinorm}, 
\begin{equation*}
\divisore(\Psi_N)= a_0(N) \cH_0(\Lambda_N)+a_{\xi}(N)\cH_\xi(\Lambda_N)+a_{\zeta}(N)\cH_{\zeta}(\Lambda_N) +a_{\zeta'}(N)\cH_{\zeta'}(\Lambda_N). 
\end{equation*}
(Of course, by~\Ref{prp}{numinorm}, $0=a_{\zeta}(N)=a_{\zeta'}(N)$  if $N\le 18$ and $N\not\in\{3,4,11,12\}$.)   It is clear that $\Sat\la  v_0,\Lambda_N^{\bot}\ra=\la  v_0,\Lambda_N^{\bot}\ra$, and hence $a_0(N)=1$. In order to compute $a_{\xi}(N)$, $a_{\zeta}(N)$, and $a_{\zeta'}(N)$, we will refer to the 
embedding $\Lambda_N\subset\II_{2,26}$ of~\Ref{rmk}{gluevectors}. Thus we let $(N-2)=8k+a$, where $k\ge 0$, $a\in\{1,\ldots,8\}$, and we identify $\Lambda_N$ with $\II_{2,2+8k}\oplus D_a$. 
Notice that 
\begin{equation}
\xi:=[(0_{4+8k},(\underbrace{1,0,\ldots,0}_{a}),0_{26-N})]\in A_{\Lambda_N}
\end{equation}
is an admissible decoration of $\Lambda_N$ with respect to  the  embedding of~\Ref{rmk}{gluevectors}.
We choose the minimal norm vector $v_\xi:=(0_{4+8k},(2,0,\ldots,0),0_{26-N})$ (if $N=3$, we let $v_\xi:=(2e,(2,0,\ldots,0),0_{26-N})$, where $e\in \II_{2,2+8k}$ is primitive and isotropic). Then the saturation of $\la v_{\xi},\Lambda_N^{\bot}\ra$ is generated  by 
$\la v_{\xi},\Lambda_N^{\bot}\ra$ and the vector $u_3$  of~\Ref{rmk}{gluevectors}  (if $N=3$, we add $(e,(1,0,\ldots,0),(1,0,\ldots,0))$). It follows that $\Sat\la v_{\xi},\Lambda_N^{\bot}\ra$ is isomorphic to $D_{27-N}$, and hence $a_{\xi}(N)=2(26-N)$. We choose  the minimal norm vector 
\begin{equation*}
v_{\zeta}:=
\begin{cases}
(0_{4+8k},(\underbrace{1,\ldots,1}_{a}),0_{26-N}) & \text{if $N$ is even,} \\
(0_{4+8k},(\underbrace{2,\ldots,2}_{a}),0_{26-N}) & \text{if $N$ is odd.}
\end{cases}
\end{equation*}
By~\Ref{prp}{numinorm} $a_{\zeta}(N)=0$, unless  $19\le N\le 25$, or $a\in\{1,2\}$ and $N\in\{3,4,11,12\}$. Assume first that $19\le N\le 25$. Then $a\in\{1,\ldots,7\}$, and $\Lambda_N^{\bot}=D_{26-N}=D_{8-a}$. The saturation of 
$\la v_{\zeta}, D_{8-a}\ra$ is generated by $\la v_{\zeta}, D_{8-a}\ra$ together with the vector $u_1$ of~\Ref{rmk}{gluevectors}; it follows easily that 
$\Sat\la v_{\zeta}, D_{8-a}\ra\cong E_{9-a}$. Noting that $a=N-18$, we get that $\Sat\la v_{\zeta}, \Lambda_N^{\bot}\ra\cong E_{27-N}$. Looking at~\eqref{rootsofe}, we get 
 that $a_{\zeta}(N)=\mu(N)$. 
 
Next, suppose that  $a=1$, and hence $N\in\{3,11\}$. The saturation of 
$\la v_{\zeta}, D_{23-8k}\ra$ is the overlattice $D^{+}_{(3-k)8}$ of $D_{(3-k)8}$ obtained by adjoining the vector $(1/2,\ldots,1/2)$. Then  $a_{\zeta}(N)=\mu(N)$ follows from the equality $|R(D_{(3-k)8}^{+}|=|R(D_{(3-k)8}|=2\cdot (24-8k)\cdot (23-8k)$, valid for $k\in\{0,1\}$.

Next, suppose that  $a=2$, and hence $N\in\{4,12\}$. The saturation of 
$\la v_{\zeta}, D_{22-8k}\ra$ is generated by $\la v_{\zeta}, D_{22-8k}\ra$ together with the vector $u_1$ of~\Ref{rmk}{gluevectors}; one verifies easily that the roots of $\Sat\la v_{\zeta}, D_{22-8k}\ra$ 
are exactly $\pm(0_{4+8k},(1,1),0_{26-N})$,  and hence $a_{\zeta}(N)=1=\mu(N)$. 

Lastly,  $a_{\zeta'}(N)=a_{\zeta}(N)$ because there is an automorphism of $\II_{2,26}$ mapping $\Lambda_N$ to itself, and exchanging $\zeta$ and $\zeta'$.
\qed

\medskip
\noindent{\it Proof of~\Ref{thm}{automdue}.\/}
The proof is analogous to that of~\Ref{thm}{automone}, we leave details to the reader.
\qed
\subsubsection{Borcherds' relations for divisors on $\cF_{\Lambda_N}(\wt{O}^{+}(\Lambda_N))$ and on $\cF(N)$}
We let $\lambda(\wt{O}^{+}(\Lambda_N))$ be the Hodge orbifold line-bundle on $\cF_{\Lambda_N}(\wt{O}^{+}(\Lambda_N))$.
\begin{lemma}\label{lmm:relstableone}
Let  $3\le N\le 25$. Let $\xi$ be a decoration of $\Lambda_N$, and let $\zeta,\zeta'\in A_{\Lambda_N}$ be the remaining non-zero elements. Then in $\Pic(\cF_{\Lambda_N}(\wt{O}^{+}(\Lambda_N)))_{\QQ}$ we have the relation
\begin{equation}\label{relstableone}
\scriptstyle
2(12+(26-N)(25-N))\lambda(\wt{O}^{+}(\Lambda_N))=  H_0(\Lambda_N)+\epsilon(N)2(26-N) H_\xi(\Lambda_N)+\tau(N)\mu(N) (H_{\zeta}(\Lambda_N) +H_{\zeta'}(\Lambda_N)), 
\end{equation}
where $\epsilon(N)$ is $1$ if $N$ is odd, and is equal to $2$ if $N$ is even, while $\tau(N)$ is equal to $1$ if $N\equiv 3,4\pmod{8}$, and is equal to $2$ otherwise.
\end{lemma}
\begin{proof}
Let $\Lambda_N\subset\II_{2,26}$ be a saturated embedding such that the decoration $\xi$ is admissible  (see~\Ref{dfn}{decoammiss}).
Let $\Psi_N$ be the automorphic form in~\Ref{thm}{automone}. If $k\in\NN_{+}$ is sufficiently divisible, then $\Psi^k_N$ descends to a regular section $\sigma_N(k)$ of 
$\lambda(\wt{O}^{+}(\Lambda_N))^{k(12+(26-N)(25-N))}$, whose zero divisor pulls-back to $k$ times the right-hand side of~\eqref{automone}. Taking into 
account~\Ref{clm}{reflheeg2}, one gets that $\divisore(\sigma_N(k))$ is equal to  $k$ times the right-hand side of~\eqref{relstableone}. Dividing by $k$, 
we get~\eqref{relstableone}.
\end{proof}
\begin{proposition}[=\Ref{thm}{thmborcherds1}]\label{prp:relstableone}
Let  $3\le N\le 25$. Then in $\Pic(\cF(N))_{\QQ}$ we have the relation
\begin{equation}\label{reldiperone}
%
2(12+(26-N)(25-N))\lambda(N)=  H_n(N)+2(26-N) H_h(N)+\tau(N)\mu(N) H_u(N), 
\end{equation}
where $\tau(N)$ is as in~\Ref{lmm}{relstableone}.  If $N\in\{6,14\}$ we also have the relation
\begin{equation}\label{reldiperonebis}
%
2(12+(26-N)(25-N))\lambda(N)=  H_n(N)+2(26-N)  H_u(N). 
\end{equation}
\end{proposition}
\begin{proof}
If $N$ is odd $\cF_{\Lambda_N}(\wt{O}^{+}(\Lambda_N))=\cF(N)$, and~\eqref{reldiperone} is simply a rewriting of~\eqref{relstableone}. 

Let $N$ be even, and let $\rho\colon 
\cF_{\Lambda_N}(\wt{O}^{+}(\Lambda_N))\to\cF(N)$ be the natural double covering map corresponding to the choice of decoration of $\Lambda_N$ given by $\xi$ (the only one if $N\not\equiv 6\pmod{8}$). 
Then~\eqref{reldiperone} is given by the push-forward by $\rho$ of~\eqref{relstableone}.   More precisely, $\rho_{*}\lambda(\wt{O}^{+}(\Lambda_N))=2\lambda(N)$, because 
$\rho^{*}\lambda(N)=\lambda(\wt{O}^{+}(\Lambda_N))$, and ~\Ref{crl}{otildegamma} gives  formulae for $\rho_{*}$ of the Heegner divisors in the right-hand side 
of~\eqref{relstableone}.

Now assume that $N\in\{6,14\}$, and choose the decoration of $\Lambda_N$ to be $\zeta$;  
then~\eqref{reldiperonebis} is given by the push-forward by $\rho$ of~\eqref{relstableone}.
\end{proof}
\begin{corollary}\label{crl:relstableone}
If $N\in\{6,14\}$, then  in $\Pic(\cF(N))_{\QQ}$ we have the relation $H_h(N)= H_u(N)$.
\end{corollary}
\begin{proof}
Subtract~\eqref{reldiperonebis} from~\eqref{reldiperone}.
\end{proof}
Starting from~\Ref{thm}{automdue}, and arguing as in the proof of~\Ref{lmm}{relstableone} and ~\Ref{prp}{relstableone}, one gets the following results.
\begin{lemma}\label{lmm:relstabletwo}
Let  $3\le N\le 17$. Let $\xi$ be a decoration of $\Lambda_N$, and let $\zeta,\zeta'\in A_{\Lambda_N}$ be the remaining non-zero elements. Then in $\Pic(\cF_{\Lambda_N}(\wt{O}^{+}(\Lambda_N)))_{\QQ}$ we have the relation
\begin{equation}
\scriptstyle
2(132+(18-N)(17-N))\lambda(\wt{O}^{+}(\Lambda_N))=  H_0(\Lambda_N)+\epsilon(N)2(18-N) H_\xi(\Lambda_N)+\tau(N)\mu(N+8) (H_{\zeta}(\Lambda_N) +H_{\zeta'}(\Lambda_N)), 
\end{equation}
where $\epsilon(N)$ and $\tau(N)$ are as in~\Ref{lmm}{relstableone}. 
\end{lemma}
\begin{proposition}[=\Ref{thm}{thmborcherds2}]\label{prp:relstabletwo}
Let  $3\le N\le 17$. Then in $\Pic(\cF(N))_{\QQ}$ we have the relation
\begin{equation}
%
2(132+(18-N)(17-N)\lambda(N)=  H_n(N)+2(18-N) H_h(N)+\tau(N)\mu(N+8) H_u(N), 
\end{equation}
where  $\tau(N)$ is as in~\Ref{lmm}{relstableone}. 
\end{proposition}
%

\section{The boundary divisor $\Delta(\Lambda_N,\xi_N)$}\label{sec:whydelta}
 In~\Ref{sec}{gitper} we have identified $\cF(19)$  with the period space of quartic $K3$ surfaces.
Let $\gM(19)$ be the  GIT moduli space of quartic surfaces in $\PP^3$;  then the natural period map 
\begin{equation}\label{mappaperiodi}
\gp_{19}\colon\gM(19)\dra \cF(19)^{*}
\end{equation}
is birational by  the Global Torelli Theorem for $K3$ surfaces (Piateshky-Shapiro and Shafarevich).
The projective variety  $\gM(19)=|\cO_{\PP^3}(4)|//\PGL(4)$ 
has Picard group (tensored with $\QQ$) of rank $1$; let $L(19)$ be the  generator of $\Pic(\gM(19))_{\QQ}$ induced by the hyperplane line bundle on $|\cO_{\PP^3}(4)|$. In the present section we will prove that 
\begin{equation}\label{pullample}
(\gp_{19}^{-1})^{*}L(19)|_{\cF(19)}=\lambda(19)+\Delta(19).
\end{equation}
Similar arguments also give that 
\begin{equation}\label{pullample18}
(\gp_{18}^{-1})^{*}L(18)|_{\cF(18)}=2(\lambda(18)+\Delta(18)).
\end{equation}
These are the computations that motivate our choice of boundary divisor for $\cF(N)$ (for any $N$). 
\begin{remark}
Let $\gM(20)$ be the GIT moduli space of double EPW sextics, let $\delta$ be the duality involution of  $\gM(20)$, and 
$\gp_{20}\colon \gM(20)/\la\delta\ra\dra \cF(20)^{*}$ be the period map, see~\eqref{per20}. Let $L(20)$ be an ample generator of the Picard group of $\gM(20)/\la\delta\ra$. 
We expect that a result similar to~\eqref{pullample} holds for $(\gp_{20}^{-1})^{*}L(20)|_{\cF(20)}$, namely that it is a positive multiple of $\lambda(20)+\Delta(20)$.
\end{remark}
An  important result follows from Equation~\eqref{pullample}. In order to state  it we introduce some notation. For $\beta\in[0,1]\cap\QQ$, let 
\begin{equation}
\scriptstyle
\cR(N,\beta):=\bigoplus_{m=0}^{\infty} H^0(\cF(N),m(\lambda(N)+\beta\Delta(N))),\quad \cF(N,\beta):=\Proj \cR(N,\beta).
\end{equation}
(As usual,  $H^0(\cF(N),m(\lambda(N)+\beta\Delta(N)))=0$ unless $m(\lambda(N)+\beta\Delta(N))$ is an integral Cartier divisor.)
 Thus $\cR(N,0)$ is the finitely generated algebra of automorphic forms, and hence $\cF(N,0)$ is the Baily-Borel compactification $\cF(N)^{*}$. If $3\le N\le 10$, then , by~\eqref{borcherds3} $\Delta(N)$ is a positive multiple of $\lambda(N)$, and hence in this range $\cF(N,\beta)=\cF(N)^{*}$ for all $\beta\in[0,1]\cap\QQ$. On the other hand, we will see that if $N\ge 11$ then $\cF(N,\beta)$ undergoes birational modifications as $\beta$ moves in $[0,1]\cap\QQ$.  

Now let  $N\in\{18,19\}$, and let $\rho(N)$ be equal to $1$ if $N=19$, and equal to $2$ if $N=18$.
 Let $m\ge 0$, and let 
\begin{equation}\label{mappasezioni}
\scriptstyle
(\gp_{N}^{-1}|_{\cF(N)})^{*}\colon H^0(\gM(N),m L(N))\lra H^0(\cF(N),  m \rho(N)(\lambda(N)+\Delta(N))),
\end{equation}
 be the map induced by~\eqref{pullample} if $N=19$, and by~\eqref{pullample18} if $N=18$. The following result should be compared to Theorem~8.6 of~\cite{looijengacompact}.
 \begin{proposition}\label{prp:isomsect}
Let  $N\in\{18,19\}$. Then the map in~\eqref{mappasezioni} is an isomorphism for all $m\in\ZZ$. 
\end{proposition}
 \begin{proof}
The map is clearly injective. In order to prove surjectivity it suffices to show that $\gp_N$ contracts no divisor. In order to prove this, let   $\cU(19)\subset |\cO_{\PP^3}(4)|$  and  
 $\cU(18)\subset|\cO_{\PP^\times\PP^1}(4,4)|$ be the open subsets of~\Ref{dfn}{u19} and of~\Ref{dfn}{u18} respectively. 
 The period maps $\gp_{N}$, for $N\in\{18,19\}$,   define  isomorphisms  between $\cU(N)//G(N)$ (where $G(19)=\PGL(4)$ and $G(18)=\Aut(\PP^1\times\PP^1)$) and the complement of the support of $\Delta(N)$, 
 see~\eqref{juventus} and~\eqref{bayern}. On the other hand, the complement of $\cU(N)\gquot G(N)$ in $\gM(N)$ has codimension greater than $1$. Thus $\gp_N$ contracts no divisor, as claimed.
\end{proof}
By~\Ref{prp}{isomsect}, there is an isomorphism
\begin{equation}
\cF(N,1)\cong \gM(N),\qquad\text{if $N\in\{18,19\}$.}
\end{equation}
Thus for  $N\in\{18,19\}$, the schemes $\cF(N,\beta)$ interpolate between the Baily-Borel compactification  $\cF(N)^{*}$ and the GIT moduli space $\gM(N)$. 
\subsection{Families of $K3$ surfaces}\label{subsec:famiglie}
\setcounter{equation}{0}
\subsubsection{Hodge bundle on families of $K3$ surfaces}
Let  $f\colon\cX\to B$ be a family of $K3$ surfaces. We let $\cL_B:=f_{*}\omega_{\cX/B}$; this is the \emph{Hodge bundle} on $B$ (the notation is imprecise because the Hodge-bundle is determined by  $f\colon\cX\to B$, not by $B$ alone). We let $\lambda_B:=c_1(\cL_B)\in\CH^1(B)$. Suppose that $\cX$ is a family of  polarized quartic $K3$ surfaces, and let $\mu\colon B\to \cF(19)$ be the period map: then
 \begin{equation}
 \mu^{*}\lambda(19)=\lambda_B.
\end{equation}
 Suppose that $\cX$ and $B$ are  smooth. The exact sequence
 $$0\lra f^{*}\Omega_B \overset{df}{\lra} \Omega_{\cX}\lra \Omega_{\cX/B}\lra 0$$
and the isomorphism $f^{*}\cL_B\cong \bigwedge^2\Omega_{\cX/B}$ give the formula
 \begin{equation}\label{eccolam}
f^{*}\lambda_B=c_1(K_{\cX})-f^{*}c_1(K_{B}).
\end{equation}
\subsubsection{Families of quartic surfaces}\label{subsubsec:intquartics}
Let  $F,G\in\CC[x_0,\ldots,x_3]_4$ be linearly independent, and such that $\Gamma_1:=\la \Div(F),\Div(G)\ra$ is a Lefschetz pencil  of quartic surfaces.  Let 
\begin{equation}
\cX_1:=\{([\alpha,\beta],[x])\in\Gamma_1\times\PP^3 \mid \alpha F(x)+\beta G(x)=0\}.
\end{equation}
The projection onto the second factor is a family $f_1\colon \cX_1\to \Gamma_1$ of polarized quartic surfaces. 

Now let us we define a one-parameter family of hyperelliptic $K3$ surfaces. Let $\Gamma_2:=\PP^1$. Let 
  $D\in| \cO_{\Gamma_2}(2)\boxtimes\cO_{\PP^1\times\PP^1}(4,4)|$ be generic, and let $\rho\colon\cX_2\to \Gamma_2\times(\PP^1)^2$ be the double cover branched over $D$.  Let $f_2\colon \cX_2\to\Gamma_2$ be the composition of the double cover map $\rho$ and the projection $\Gamma_2\times(\PP^1)^2\to\Gamma_2$. Let  $t\in\Gamma_2$; then $X_{2,t}:=f_2^{-1}(t)$ is the double cover of $\PP^1\times\PP^1$ branched over the $(4,4)$-curve $D_t:=D|_{\{t\}\times(\PP^1)^2}$. 
  Thus $X_{2,t}$ is  a $K3$ surface,   $\rho^{*}(\cO_{\PP^1}(1)\boxtimes\cO_{\PP^1}(1))$ restricts to a polarization of  $X_{2,t}$, and with this polarization $X_{2,t}$ is a quartic hyperelliptic $K3$. From here on we assume that the pencil of branch curves $D_t$, for $t\in\Gamma_2$, is a Lefschetz pencil.

Next, we define a one-parameter family of unigonal $K3$ surfaces. Let $Y:=\PP(\cO_{\PP^2}(4)\oplus\cO_{\PP^2})$. Let $\varphi\colon Y\to\PP^2$ be the structure map, and $F:=\varphi^{-1}(\text{line})$. Let $A:=\PP(\cO_{\PP^2}(4))\subset Y$. Adjunction on $F\cong\FF_4$ gives that 
\begin{equation}\label{kappay}
K_Y\equiv -2A-7F.
\end{equation}
Let  $B\in|3A+12F|$ be generic, in particular it is smooth and it does not intersect $A$. Let $\pi\colon Z\to Y$ be the double cover branched over $A+B$. If $F$ is as above and generic then $\pi^{-1}F$ is a smooth unigonal $K3$ surface. We get a family of such $K3$'s by choosing a generic $\cX_3\in |\cO_Z(\pi^{*}F)\boxtimes\cO_{\Gamma_3}(1)|$, where 
$\Gamma_3=\PP^1$. In fact let $f_3\colon\cX_3\to\Gamma_3$ be the restriction of projection, and let  $t\in\Gamma_3$. The linear map $\Gamma_3\to |F|$ associates to $t\in\Gamma_3$ a line $L_t\subset\PP^2$, and 
the surface $X_{3,t}:=f_3^{-1}(t)$ is a double cover $\pi_t\colon X_{3,t}\to \varphi^{-1}L_t\cong\FF_4$; one checks easily that $X_{3,t}:=f_3^{-1}(t)$ is a $K3$ surface. 
Let $A_t$ be the negative section of $\varphi^{-1}L_t\cong\FF_4$; then $\pi_t^{*}A_t=2R_t$, where $R_t$ is a (smooth) rational curve. Let $E_t:=\pi_t^{*}F_t$, where $F_t$ is a fiber of the fibration $\varphi^{-1}L_t\cong\FF_4\to L_t$. Then $R_t+3E_t$ is a polarization of degree $4$ of $X_{3,t}$, and $(X_{3,t},R_t+3E_t)$ is unigonal.

For $ i\in\{1,2,3\}$ we have  period maps
\begin{equation}\label{unigfam}
\Gamma_i  \overset{\mu_i}{\lra}  \cF(19).
\end{equation}
\begin{table}[tbp]
\caption{Intersection numbers of families of quartic $K3$'s}\label{table:treinter}
\vskip 1mm
\centering
\renewcommand{\arraystretch}{1.60}
\begin{tabular}{l l l l l}
 & $\lambda(19)$ & $H_n(19)$ & $H_h(19)$ & $H_u(19)$ \\
\toprule
 $\mu_{1,*}(\Gamma_1)$ & $1$ & $108$  & $0$ & $0$   \\
\midrule
 $\mu_{2,*}(\Gamma_2)$  & $1$ &  $136$  & $-2$  &  $0$   \\
\midrule
 $\mu_{3,*}(\Gamma_3)$  & $1$ &  $264$  &  $0$ & $-2$   \\
\bottomrule 
\end{tabular}
\end{table} 
\begin{proposition}
The intersection formulae of Table~\ref{table:treinter} hold.
\end{proposition}
\begin{proof}
First of all, notice that
\begin{equation*}
\mu_{i,*}\Gamma_i\cdot\lambda(19)=\Gamma_i\cdot\mu_{i,*}\lambda(19)=\deg \lambda_{\Gamma_i}. 
\end{equation*}
Next, one computes $\deg \lambda_{\Gamma_i}$ by applying~\eqref{eccolam}. 
The intersection number $\mu_{i,*}\Gamma_i\cdot H_n(19)$ is equal to the number, call it $\delta(\Gamma_i)$, of singular fibers of $f_i$, because $f_i$ is a Lefschetz fibration. The formula that gives   $\delta(\Gamma_i)$ is the following:
\begin{equation}\label{nodaldeg}
\delta(\Gamma_i)=2\chi_{top}(K3)-\chi_{top}(\cX_i)=48-\chi_{top}(\cX_i).
\end{equation}
Thus it suffices to compute the Euler characteristic of $\cX_i$; we leave details to the reader.  The intersection numbers of the third and fourth columns are obtained as follows. First 
$\es=\mu_1(\Gamma_1)\cap H_h(19)=\mu_1(\Gamma_1)\cap H_u(19)$ by~\Ref{prp}{perhyp}. Next 
$\mu_2(\Gamma_2)\cap  H_u(19)=\es$ and  $\mu_3(\Gamma_3)\cap  H_h(19)=\es$  by~\Ref{lmm}{inthypunig}. The remaining numbers are obtained by applying Borcherd's  relation for $N=19$, which reads
\begin{equation}\label{borcherds19}
108\lambda(19)=H_n(19)+14 H_h(19)+78 H_u(19),
\end{equation}
together with the computations that have already been proved.
\end{proof}
\begin{proposition}\label{prp:oldresult}
A basis of $\Pic(\cF(19))_{\QQ}$ is provided by the choice of any three of the classes $\lambda(19),H_n(19),H_h(19),H_u(19)$. The space of linear relations among 
$\lambda(19),H_n(19),H_h(19),H_u(19)$ is generated by Borcherds  relation~\eqref{borcherds19}.
\end{proposition}
\begin{proof}
Let $\cU(19)\subset|\cO_{\PP^3}(4)|$ be the open subset  of~\Ref{dfn}{u19}. Then we have the isomorphism in~\eqref{juventus}. 
Now $\CH^1(\cU(19)\gquot\PGL(4))_{\QQ}=\Pic(\cU(19)\gquot\PGL(4))_{\QQ}$ (e.g.~because $\cF(19)$ is $\QQ$-factorial), and $\Pic(\cU(19)\gquot\PGL(4))_{\QQ}$ is isomorphic to the group of $\PGL(4)$-linearized line bundles on $\cU(19)$ (tensored with $\QQ$), which is a subgroup of $\Pic(\cU(19))_{\QQ}$ because $\PGL(4)$ has no non-trivial characters. On the other hand, $\Pic(\cU(19))_{\QQ}\cong\QQ$ because  $ |\cO_{\PP^3}(4)|\setminus \cU(19)$ has codimension greater than $1$ in $ |\cO_{\PP^3}(4)|$. 
Thus $\CH^1(\cU(19)\gquot\PGL(4))_{\QQ}\cong\QQ$, and a generator is the class of the divisor parametrizing singular quartics. 
By~\eqref{juventus} it follows that
\begin{equation*}
\CH^1(\cF(19)\setminus H_h(19)\setminus H_u(19))_{\QQ}\cong\QQ,
\end{equation*}
and that the restriction of $H_n(19)$ is a generator.   By the localization sequence associated to the inclusion of $(\cF(19)\setminus H_h(19)\setminus H_u(19))$ into $\cF(19)$, it follows that $\CH^1(\cF(19))_{\QQ}$ is generated by $H_n(19),H_h(19),H_u(19)$. Table~\ref{table:treinter} shows that  $H_n(19),H_h(19),H_u(19)$ are linearly independent, and also the remaining statements of the proposition.
\end{proof}
Lastly, we define a  family of  quartic surfaces in $\PP^3$ \lq\lq degenerating\rq\rq to a hyperelliptic $K3$ surface.   Let $Q\in\CC[x_0,\ldots,x_3]_2$ be a non-degenerate quadratic form, and $G\in\CC[x_0,\ldots,x_3]_4$ such that $\Div(G)$ is transverse to $\Div(Q)$, and such that 
the pencil of quartics $B:=\la \Div(G),\Div(Q^2)\ra$ is a Lefschetz pencil  away from $\Div(Q^2)$.  
Let 
\begin{equation}
\cY:=\{([\lambda,\mu],[x])\in B\times\PP^3 \mid \lambda Q^2(x)+\mu G(x)=0\}.
\end{equation}
The family $\cY\to B$ is a family of quartics away from the inverse image of $[1,0]$. Let $\Gamma_4\to B$ be the double cover branched over $[1,0]$ and $[0,1]$, and let 
$\cY_4\to\Gamma_4$ be the pull-back of $\cY$. Let $\cX_4$ be the normalization of $\cY_4$; the natural map $f_4\colon \cX_4\to\Gamma_4$ is a family of polarized quartic $K3$ surfaces. Let $p\in\Gamma_4$ be the (unique) point mapping to $[1,0]$. The surface $f_4^{-1}(p)$ is the double cover of the smooth quadric $\Div(Q)$ branched over $V(Q,G)$, i.e.~a hyperelliptic quartic $K3$.  If $t\in (\Gamma_4\setminus\{p\})$, then $f_4^{-1}(t)$ is a non-hyperelliptc (and non-unigonal) quartic $K3$.  We have the period map
\begin{equation}\label{quattro}
\Gamma_4  \overset{\mu_4}{\lra}  \cF(19).
\end{equation}
\begin{corollary}\label{crl:renzishow}
Keep notation as above. Then
\begin{equation*}
\scriptstyle
\mu_{4,*}(\Gamma_3)\cdot\lambda(19)=1,\quad \mu_{4,*}(\Gamma_3)\cdot H_n(19)= 80,\quad   \mu_{4,*}(\Gamma_3)\cdot H_h(19)= 2,\quad 
\mu_{4,*}(\Gamma_3)\cdot H_u(19)=0.
\end{equation*}
\end{corollary}
\begin{proof}
Except for the second-to-last formula, the others are obtained by arguments similar to those employed to obtain the formulae of Table~\ref{table:treinter}. One gets the missing formula by applying Borcherd's formula~\eqref{borcherds19}. 
\end{proof}
Notice that the set-theoretic intersection $\mu_{4}(\Gamma_3)\cap H_h(19)$ consists of a single point, namely $\mu_4(p)$, and it counts with multiplicity $2a$, for some $a>1$. In order to conclude that $\mu_{4,*}(\Gamma_3)\cdot H_h(19)= 2$ one needs a non-trivial computation; we avoid this thanks to Borcherds' relation~\eqref{borcherds19}.
\begin{remark}
In the present subsection, we have invoked Borcherds' first relation in order to compute the degree of normal bundles. By employing 
the (elementary) results of~\Ref{subsec}{normale}, one can avoid the use of Borcherds' relations, and in fact one can use the computations in this subsection in order to check the validity of Borcherds' first relation  for $N=19$. 
\end{remark}
\subsection{Proof of~\eqref{pullample}}\label{subsec:whydelta}
\setcounter{equation}{0}
By~\Ref{prp}{oldresult}, there exist $x,y,z\in\QQ$ such that
\begin{equation}\label{lincomb19}
(\gp_{19}^{-1})^{*}L(19)|_{\cF(19)}=x\lambda(19)+y H_h(19)+z H_u(19).
\end{equation}
We will compute $x,y$ by equating the intersection numbers of the two sides with complete curves in $\cF(19)$. For this to make to sense the complete curves must avoid the indeterminacy locus of $\gp_{19}^{-1}$. By~\eqref{juventus}, the indeterminacy locus $I(19)$ is contained in $H_h(19)\cup H_u(19)$, and it is of codimension at least $2$ because $\cF(19)$ is normal. Thus $I(19)=I_h(19)\cup I_u(19)$, where  $I_h(19)\subset H_h(19)$ and $I_u(19)\subset H_u(19)$ are proper closed subsets. Now let $\cX_1\to\Gamma_1$ and $\cX_4\to\Gamma_4$ be the complete families defined in~\Ref{subsubsec}{intquartics}. Every surface $f_1^{-1}(t)$ is a stable quartic, and similarly for $f_4^{-1}(t)$ if $t\not=p$, while the semistable quartic surface in  $\PP^3$ corresponding to $p$ is to be understood as the double quadric $2 V(Q)$. Let $\theta_i\colon \Gamma_i\to\gM(19)$ be the corresponding modular map, for $i\in\{1,4\}$. The map $\mu_i\colon\Gamma_i\to \cF(19)$ considered in~\Ref{subsubsec}{intquartics} is equal to $\gp_{19}\circ\theta_i$. The curve $\mu_1(\Gamma_1)$ avoids the indeterminacy locus $I(19)$, because it is disjoint from  $H_h(19)\cup H_u(19)$. On the other hand, the curve  $\mu_4(\Gamma_4)$ intersects   $H_h(19)\cup H_u(19)$ in a single point, namely $\mu_4(p)$, which belongs to $H_h(19)$. Since the double cover  $f_4^{-1}(p)\to V(Q)$ has arbitrary branch curve (among those with ADE singularities), we may assume that 
$\mu_4(p)\notin I_h(19)$. Thus, 
\begin{equation*}
\mu_{i,*}(\Gamma_i)\cdot (\gp_{19}^{-1})^{*}L(19)=\theta_{i,*}(\Gamma_i)\cdot L(19)=
\begin{cases}
1 & \text{if $i=1$,} \\
2 & \text{if $i=4$.}
\end{cases}
\end{equation*}
Equation~\eqref{lincomb19}, together with Table~\ref{table:treinter} and~\Ref{crl}{renzishow}, gives
\begin{equation*}
\mu_{i,*}(\Gamma_i)\cdot (\gp_{19}^{-1})^{*}L(19)=
\begin{cases}
x & \text{if $i=1$,} \\
x+2y & \text{if $i=4$.}
\end{cases}
\end{equation*}
It follows that $x=1$ and $y=1/2$. We will prove that $z=1/2$ by showing that $\gp_{19}^{-1}$ is regular along $H_u(19)$ (i.e.~$I_u(19)=\es$).   First $\gp_{19}^{-1}$ is regular away from the proper closed subset $I_u(19)$, and the image $\gp_{19}^{-1}(H_u(19)\setminus I_u(19))$ is the single point  $q\in\gM(19)$ parametrizing the tangential developable surface of tangents to a twisted cubic in $\PP^3$. (REFERENCE?).
   Let $Z\subset \cF(19)\times\gM(19)$ be the graph of $\gp_{19}^{-1}|_{\cF(19)}$, and $\pi\colon Z\to\cF(19)$ be the projection. We must prove that $Z$ is an isomorphism over $H_u(19)$.  Assume the contrary. By Zariski's Main Theorem every fiber of $\pi$ is connected, in particular those over points of $I_u(19)$. Let $x\in I_u(19)$; then $\pi^{-1}(x)$ is of strictly positive dimension and it contains the point $(x,q)$, where $q\in\gM(19)$ is as above. Hence the image of $\pi^{-1}(x)$ in $\gM(19)$ is a connected closed subset $C$ of  strictly positive dimension,  containing $q$. Moreover $C$ is contained
   in the indeterminacy locus of $\gp_{19}$. In fact, if this is not the case, then, since $\cF(19)$ is $\QQ$-factorial, there is a locally closed codimension-$1$ subset $D\subset\gM(19)$ contained in the locus where $\gp_{19}$ is regular, with values in $\cF(19)$ (not in the boundary $\cF(19)^{*}\setminus\cF(19)$), which is contracted by $\gp_{19}$ (i.e.~$\dim\gp_{19}(D)<18$).  But $q$ is an isolated point of the indeterminacy  
  locus of $\gp_{19}$ by~\cite{shah}, this is a contradiction. Thus $\gp_{19}^{-1}$ is regular along $H_u(19)$. 
  Since the restriction of $\gp_{19}^{-1}L(19)$ to $H_u(19)$  is trivial, the line bundle $(\gp_{19}^{-1})^{*}L(19)$ is trivial on $H_u(19)$. Thus
\begin{equation*}
\mu_{3,*}(\Gamma_3)\cdot\left(\lambda(19)+\frac{1}{2} H_h(19)+z H_u(19)\right)=0.
\end{equation*}
By Table~\eqref{table:treinter}, it follows that   $z=1/2$.
\qed
%

\section{Predictions for the variation of log canonical models}\label{sec:predictions}
The goal of the present section is to predict the behavior of 
$$ \cF(N,\beta):=\Proj R(\cF(N), \lambda(N)+\beta\Delta(N))$$
 for $N\ge 3$ and   $\beta\in[0,1]\cap\QQ$. Specifically, we will give a conjectural decomposition of $\gp_N^{-1}$ as a product of flips and, as a last step,  the contraction of the strict transform of the support of $\Delta(N)$ (denoted 
 $\Delta^{(1)}(N)$, see~\Ref{subsubsec}{arrstrata}). Each flip will correspond to a critical value $\beta=\beta^{(k)}(N)\in(0,1)\cap\QQ$. The centre of the flip corresponding to  $\beta^{(k)}(N)$
 will be (the strict transform of) a building block of the $D$ tower (see~\Ref{subsec}{diserie}), and it has 
   codimension  $k$.   (Of course the value  $\beta^{(1)}(N)=1$ corresponds to 
 the contraction of  $\Delta^{(1)}(N)$.)  In the case of quartic surfaces ($N=19$) it is possible to match our predicted flip centers  with geometric loci in the GIT moduli space $\gM(19)$, this is discussed in the companion note~\cite{announcement}. In the case of hyperelliptic quartics ($N=18$), using  VGIT, we can go further and not only match the arithmetic predictions with the geometric ones (essentially the same matching as for quartics), but actually verify that both the arithmetic predictions and the geometric matching are correct. This will be discussed elsewhere. 
\subsection{Main Conjectures and Results}\label{subsec:predictions}
Let $\beta\in(0,1]\cap\QQ$. Then the following hold:     
\begin{itemize}
\item $\lambda(N)+\beta \Delta(N)$ is a big line bundle with base locus contained in $\Delta^{(1)}(N)$, because $\lambda(N)$ is ample, and $\Delta(N)$ is effective. 
\item
If $\beta<1$,  the restriction of $\lambda(N)+\beta \Delta(N)$ to  $\Delta^{(1)}(N)$ is  big with base locus contained in $\Delta^{(2)}(N)$, see~\Ref{prp}{ennekappaelbeta}, ~\Ref{prp}{restunig3}, and ~\Ref{prp}{restunig4}. In particular, assuming lifting of sections from $\Delta^{(1)}(N)$ to $\cF(N)$, the base locus of $\lambda(N)+\beta \Delta(N)$ is contained in $\Delta^{(2)}(N)$.
\item for $N\in\{18,19\}$, $\cF(N)^{*}\dra \cF(N,1)\cong\gM(N)$ is a birational contraction of $\Delta^{(1)}(N)$, by~\Ref{prp}{isomsect}. This is  expected to hold for any $N$ by the arguments of Looijenga. 
\end{itemize}
In order to understand what should be going on, it is convenient to let $\beta\in(0,1]\cap\QQ$ decrease, starting from $\beta=1$.  For anyone familiar with linearized arrangements, and in fact implicitly contained in Looijenga~\cite{looijengacompact}, $\Delta^{(1)}(N)$ should be contractible regularly away from $\Delta^{(2)}(N)$ (with corresponding $\beta^{(1)}(N)=1$), but  in order to contract it regularly, one must first flip $\Delta^{(2)}(N)$, at least away from $\Delta^{(3)}(N)$, with corresponding  $0<\beta^{(2)}(N)<1$. The first-order predictions  (\S\ref{firstorder}) are obtained by iterating this procedure, 
i.e.~following  Looijenga~\cite{looijengacompact}.  These predictions are based on the combinatorics of the linearized arrangement $\wt{\Delta}^{(1)}(N)$ (see~\Ref{subsubsec}{arrstrata}); one essentially  computes the log canonical threshold (at the generic point of $\Delta^{(k)}(N)$) as  in~\cite{mustata},  keeping track of the ramification.  
\begin{remark}
There are two well-known examples of this type of behaviour: the moduli spaces of degree $2$ $K3$ surfaces (\cite{shah,looijengavancouver}), and of cubic fourfolds (\cite{lcubic,cubic4fold}). In both cases there are analogues of $\Delta^{(k)}(N)$; we will denote them by $\Delta^{(k)}$.  In the first example, $\Delta^{(2)}$ is empty, and thus  $\Delta^{(1)}$ is  contracted right away. In the second example $\Delta^{(2)}\not=\emptyset$, but $\Delta^{(3)}=\emptyset$, and hence one has a flip of $\Delta^{(2)}$ followed by a contraction. The case of quartic surfaces, or more generally $D$ locally symmetric varieties, for any $N\ge 3$, is at the opposite end of the spectrum. In fact $\Delta^{(N)}(N)\neq \emptyset$, and hence a key hypothesis in Looijenga's treatment~\cite{looijengacompact}, namely  the assumption in Corollary~7.5,  is not satisfied. 
\end{remark}
As noted in the introduction, we have observed that for quartics ($N=19$) the predictions above definitely fail for $k>9$. We will show below 
(see~\Ref{subsubsec}{refinedpredict} for precise statements) that there is a subtle arithmetic reason (coming from Borcherds' relations) for this. This leads to refined predictions for $\beta^{(k)}(N)$, which we can match geometrically for quartics and verify for hyperelliptic quartics. Finally, using Gritsenko's relation, we can actually check that 
$\lambda(N)+\beta \Delta(N)$ is \lq\lq numerically nef\rq\rq\ on $\cF(N)$ for $\beta<\frac{1}{N-10}$. 

Our predictions are based on the computation of the restriction of $(\lambda(N)+\beta\Delta(N))$ to the $k$-th stratum $\Delta^{(k)}(N)$ of the boundary. To obtain precise values for $\beta^{(k)}(N)$ we make heavy use of the structure of the $D$-tower, and its periodic structure. 

\begin{remark}
In the present discussion, we are ignoring the behavior along the Baily-Borel boundary. This is a very interesting aspect that should be discussed elsewhere: Roughly speaking, $\Delta(N)$ is a $\bQ$-Cartier divisor on $\cF(N)$, but its closure in $\cF(N)^*$ fails to be $\bQ$-Cartier, if $N$ is large. For $0<\beta=\epsilon\ll 1$, one should obtain a $\bQ$-factorialization of $\Delta(N)$ as described by Looijenga \cite{looijengacompact}. Thus, for $\beta$ small, $\cF(N,\beta)$ should differ from  $\cF(N)^{*}$ only at the boundary, in fact it should be a small blow-up of the boundary.
 For $\beta$ larger, we predict the opposite behaviour: the modifications at the boundary are induced from the modifications of the interior. The focus of our paper is to understand the interior modifications. 
\end{remark}
\subsubsection{First Order Predictions}\label{firstorder} 
As we explained, the starting point of Looijenga is the observation that, in the embedding 
$\cD^{+}_N\subset \check \cD^{+}_N\subset \bP^{N+1}$ (we let $\cD^{+}_N=\cD^{+}_{\Lambda_N}$), the automorphic bundle is the restriction of $\calO_{\PP^{N+1}}(-1)$, while a Heegner divisor is a section of $\calO_{\cD^{+}_N}(1)$. This fact suggests that $\lambda(N)+\Delta(N)$ should contract $\Delta^{(1)}(N)$. As always, a linearized   arrangement is stratified by linear strata of the intersections, and the first order predictions (leading to candidates for the critical values $\beta^{(k)}(N)$) are a simple function of combinatorics, namely the number of hyperplanes intersecting in a stratum, versus the codimension of that strata (compare \cite{mustata}). Of course, in this situation, there is a slight complication due to the fact that  these hyperplanes are reflection hyperplanes,  (note that our $\Delta$ involves a $\frac{1}{2}$ factor for this reason). 

Concretely,  for most values of $N$ and $k$, we will prove 
 (see \Ref{prp}{ennekappaelbeta} for a precise statement) that the following formula holds:
\begin{equation}\label{induzione}
(\lambda(N)+\beta\Delta(N))_{\mid \Delta^{(k)}(N)}=(1-k\beta)\lambda(N) + \frac{1}{2}\beta c_1\left(\cO_{\Delta^{(k)}(N)}(\Delta^{(k+1)}(N))\right).
\end{equation}
We recall (see~\Ref{prp}{deltakappa}) that for most choices of $N$ and $k$, $\Delta^{(k)}(N)=\im f_{N-k,N}$, and~\eqref{induzione} should be read as 
 $f^{*}_{N-k,N}(\lambda(N)+\beta\Delta(N))=(1-k\beta)\lambda(N-k) +\beta\Delta(N-k)$.  

If, following Looijenga, we assume that the stratum $\Delta^{(k+1)}(N)$ is flipped \emph{before} the stratum $\Delta^{(k)}(N)$, we get the prediction
\begin{equation}\label{prediction}
\beta^{(k)}(N)=\frac{1}{k},\qquad k\in\{1,\ldots, N\}.
\end{equation}
In any case, note that $(\lambda(N)+\beta\Delta(N))_{\mid \Delta^{(k)}}(N)$ is big for $\beta<\frac{1}{k}$ by~\eqref{induzione}. Thus assuming a certain lifting of sections (a reasonable assumption, that we can prove in some cases), we get that the generic point of  $\Delta^{(k)}(N)$ is not affected by the birational transformations occurring for $\beta\in(0,\frac{1}{k})$. 

The two key  ingredients that give Equation~\eqref{induzione} are the computation of intersections of distinct Heegner divisors (see~\Ref{subsec}{restrram}) and a normal bundle computation based on adjunction (see~\Ref{subsec}{normale}). Of course, the actual computations depend on our particular lattices and Heegener divisors, but the method adopted  here can be applied in many other instances.  (Similar computations are implicitly or explicitly contained in \cite{ghs} or \cite{looijengacompact}.)
\subsubsection{Refined Predictions}\label{subsubsec:refinedpredict} 
For a long time, we have been puzzled by the observation that the first order predictions for quartic $K3$'s are definitely wrong for $k>9$; in fact the  indeterminacy locus of the period map $\gp_{19}$ has dimension $8$, while the first order predictions would give that it has higher dimension. It took us some time to understand the arithmetic reason for this, and we regard it as the main result of the present paper. 

The key assumption in our Looijenga/first order predictions is the contractibility of $\Delta^{(k+1)}(N)$ inside $\Delta^{(k)}(N)$. This appears to be a  natural assumption: each $\Delta^{(k)}(N)$ is (essentially) a step in the $D$-tower, there is a modulo $8$ arithmetic periodicity, etc. However, to our surprise, we have discovered the following trichotomy for the $D$-tower:
\begin{enumerate}
\item[(1)] $H_h(N)$ is birationally contractible in $\cF(N)$ for $N>14$;
\item[(2)] $H_h(N)$ moves in a linear system of dimension at least $1$ if $N\le 14$;
\item[(3)] $H_h(N)$ is ample on $\cF(N)$ for $N\le 10$. 
\end{enumerate}
Note: (2) and (3) are theorems (in fact (3) is essentially one of the main results of Gritsenko \cite{gritsenko}), while (1) is conjectural in general, but known for $N=18,19$ cf.~\Ref{sec}{whydelta} (probably one can prove it also $N=15,16,17$ via  VGIT for hyperelliptic quartics, and with some effort also for $N=20$ via double EPW sextics). The explanation for this comes from Borcherds' relations discussed in \Ref{sec}{picdtower}. The computations here are very much in the same spirit as those of Gritsenko--Hulek--Sankaran \cite{ghs} (essentially, we vary the dimension of the domain, versus the discriminant group of the lattice in \cite{ghs}). 

Taking into account the behavior of the hyperelliptic divisor described above, we  correct our first order predictions and arrive to the main result of the paper:
\begin{prediction}\label{pred:mainpred}
Let $N\ge 15$. The ring of sections   $R(\cF(N), \lambda(N)+\beta\Delta(N))$ is finitely generated for $\beta\in[0,1]\cap \bQ$, and  the walls of the Mori chamber decomposition of the cone 
\begin{equation*}
\{\lambda(N)+\beta\Delta(N) \mid \beta\in\QQ,\ \beta>0\} 
\end{equation*}
 are generated by 
$\lambda(N)+\frac{1}{k}\Delta(N)$, where
$k\in\{1,\ldots,N-10\}$ and $k\not=N-11$.
The behaviour of $\lambda(N)+\frac{1}{k}\Delta(N)$, for $k$ as above, is described as follows. For $k=1$, $\cF(N,1)$ is obtained from $\cF(N,1-\epsilon)$ by contracting the strict transform of $\Delta^{(1)}(N)$. 
If $2\le k$, then the birational map between $\cF(N,\frac{1}{k}-\epsilon)$ and 
$\cF(N,\frac{1}{k}+\epsilon)$ is a flip whose center is
\begin{enumerate}
\item the strict transform of $\im f_{N-k,N}$, if  $2\le k\le N-14$, $k\not\equiv N-2\pmod{8}$, and either $k\not=4$ or $N\not\equiv 4\pmod{8}$,
\item the  union of the strict transforms of $\im f_{N-4,N}$ and $\im(f_{N-1,N}\circ l_{N-1})$,  if $k=4$ and $N\equiv 4\pmod{8}$, 
\item the union of the strict transforms of $\im f_{N-k,N}$ and $\im(f_{N-(k-1),N}\circ l_{N-(k-1)})$,  if  $3\le k\le N-10$, and $k\equiv N-2\pmod{8}$ (notice that this includes the case $k=N-10$),
\item the strict transform of $\im(f_{13,N}\circ q_{13})$ if $k=(N-13)$,
\item the strict transform of $\im(f_{12,N}\circ m_{12})$ if $k=(N-12)$,
\end{enumerate}
\end{prediction}

\begin{remark}[Early Termination] 
The predictions in~\Ref{pred}{mainpred} differ from  the first order predictions for  $k>N-14$. In fact  the generic point of $\Delta^{(k)}(N)$ is unaffected by the birational transformations for $\beta<\frac{1}{N-14}$. More precisely,  $f_{13,N}(\cF(13)\setminus H_u(13))$ (which is contained in 
$\Delta^{(N-13)}(N)$) will be flipped when $\beta=\frac{1}{N-14}$  (at once with $\Delta^{(N-14)}$). Moreover, 
 $\Delta^{(N-10)}(N)$ consists of two disjoint components $f_{10,N}(\cF(10)))$ and $f_{11,N}(H_u(11))$ (see~\Ref{prp}{deltakappa}). The first component will be flipped at $\beta=\frac{1}{N-14}$, while the second is the center of the flip for  $\beta=\frac{1}{N-10}$ (the smallest critical value). 
\end{remark}
In particular, for quartic surfaces, we get
\begin{prediction}
The variation of log canonical models interpolating between Baily-Borel and GIT for quartic surfaces has critical slopes
$$\beta\in\left\{1,\frac{1}{2},\frac{1}{3},\frac{1}{4},\frac{1}{5},\frac{1}{6},\frac{1}{7},\frac{1}{9}\right\}.$$
The center of the flip for the exceptional values $\beta<\frac{1}{5}$ correspond to $T_{3,3,4}$, $T_{2,4,5}$, and $T_{2,3,7}$ marked $K3$ surfaces (loci inside $\cF(19)$) and those are flipped to loci of quartics with $E_{14}$, $E_{13}$, and  $E_{12}$ singularities in the GIT model (see Shah \cite{shah}). 
\end{prediction}
In a  companion note \cite{announcement}, we give a full geometric matching (and some evidence) of the GIT stratification from Shah's model to the arithmetic stratification predicted above. In the case of hyperelliptic quartics, we can go further, and following \cite{g4git,g4ball} give actual proofs that the predictions of $\beta$ and the geometric meaning are actually accurate for $N=18$, this will be discussed elsewhere. 
\begin{remark}
A similar VGIT realization of a Hassett-Keel-Looijenga variation of Baily-Borel is in~\cite{raduthesis}. In that paper, the first named author was not yet aware of the deeper structure associated to the variation of log canonical models (i.e.~what we might call the Hassett--Keel--Looijenga program for $K3$s), but nonetheless the arithmetic stratification associated to the realization of a certain moduli space of pairs as a locally symmetric variety 
 was the key to the understanding of the deformations of a certain class of singularities.
 \end{remark}
\subsection{Intersection of two distinct Heegner divisors}\label{subsec:restrram}
\setcounter{equation}{0}
We proved (see~\Ref{prp}{hyplocsymm}, \Ref{prp}{unigcong3}, and~\Ref{prp}{unigcong4}) that for $4\le N$ there are isomorphisms
 \begin{equation}\label{efferitorna}
f_N\colon \cF(N-1)\overset{\sim}{\lra}  H_h(N). 
\end{equation}
We also proved that,  for $k\ge 0$,  there are isomorphisms
\begin{equation}\label{elleritorna}
l_{8k+3}\colon \cF( II_{2,2+8k})\overset{\sim}{\lra} H_u(8k+3), 
\end{equation}
 and 
\begin{equation}\label{accaritorna}
m_{8k+4}\colon \cF(II_{2,2+8k}\oplus  A_1)\overset{\sim}{\lra} H_u(8k+4). 
\end{equation}
The key ingredient in the heuristics for our predictions is the computation of the restrictions of the line bundle $\lambda(N)+\beta \Delta(N)$ to the various strata $\Delta^{(k)}(N)$. 
Via pull-back by $f_N$ (and $l_N$, $m_N$, or $q_N$ if applicable), this restriction is computed in an inductive manner. The  result below is our starting point.
\begin{proposition}
With notation as above, the following  formulae hold:
\begin{eqnarray}
\scriptstyle
f_N^{*} H_n(N) & \scriptstyle = & 
\begin{cases}
\scriptstyle H_n(N-1) +2H_h (N-1) & \scriptstyle \text{if  $N\not \equiv 5 \mod 8$,}\\
\scriptstyle H_n(N-1) +2H_h (N-1) +H_u(N-1) & \scriptstyle \text{if  $N\equiv 5 \mod 8$,}
\end{cases} \label{effennehyper} \\
\scriptstyle f_N^{*} H_u(N) & \scriptstyle = & 
\begin{cases}
\scriptstyle 0 & \scriptstyle \text{if $N\equiv 2\pmod{8}$,} \\
\scriptstyle 2H_u(N-1) & \scriptstyle \text{if $N\not\equiv 2\pmod{8}$,}
\end{cases} \label{effenneunig} \\
 l_{8k+3}^{*} H_n(8k+3) &   = &   H_n( II_{2,2+8k}), \label{ellennenod}  \\
  l_{8k+3}^{*} H_h(8k+3) &  = &  0,  \label{ellennehyper}  \\
 m_{8k+4}^{*} H_n(8k+4)  &  = &  H_n(II_{2,2+8k}\oplus  A_1), \label{emmennenod} \\
 m_{8k+4}^{*} H_h(8k+4) &   = &  2 H_u(II_{2,2+8k}\oplus  A_1), \label{emmennehyper} \\
 q_{8k+5}^{*} H_n(8k+5) &   = &  H_n(II_{2,2+8k}\oplus A_2)+2 H_u(II_{2,2+8k}\oplus A_2), \label{quennenod} \\
 q_{8k+5}^{*} H_h(8k+5) &   = &  2 H_u(II_{2,2+8k}\oplus A_2). \label{quennehyper} 
\end{eqnarray}
\end{proposition}
\begin{remark}\label{rmk:zero}
Let $N\equiv 3\pmod{8}$. Then~\eqref{effenneunig} reads  $f_N^{*}H_u(N)=0$ because $H_u(N-1)=0$, see Item~(3) of~\Ref{dfn}{nodhypuni}. 
\end{remark}
\subsubsection{Set-up}\label{subsubsec:garlasco}
Let $(\Lambda_N,\xi_N)$ be our standard decorated dimension-$N$ $D$-lattice.  
Let $v\in\Lambda_N$, and suppose that one of the following holds:
\begin{enumerate}
\item
$v$ is  a hyperelliptic vector; thus $v^{\bot}\cong\Lambda_{N-1}$, and it comes with a decoration $\xi_{N-1}$, see~\Ref{subsec}{hyperheeg}.
\item
$N=8k+3$ and $v$ is unigonal; thus $v^{\bot}\cong II_{2,2+8k}$, see~\Ref{subsubsec}{unigon3}.
\item
$N=8k+4$ and $v$ is unigonal; thus $v^{\bot}\cong II_{2,2+8k}\oplus A_1$, see~\Ref{subsubsec}{unigon4}.
\end{enumerate}
Let $\Gamma_v<\Gamma_{\xi_N}$ be the stabilizer of $v$. Thus
\begin{equation*}
\Gamma_v=
\begin{cases}
\Gamma_{\xi_{N-1}} & \text{if (1) holds,} \\
O^{+}(II_{2,2+8k}) & \text{if (2) holds,} \\
O^{+}(II_{2,2+8k}\oplus A_1) & \text{if (3) holds.} \\
\end{cases}
\end{equation*}
Then we have a natural isomorphism
\begin{equation*}
\varphi_v\colon \Gamma_v\backslash \cD_{v^{\bot}}^{+}\overset{\sim}{\lra} H_v(N),
\end{equation*}
where $H_v(N):=H_{v,\Lambda_N}(\Gamma_{\xi_N})$. In fact if Item~(1) holds then $\varphi_v=f_N$, if Item~(2) holds then   $\varphi_v=l_N$, and if Item~(3) holds then   
$\varphi_v=m_N$. 
\begin{definition}
Let $v\in\Lambda_N$ be as above, i.e.~we assume  that one of Items~(1), (2), (3) holds.
Given $w\in\Lambda_N$  such that  $\la v,w\ra$ is a rank-$2$ subgroup of $\Lambda_N$, we let $\pi_{v^{\bot}}(w)$ be a generator of the intersection $(\QQ v\oplus\QQ w)\cap v^{\bot}_{\ZZ}$ (thus $\pi_{v^{\bot}}(w)$ is non-zero, and determined up to $\pm 1$). 
\end{definition}
\begin{proposition}\label{prp:interhyp}
Let $v\in\Lambda_N$ be as above, i.e.~we assume  that one of Items~(1), (2), (3) holds.
 Assume that $w_0\in\Lambda_N$ is a primitive vector of negative square, such that the associated Heegner divisor $H_{w_0}(N)=H_{w_0,\Lambda_N}(\Gamma_{\xi_{N}})$ is different from $H_v(N)$. The following set-theoretic equality holds:
\begin{equation*}
\varphi_v^{-1} H_{w_0}(N)=\bigcup_{\la v,w\ra <0} H_{\pi_{v^{\bot}}(w),v^{\bot}}(\Gamma_v),
\end{equation*}
where $\la v,w\ra <0$  means that $\la v,w\ra$ is a negative definite sublattice of $\Lambda_N$. (And hence $\pi_{v^{\bot}}(w)$ is a vector of negative square.)
\end{proposition}
\begin{proof}
There  is a single observation to make, namely  that if $[\sigma]\in v^{\bot}\cap w^{\bot}\cap\cD^{+}_{\Lambda_N}$, then $\la v,w\ra$ is a negative definite sublattice of 
$\Lambda_N$.
\end{proof}
\subsubsection{Intersection with the hyperelliptic Heegner divisor}\label{subsubsec:inthyp}
\begin{proposition}\label{prp:invnod}
\begin{equation}\label{setenne} 
\scriptstyle f_N^{-1} H_n(N)=
\begin{cases}
\scriptstyle  H_n(N-1) \cup H_h (N-1) & \scriptstyle  \text{if  $N\not \equiv 5 \mod 8$,} \\
\scriptstyle  H_n(N-1) \cup H_h (N-1) \cup H_u(N-1) & \scriptstyle  \text{if  $N\equiv 5 \mod 8$,} 
\end{cases}
\end{equation}
and 
\begin{equation}\label{setunig} 
f_N^{-1} H_u(N)=
\begin{cases}
\es & \text{if $N\equiv 2\pmod{8}$,} \\
H_u(N-1) & \text{if $N\not\equiv 2\pmod{8}$.}
\end{cases}
\end{equation}
\end{proposition}
\begin{proof}
We adopt the notation introduced in~\Ref{subsubsec}{garlasco}, in particular  $v\in\Lambda$ is a \emph{fixed} hyperelliptic vector. 
Let us prove that the right-hand side of~\eqref{setenne} is contained in the left-hand side. 

Let $w\in v^{\bot}$ be of square $-2$  and divisibility $1$ (as vector of $v^{\bot}$). Then $w$ is a nodal vector of $(\Lambda_N,\xi_N)$, and $\pi_{v^{\bot}}(w)=w$; thus  
 $H_n(N-1)\subset f_N^{-1} H_n(N)$ by~\Ref{prp}{interhyp}. 

Let us prove that $H_h (N-1) \subset f_N^{-1} H_n(N)$. Let $u$ be a hyperelliptic vector of $(v^{\bot},\xi_{N-1})$. Let  $w:=(v+u)/2$; then $w\in\Lambda_N$, see the definition of the decoration of $v^{\bot}$ in~\Ref{subsec}{hyperheeg}. We claim that $w$ is a nodal vector. In fact $w^2=-2$, and  $\divisore(w)=1$ because there exists $z\in v^{\bot}$ such that $(u,z)=2$ (because $\divisore_{v^{\bot}}(u)=2$),  and hence $(w,z)=1$. Since $\pi_{v^{\bot}}(w)=u$, it follows that 
$H_h (N-1) \subset f_N^{-1} H_n(N)$ by~\Ref{prp}{interhyp}. 

Now assume that $N=8k+5$. Then  $H_u(N-1)\subset f_N^{-1} H_n(N)$ by Item~(4) of~\Ref{prp}{liftrefl}. 

Next, let us prove that the left-hand side of~\eqref{setenne} is contained in the right-hand side. By~\Ref{prp}{interhyp}, it suffices to prove that, if $w$ is a nodal vector of 
$(\Lambda,\xi_N)$ such that  $\la v,w\ra$ is negative definite, then  $\pi_{v^{\bot}}(w)$ is either a nodal or hyperelliptic vector, or  $N\equiv 5\pmod{8}$ and  $\pi_{v^{\bot}}(w)$ is a unigonal vector. 
Computing the determinant of the restriction of the quadratic form to $\la v,w\ra$ (which is strictly positive), we conclude that $(v,w)\in\{0,\pm 2\}$; multiplying $w$ by $(-1)$ if necessary we may assume that $(v,w)\in\{0,2\}$. 

If $(v,w)=0$, then $\divisore_{v^{\bot}}(w)\in\{1,2\}$. If $\divisore_{v^{\bot}}(w)=1$ then $w$ is a nodal class of $v^{\bot}$. If $\divisore_{v^{\bot}}(w)=2$, then $w$ is a unigonal class of $v^{\bot}$, and  then necessarily $N\equiv 5\pmod{8}$ because the discriminant group $A_{v^{\bot}}=A_{\Lambda_{N-1}}$ contains the class $[w/2]$ of square $-1/2$ modulo $2\ZZ$.  

If $(v,w)=2$, let $u:=v+2w$. Then $\pi_{v^{\bot}}(w)=u$.  We claim that $u$ is a hyperelliptic vector of $(v^{\bot},\xi_{N-1})$. First,   $u^2=-4$. It remains to check that 
$\divisore_{v^{\bot}}(u)=2$ and $u^{*}$ is equal to the decoration that $v^{\bot}$ inherits from the decoration $\xi_N$ of $\Lambda_N$. 
Clearly $\divisore_{v^{\bot}}(u)$ is a multiple of $2$, hence we must show that  $\divisore_{v^{\bot}}(u)\not=4$. 
  If $\divisore_{v^{\bot}}(u)=4$, then $N=8k+4$, and hence $\Lambda_N\cong II_{2,2+8k}\oplus A_1^2$; let $a,b$ be generators of the two  $ A_1$-summands. We may assume that $v=a+b$, because any two decorated dimension-$N$ $D$-lattices are isomorphic. Then there exist $r,m\in\ZZ$, and a \emph{primitive} $c\in II_{2,2+8k}$, such that
\begin{equation*}
w=ra-(1+r)b+(2m+1)c.
\end{equation*}
In fact write $w=ra+sb+tc$, where $r,s,t\in\ZZ$, and $c\in II_{2,2+8k}$ is primitive. Then $s=-(1+r)$ because $(v,w)=2$, and $t$ is odd because the divisibility of $w$ in $\Lambda_N$ is $1$. Thus $u=(1+2x)(a-b)+2(2m+1)c$, and since $II_{2,2+8k}$ is unimodular it follows that  $\divisore_{v^{\bot}}(u)\not=4$. In order to prove that $u^{*}$ is the decoration of $v^{\bot}$, we recall that the latter is equal to $[z/2]$, where $z\in v^{\bot}$ is such that $(v+z)/2\in\Lambda$, see~\Ref{subsec}{hyperheeg} (we have denoted by $z$ the vector denoted $w$ in~\Ref{subsec}{hyperheeg}).  Since $u=v+2w$, it follows that $[u/2]=-[z/2]=[z/2]$ in the discriminant group $A_{v^{\bot}}$.
This proves that 
  $\pi_{v^{\bot}}(w)$ is hyperelliptic, and finishes the proof of~\eqref{setenne}. 
  
Now let us prove~\eqref{setunig}. Let $N-2=8k+a$. If $a=0$, then $H_u(N)=0$, and hence~\eqref{setunig} holds trivially.  Thus we may assume that $a\in\{1,\ldots,7\}$. First we will prove that the right-hand side is contained in the left-hand side.  We may assume that $a\in\{2,\ldots,7\}$ because if $a=1$, then $H_h(N-1)=0$. By~\Ref{rmk}{periodeight}, we may  identify $\Lambda_N$ with 
$ II_{2,2+8k}\oplus D_a$. Let $v=(0_{4+8k},(0,\ldots,0,2)$; then $v^2=-4$ and $\divisore(v)=2$. 
We let $\xi_N=[v/2]$, and hence $v$ is hyperelliptic.  Now suppose that $N$ is odd. Let $w=(0_{4+8k},(2,\ldots,2))$; then $w$ is a unigonal vector of 
$(\Lambda_N,\xi_N)$, see the proof of ~\Ref{prp}{minorm}. Then $\pi_{v^{\bot}}(w)= (0_{4+8k},(1,\ldots,1)\in  II_{2,2+8k}\oplus D_{a-1}$. Since $\pi_{v^{\bot}}(w)$ is a unigonal vector of $(\Lambda_{N-1},\xi_{N-1})$ (see  the proof of ~\Ref{prp}{minorm}), it follows that $f_N^{-1}H_u(N)\supset H_u(N-1)$ if $N$ is odd. The proof for $N$ even is analogous. 

Lastly we prove that the left-hand side of~\eqref{setunig} is contained in the right-hand side. By~\Ref{prp}{interhyp}, it suffices to prove that, if $w$ is a unigonal vector of 
$(\Lambda,\xi_N)$ such that  $\la v,w\ra$ is negative definite, then  $\pi_{v^{\bot}}(w)$ is  a unigonal  vector. Let $a\in\{1,\ldots,7\}$ be the residue of $N-2$ modulo $8$. Let us assume that $N$ is odd. Then $w^2=-4a$ and $\divisore (w)=4.$ We claim that
\begin{equation}\label{duefatti}
(v,w)=\pm 4,\qquad (v+w)/2\in \Lambda.
\end{equation}
In fact we may assume that $w^{*}=\zeta$, and $\xi=2\zeta$. Thus
$$(v/2,w/4)\equiv (\xi,\zeta)\equiv 2q_{\Lambda}(\zeta)\equiv-a/2 \pmod{\ZZ},$$
and hence there exists $s\in\ZZ$ such that $(v,w)= -4a +8s$. 
Since  $\la v,w\ra$ is negative definite, it follows that $a\in\{3,7\}$ (recall  that  $H_h(N-1)=0$ if $a=1$), and that $(v,w)=\pm 4$. Moreover $(v+w)/2\in \Lambda$ because $2[w/4]=[v/2]$ in the discriminant group. 
We have proved~\eqref{duefatti}. Multiplying $w$ by $(-1)$ if necessary, we may assume that $(v,w)=4$; it follows that $\pi_{v^{\bot}}(w)=(v+w)/2$. Now $(v+w)/2$ is a unigonal vector of $v^{\bot}$, as required. Now assume that $N$ is even. Arguing as above, one shows that $(v,w)=\pm 2$, and hence we may assume that $(v,w)=2$. Then $\pi_{v^{\bot}}(w)=v+2w$. Since $v+2w$ is primitive, $\divisore_{v^{\bot}}(v+2w)=4$, and $(v+2w)^2=-4(a-1)$, this proves that $\pi_{v^{\bot}}(w)$ is a unigonal vector of $v^{\bot}$. 
\end{proof}
Let us prove~\eqref{effennehyper} and~\eqref{effenneunig}. First notice that $H_n(N-1)$, $H_h(N-1)$ and $H_u(N-1)$ are irreducible. By~\Ref{prp}{invnod}, it  remains to show that
\begin{eqnarray}
\scriptstyle\mult_{f_N(H_n(N-1))}H_n(N)\cdot H_h(N) & \scriptstyle= & \scriptstyle1,\label{carlo}\\
\scriptstyle\mult_{f_N(H_h(N-1))}H_n(N)\cdot H_h(N) & \scriptstyle= & \scriptstyle2,\label{emilio}\\ 
\scriptstyle\mult_{f_N(H_u(N-1))}H_n(N)\cdot H_h(N) & \scriptstyle= & \scriptstyle1\ \ \text{if $N\equiv 5\pmod{8}$,}\label{gadda}\\ 
\scriptstyle\mult_{f_N(H_u(N-1))}H_u(N)\cdot H_h(N) & \scriptstyle= & \scriptstyle 2\ \ \text{if $N\equiv b\pmod{8}$ and $4\le b\le 9$.} \label{lombardo}
\end{eqnarray}
In order to prove the above equalities, we fix a hyperelliptic vector $v$ of $(\Lambda_N,\xi_N)$, and we let  $w\in v^{\bot}\cong\Lambda_{N-1}$ be a vector 
such that one of the following holds:
\begin{enumerate}
\item
$w$ is a nodal vector of $(\Lambda_{N-1},\xi_{N-1})$.
\item
$w$ is  a hyperelliptic vector  of $(\Lambda_{N-1},\xi_{N-1})$.
\item
$w$ is  a unigonal vector  of $(\Lambda_{N-1},\xi_{N-1})$. 
\end{enumerate}
Let   $[\sigma]\in \{v,w\}^{\bot}\cap\cD_{\Lambda_N}^{+}$. Then $[\sigma]\in f_N(H_n(N-1))$ if (1) holds,  $[\sigma]\in f_N(H_h(N-1))$ if (2) holds, and  $[\sigma]\in f_N(H_u(N-1))$ if (3) holds.  Since each of  $H_n(N-1)$, $H_h(N-1)$ and $H_u(N-1)$ is irreducible, we may prove Formulae~\eqref{carlo}-\eqref{lombardo} by analyzing the local structure of $\cF(N)$, $H_h(N)$ etc.~at the $\Gamma_{\xi_N}$-orbit of $[\sigma]$. 
We  claim that the following hold:
\begin{enumerate}
\item[($1'$)]
If $w$ is a nodal vector of $(\Lambda_{N-1},\xi_{N-1})$, then 
\begin{equation*}
\Lambda_N\cap \la v,w\ra_{\QQ}=\la v,w\ra_{\ZZ}:=\ZZ v+\ZZ w,
\end{equation*}
and  every non trivial isometry of $\la v,w\ra_{\ZZ}$ is either the reflection  in a reflective vector of $\la v,w\ra_{\ZZ}$, or multiplication by $-1$.
\item[($2'$)]
If $w$ is  a hyperelliptic vector  of $(\Lambda_{N-1},\xi_{N-1})$,  then 
\begin{equation*}
\Lambda_N\cap \la v,w\ra_{\QQ}=\la (v+w)/2,(v-w)/2\ra_{\ZZ}\cong D_2,
\end{equation*}
and  every non trivial isometry of $\la (v+w)/2,(v-w)/2\ra_{\ZZ}$ is either the reflection  in a reflective vector of $\la (v+w)/2,(v-w)/2\ra_{\ZZ}$, or multiplication by $-1$.
\begin{enumerate}
\item[($3a'$)]
If $w$ is  a unigonal vector  of $(\Lambda_{N-1},\xi_{N-1})$, and $N\equiv 4\pmod{8}$, then  $\Lambda_N\cap \la v,w\ra_{\QQ}=\la (v+w)/2,(v-w)/2\ra_{\ZZ}\cong D_2$, and   every non trivial isometry of $\la v,w\ra_{\ZZ}$ is either the reflection  in a reflective vector of $\la (v+w)/2,(v-w)/2\ra_{\ZZ}$, or multiplication by $-1$. 
\item[($3b'$)]
If $w$ is  a unigonal vector  of $(\Lambda_{N-1},\xi_{N-1})$, and $N\not\equiv 4\pmod{8}$, then  $\Lambda_N\cap \la v,w\ra_{\QQ}=\la v,w\ra_{\ZZ}$, and   every non trivial isometry of $\la v,w\ra_{\ZZ}$ is either the reflection  in a reflective vector of $\la v,w\ra_{\ZZ}$, or multiplication by $-1$. 
\end{enumerate}
\end{enumerate}
In  fact, in order to determine  $\Lambda_N\cap \la v,w\ra_{\QQ}$, it suffices to recall that $\Lambda_N$ is generated (over  $\ZZ$) by $\ZZ v\oplus v^{\bot}$, together with $(v+u)/2$, where $u\in v^{\bot}=\Lambda_{N-1}$ is a hyperelliptic vector of  $(\Lambda_{N-1},\xi_{N-1})$. Once  $\Lambda_N\cap \la v,w\ra_{\QQ}$ has been determined, the statements about non trivial isometries are...trivial.
Now let   $[\sigma]$ be a  very general point of $\{v,w\}^{\bot}\cap\cD_{\Lambda_N}^{+}$. Then
\begin{equation*}
\sigma^{\bot}\cap\Lambda_{N,\QQ}=\la v,w\ra_{\QQ}, 
\end{equation*}
and $-1_{\{v,w\}^{\bot}}$ is 
the only non-trivial element of  $O(\Lambda_N\cap \{v,w\}^{\bot}_{\QQ})$ stabilizing $[\sigma]$.
 It follows that  there is a natural embedding
 \begin{equation}\label{noia}
T_{\sigma}\colon\Stab([\sigma])\hra O(\Lambda_N\cap \la v,w\ra_{\QQ}) \times  \la-1_{\{v,w\}^{\bot}}\ra.
\end{equation}
(Here $\Stab([\sigma])<\Gamma_{\xi_N}$ is the stabilizer of $[\sigma]$.) 
\begin{claim}\label{clm:unodue}
Let   $[\sigma]$ be a  very general point of $\{v,w\}^{\bot}\cap\cD_{\Lambda_N}^{+}$, and hence  there is the natural embedding~\eqref{noia}. 
\begin{enumerate}
\item[(I)]
If $w$ is a reflective vector of $(\Lambda_{N-1},\xi_{N-1})$, i.e.~Item(1) or (2) above holds, or $N\equiv 4,5\pmod{8}$ and Item~(3) holds, then   $T_{\sigma}$ is surjective.
\item[(II)]
If $w$ is \emph{not} a reflective vector of $(\Lambda_{N-1},\xi_{N-1})$, i.e.~$N\not\equiv 4,5\pmod{8}$ and Item~(3) holds, then  $\Stab([\sigma])=\{ \pm\rho_v\}$. 
\end{enumerate}
\end{claim}
\begin{proof}
Let us prove that~(I) holds. By Items~($1'$), ($2'$), ($3'$), and~\Ref{prp}{liftrefl}, $O(\Lambda_N\cap \la v,w\ra_{\QQ})$ is generated by reflections in vectors which are reflective as vectors of $(\Lambda_N,\xi_N)$; it follows that  the image of  $T_{\sigma}$ contains $O(\Lambda_N\cap \la v,w\ra_{\QQ}) \times \{1_{\{v,w\}^{\bot}}\}$. Thus  Item~(I) holds because $-1_{\Lambda_N}\in \Stab([\sigma])$.  Lastly, we prove that~(II) holds. First $\Stab([\sigma])\supset\{ \pm\rho_v\}$.   Next, we notice that if $u\in \la v,w\ra_{\ZZ}$ is a primitive vector not equal to $\pm v$, then $u$ is not a reflective vector of $(\Lambda_N,\xi_N)$. It follows that there is a single non trivial element of   $O(\la v,w\ra_{\ZZ}) \times \{1_{\{v,w\}^{\bot}}\}$ in the image of $T_{\sigma}$, namely the image of $\rho_v$. Item~(II) follows. 
\end{proof}
Let $w\in v^{\bot}$ be a vector such that~(1), (2), or (3) above holds. Then~\Ref{clm}{unodue} allows us to describe explicitly a neighborhood in $\cF(N)$ of the $\Gamma_{\xi}$-orbit of a very general    point  $[\sigma]\in\{v,w\}^{\bot}\cap\cD_{\Lambda_N}^{+}$. First of all, since $-1_{\Lambda_N}$ acts trivially, we must deal only with the action of 
$\Stab^0([\sigma]):=\Stab([\sigma])/\la -1_{\Lambda_N}\ra$, and secondly, by~\Ref{clm}{unodue} the germ of  $\Gamma_{\xi}[\sigma]$  in $\cF(N)$ is naturally isomorphic to the product of the smooth germ 
$(\{v,w\}^{\bot}\cap\cD_{\Lambda_N}^{+},[\sigma])$ and the  germ of  $\Stab^0([\sigma]) \backslash\la v,w\ra_{\CC}$ at the origin. Let  $(x_1,x_2)$ be coordinates on 
the vector space $\la v,w\ra_{\CC}$ corresponding to the basis $\{ v,w\}$. Thus 
\begin{equation*}
\rho^{*}_v(x_1,x_2)=(-x_1,x_2).
\end{equation*}
We claim that the following hold:
\begin{enumerate}
\item[($1''$)]
If $w$ is a nodal vector of $(\Lambda_{N-1},\xi_{N-1})$, then $\Stab^0([\sigma]) \backslash\la v,w\ra_{\CC}\cong\CC^2$ with coordinates $(y_1,y_2)$ given by $y_1=x_1^2$, $y_2=x_2^2$.
\item[($2''$)]
If $w$ is  a hyperelliptic vector  of $(\Lambda_{N-1},\xi_{N-1})$,  then  $\Stab^0([\sigma]) \backslash\la v,w\ra_{\CC}\cong\CC^2$ with coordinates  $(y_1,y_2)$ given by 
$y_1=x_1^2+x_2^2$, $y_2=x_1^2x_2^2$.
\begin{enumerate}
\item[($3''a$)]
If $w$ is  a unigonal vector  of $(\Lambda_{N-1},\xi_{N-1})$, and   $N\equiv 4\pmod{8}$, then  $\Stab^0([\sigma]) \backslash\la v,w\ra_{\CC}\cong\CC^2$ with coordinates 
  $(y_1,y_2)$ given by 
$y_1=x_1^2+x_2^2$, $y_2=x_1^2x_2^2$.
\item[($3''ba$)]
If $w$ is  a unigonal vector  of $(\Lambda_{N-1},\xi_{N-1})$, and   $N\equiv 5\pmod{8}$, then  $\Stab^0([\sigma]) \backslash\la v,w\ra_{\CC}\cong\CC^2$ with coordinates 
  $(y_1,y_2)$ given by $y_1=x_1^2$, $y_2=x_2^2$.
\item[($3''bb$)]
If $w$ is  a unigonal vector  of $(\Lambda_{N-1},\xi_{N-1})$, and   $N\not\equiv 4,5\pmod{8}$, then $\Stab^0([\sigma]) \backslash\la v,w\ra_{\CC}\cong\CC^2$ with coordinates 
  $(y_1,y_2)$ given by $y_1=x_1^2$, $y_2=x_2$.
\end{enumerate}
\end{enumerate}
In fact the above statements follow from ($1'$), ($2'$), ($3'$), \Ref{clm}{unodue}, and Items~(3), (4) of~\Ref{prp}{liftrefl}. Now we are ready to prove the multiplicity formulae~\eqref{carlo}, \eqref{emilio}, \eqref{gadda}, and~\eqref{lombardo}. Suppose that Item~(1) holds. Since the relevant nodal vector of $(\Lambda_N,\xi_N)$ is $w$ itself,
 Item~($1''$) shows that  the left-hand side of~\eqref{carlo} is equal to the intersection number at $(0,0)$ of the coordinate axes of the plane $\CC^2$ (with coordinates $(y_1,y_2)$), hence $1$. Now suppose that Item~(2) holds. Since the relevant nodal vector of $(\Lambda_N,\xi_N)$ is $(v\pm w)/2$, 
  Item~($2''$) shows that  the left-hand side of~\eqref{emilio} is equal to the intersection number at $(0,0)$ of  $V(y_1^2-4y_2),V(y_2)\subset\CC^2$, hence $2$. Next
   suppose that Item~(3) holds, and that  $N\equiv 5\pmod{8}$. Since the relevant nodal vector of $(\Lambda_N,\xi_N)$ is $w$ itself,
 Item~($3''ba$) shows that  the left-hand side of~\eqref{gadda} is equal to  $1$. Lastly, suppose that Item~(3) holds, and that  $N\not\equiv 5\pmod{8}$. The relevant  unigonal vector of $(\Lambda_N,\xi_N)$ is $(v\pm w)/2$ if $N$ is odd, and $(v\pm 2w)$ if $N$ is even (see the proof of~\Ref{prp}{invnod}); it follows  by Items~($3''a$) and~($3''bb$)  that  the left-hand side of~\eqref{lombardo} is equal to  $1$. 
 \qed
\subsubsection{Intersections with the (reflective) unigonal Heegner divisor, $N\equiv 3\pmod{8}$}\label{subsubsec:intunig3}
We will prove~\eqref{ellennenod} and~\eqref{ellennehyper}. Throughout the present subsubsection we let $N=8k+3$. 
First let us prove the set-theoretic equalities
\begin{equation}\label{sullunig3}
l_N^{-1} H_n(N)=H_n( II_{2,2+8k}),\qquad l_N^{-1} H_h(N)=\es.
\end{equation}
Let $v$ be a unigonal vector of $(\Lambda_N,\xi_N)$. Thus $v^2=-4$, $\divisore(v)=4$, and $v^{\bot}\cong II_{2,2+8k}$. 

Let $w\in v^{\bot}$ be a nodal vector, i.e.~$v^2=-2$. Then $w$ is a nodal vector of $(\Lambda_N,\xi_N)$, and hence $l_N^{-1} H_n(N)\supset H_n( II_{2,2+8k})$ by~\Ref{prp}{interhyp}. On the other hand, suppose that $w\in\Lambda_N$ is a nodal vector such that $\la v,w\ra$ is negative definite; then $w\bot v$ because the determinant of the  restriction of $(,)$ to $\la v,w\ra$ is positive, and hence $w$ is a nodal vector of $v^{\bot}$. By~\Ref{prp}{interhyp}, this proves the reverse inclusion, i.e.~ $l_N^{-1} H_n(N)\subset H_n( II_{2,2+8k})$. This finishes the proof of the first equality 
of~\eqref{sullunig3}. 

The second equality of~\eqref{sullunig3} follows from~\eqref{setunig}, because $H_u(2)=0$. 

It remains to prove that $\mult_{l_N(H_n(II_{2,2+8k}))}H_u(N)\cdot H_n(N)=1$. The proof is analogous to the proof of~\eqref{carlo}; we leave details to the reader.
\qed
\subsubsection{Intersections with the (reflective) unigonal Heegner divisor, $N\equiv 4\pmod{8}$}\label{subsubsec:intunig4}
We will prove~\eqref{emmennenod} and~\eqref{emmennehyper}. Throughout the present subsubsection we let $N=8k+4$. First let us prove the set-theoretic equalities
\begin{equation}\label{sullunig4}
\scriptstyle
 m_N^{-1} H_n(N)  =  H_n( II_{2,2+8k}\oplus A_1),\quad 
 m_N^{-1} H_h(N)   =  H_u( II_{2,2+8k}\oplus A_1).
\end{equation}
Let $v$ be a unigonal vector of $(\Lambda_N,\xi_N)$. Thus $v^2=-2$, $\divisore(v)=2$, $v^{\bot}\cong II_{2,2+8k}\oplus A_1$, and $\Lambda_N=\la v\ra\oplus v^{\bot}$. 

If $w\in v^{\bot}$ is a nodal vector, then $w$ is a nodal vector of $(\Lambda_N,\xi_N)$, and hence $m_N^{-1} H_n(N)\supset H_n( II_{2,2+8k}\oplus A_1)$ by~\Ref{prp}{interhyp}. On the other hand, let $w\in\Lambda_N$ be a nodal vector such that $\la v,w\ra$ is negative definite. Then $w\bot v$ because the determinant of the restriction of $(,)$ to $\la v,w\ra$ is positive. Moreover   $\divisore_{v^{\bot}}(w)=1$ because of the direct-sum decomposition $\Lambda_N=\la v\ra\oplus v^{\bot}$. Thus $w$ is a nodal vector of $v^{\bot}$. Thus $m_N^{-1} H_n(N)\subset H_n( II_{2,2+8k}\oplus A_1)$ by~\Ref{prp}{interhyp}. This finishes the proof of the first equality of~\eqref{sullunig4} 

Let us prove that $m_N^{-1} H_h(N)\supset  H_u( II_{2,2+8k}\oplus A_1)$. We may identify $\Lambda_N$ with $II_{2,2+8k}\oplus D_2$. Let 
$w:=(0_{4+8k},(0,2))$. Then $w^2=-4$ and $\divisore(w)=2$; thus we may choose $\xi_N=[w/2]$. Then $w$ is a hyperellitpic vector of 
$(\Lambda_N,\xi_N)$. Let $v=(1,1)$; then $v$ is a unigonal vector of $(\Lambda_N,\xi_N)$.  Let $z=(1,-1)$; then $z$ is a unigonal vector of $v^{\bot}$. Since $v^{\bot}\cap w^{\bot}=v^{\bot}\cap z^{\bot}$, it follows that  $m_N^{-1} H_h(N)\supset  H_u( II_{2,2+8k}\oplus A_1)$. 

We will prove the reverse inclusion 
by applying~\Ref{prp}{interhyp}. Suppose that $v,w$ are respectively a  unigonal and a hyperelliptic vector of $(\Lambda_N,\xi_N)$, and that $\la v,w\ra$ is negative definite. Since the determinant of the restriction of $(,)$ to $\la v,w\ra$ is positive, either $v\bot w$, or $(v,w)=\pm 2$.   If $w\in v^{\bot}$, then $\divisore_{v^{\bot}}(w)=2$; since there is no vector of square $-4$ and divisibility $2$ in $v^{\bot}\cong II_{2,2+8k}\oplus A_1$, this is a contradiction. Thus  $(v,w)=\pm 2$, and hence we may assume that $(v,w)=2$. Then $(v+w)\in v^{\bot}$, and $(v+w)^2=-2$, $\divisore_{v^{\bot}}(v+w)=2$. Thus $(v+w)$ is a unigonal vector of $v^{\bot}$. By~\Ref{prp}{interhyp}, this proves that $m_N^{-1} H_h(N)\subset  H_u( II_{2,2+8k}\oplus A_1)$. We have completed the proof of~\eqref{sullunig4}.

 It remains to show that 
 \begin{equation}\label{multipemme}
\scriptstyle
\mult_{m_N(H_n(II_{2,2+8k}\oplus  A_1))}H_n(N)\cdot H_u(N)=1,\quad \mult_{m_N(H_u(II_{2,2+8k}\oplus  A_1))}H_h(N)\cdot H_u(N)=2.
\end{equation}
  The proof of the first equality is analogous to the proof of~\eqref{carlo}; we leave details to the reader. Next, we notice that  
 \begin{equation*}
f_N(H_u(N-1))=m_N(H_u(II_{2,2+8k}\oplus A_1)) 
\end{equation*}
by~\eqref{setunig} and~\eqref{sullunig4}, and hence 
  the second equality of~\eqref{multipemme} follows from~\eqref{lombardo}. 
\qed
\subsubsection{Intersections with the (non reflective) unigonal Heegner divisor, $N\equiv 5\pmod{8}$}\label{subsubsec:intunig5}
We will prove~\eqref{quennenod} and~\eqref{quennehyper}.  Throughout the present subsubsection we let $N=8k+5$. First let us show that
\begin{equation}\label{sullunig5}
\scriptstyle
 q_N^{-1} H_n(N)  =  H_n( II_{2,2+8k}\oplus A_2)\cup  H_u( II_{2,2+8k}\oplus A_2),\quad 
 q_N^{-1} H_h(N)   =  H_u(II_{2,2+8k}\oplus A_2).
\end{equation}
Let $v$ be a unigonal vector of $(\Lambda_N,\xi_N)$. If needed, we may assume that we are given an identification $\Lambda_N\cong \II_{2,2+8k}\oplus D_3$, and $v=(0_{4+8k},(2,2,2))$. 

If $w\in v^{\bot}$ is a nodal vector, then $w$ is a nodal vector of $(\Lambda_N,\xi_N)$, and hence $q_N^{-1} H_n(N)\supset H_n( II_{2,2+8k}\oplus A_2)$ by~\Ref{prp}{interhyp}. 
Next, let $v$ be the explicit unigonal vector defined above, and let $z:=(0_{4+8k},(-1,-1,2))$; thus $z$  is a unigonal vector of $v^{\bot}$. Let
$w:=(0_{4+8k},(-1,-1,0))$; thus $w$ is a nodal vector of $\Lambda$ (notice that $(v,w)=4$). Since
 $z-v=3w$, we get  that  $q_N^{-1} H_n(N)\supset H_u(II_{2,2+8k}\oplus A_2)$. 

On the other hand, let $w\in\Lambda_N$ be a nodal vector such that $\la v,w\ra$ is negative definite. Then $(w, v)\in\{0,\pm 4\}$ because the determinant of the restriction of $(,)$ to $\la v,w\ra$ is positive. If $(w,v)=0$, then $w$ is a nodal vector of $v^{\bot}$, if $(w,v)\in\{\pm 4\}$, we may assume that $(w,v)=4$, and then $v+3w$ is a unigonal vector of $v^{\bot}$.  This shows that $q_N^{-1} H_n(N)\subset H_n( II_{2,2+8k}\oplus A_2)\cup  H_u( II_{2,2+8k}\oplus A_2)$, and  finishes the proof of the first equality of~\eqref{sullunig5}.  

Let us prove  the second equality of~\eqref{sullunig5}.  We have
\begin{equation}\label{ballottaggio}
\scriptstyle
q_N^{-1} H_h(N)=q_N^{-1} (f_N(H_u(N-1))=q_N^{-1} (\im (f_N\circ m_{N-1})=q_N^{-1} (\im (q_N\circ r_{N-1})=\im r_{N-1}=H_u( II_{2,2+8k}\oplus A_2).
\end{equation}
In fact, the first equality holds by~\eqref{effennehyper}, and the third equality holds by~\Ref{clm}{triangolo}.

It remains to prove  that 
\begin{eqnarray}
\scriptstyle\mult_{q_N(H_n(\II_{2,2+8k}\oplus A_2))}H_n(N)\cdot H_u(N) & \scriptstyle= & \scriptstyle  1,\label{unus}\\
\scriptstyle\mult_{q_N(H_u(\II_{2,2+8k}\oplus A_2))}H_n(N)\cdot H_u(N) & \scriptstyle= & \scriptstyle 2,\label{duo}\\ 
\scriptstyle\mult_{q_N(H_u(\II_{2,2+8k}\oplus A_2))}H_h(N)\cdot H_u(N) & \scriptstyle= & \scriptstyle 2. \label{tres}
\end{eqnarray}
These intersection multiplicities are computed by methods that we have already employed; we leave details to the reader. 
We notice that Equation~\eqref{tres} also follows directly from~\eqref{lombardo}, since, as noted above, $f_N(H_u(N-1))=q_N(H_u(\II_{2,2+8k}\oplus A_2))$. 

\subsection{Normal bundle formulae}\label{subsec:normale}
\setcounter{equation}{0}
In the present subsection we will prove the following result.
\begin{proposition} 
Let  $4\le N$, and let $f_N:\cF(N-1)\overset{\sim}{\lra} H_u(N)$ be the isomorphism of~\Ref{prp}{hyplocsymm}. Then, the following equalities hold:
\begin{equation}
\scriptstyle  f_N^{*}H_h(N)  \scriptstyle  =  \scriptstyle  -2\lambda(N-1)+H_h(N-1)+\begin{cases}
\scriptstyle 0 & \scriptstyle \text{if $N\not\equiv 4\pmod{8}$} \\
\scriptstyle H_u(N-1) & \scriptstyle \text{if $N\equiv 4\pmod{8}$}
\end{cases}, \label{normalhyper} 
\end{equation}
Similarly, let $l_{8k+3}\colon\cF(\II_{2,2+8k})\overset{\sim}{\lra} H_u(8k+3)$, $m_{8k+4}\colon \cF(\II_{2,2+8k}\oplus A_1)\overset{\sim}{\lra} H_h(8k+4)$, and 
$q_{8k+5}\colon \cF(\II_{2,2+8k}\oplus A_2)\overset{\sim}{\lra} H_h(8k+5)$ be the isomorphisms in~\eqref{ellekappa}, \eqref{emmekappa}, and~\eqref{qukappa}. Then
\begin{eqnarray}
\scriptstyle   l_{8k+3}^{*}H_u(8k+3) & \scriptstyle  = & \scriptstyle  -2\lambda (\II_{2,2+8k}), \label{normunig3} \\
\scriptstyle m_{8k+4}^{*}H_u(8k+4) & \scriptstyle  = & \scriptstyle  -2\lambda (\II_{2,2+8k}\oplus A_1)+H_u(\II_{2,2+8k}\oplus A_1),\label{normunig4} \\
\scriptstyle q_{8k+5}^{*}H_u(8k+5) & \scriptstyle  = & \scriptstyle  -\lambda (\II_{2,2+8k}\oplus A_2)+\frac{3}{2}H_u(\II_{2,2+8k}\oplus A_2).\label{normunig5}
\end{eqnarray}
\end{proposition}
\subsubsection{Adjunction}
Let $\cF=\cF_{\Lambda}(\Gamma)$ be a locally symmetric variety of Type IV, notation as in~\Ref{subsec}{spaziperiodi}. We let $N:=\dim\cF$. Since  $\cF$ has quotient singularities,  the canonical bundle of $\cF$  is a well-defined element  $K_{\cF}\in \Pic(\sF_{\Lambda}(\Gamma))_{\QQ}$. Let $\lambda_{\cF}$ be the automorphic $\QQ$-line bundle on $\cF$. Thus sections of  $\lambda^{\otimes d}_{\cF}$ are identified with weight-$d$ automorphic forms on $\cD^{+}_{\Lambda}$, and, letting $\pi\colon\cD^{+}_{\Lambda}\to\cF$ be the qutient map, $\pi^{*}\lambda_{\cF}\cong \cO_{\cD^{+}_{\Lambda}}(-1)$. 

 If $\Gamma$ acts freely on $\sD^{+}_{\Lambda}$ (and hence $\sF$  is smooth), then 
\begin{equation}\label{cansmooth}
K_{\cF}=N\lambda_{\cF}.
\end{equation}
In fact the above formula follows by descent from the analogous formula  $K_{\cD^{+}_{\Lambda}}=N\cO_{\cD^{+}_{\Lambda}}(-1)$ (adjunction formula for a smooth quadric in $\PP^{N+1}$). In general, there will be a correction term  coming from the ramification of the map $\pi\colon\cD^{+}_{\Lambda}\to \cF$.  
\begin{definition}
Let $B(\cF)$ be the divisor on $\cF$ which is the sum of all Heegner divisors  $H_{v,\Lambda}(\Gamma)$ 
 where  $v\in\Lambda$ is primitive, $v^2<0$, $\pm\rho_v\in\Gamma$, where only  \emph{one} hyperplane appears for  each couple $\{v,-v\}$. 
\end{definition}
\begin{proposition}\label{prp:propriemhur}
Keep notation as above.  
Then  in $\Pic(\cF)_{\QQ}$
\begin{equation}\label{canoneffe}
K_{\cF}=N\lambda_{\cF}-\frac{1}{2}B(\cF).
\end{equation}
\end{proposition}
\begin{proof}
The ramification  divisor of the quotient map $\pi\colon\cD^{+}_{\Lambda}\to \cF$ is the sum  of the pre-Heegner divisors $\cH_{v,\Lambda}(\Gamma)$ 
 where  $v\in\Lambda$ is primitive, $v^2<0$, $\pm\rho_v\in\Gamma$, where only  \emph{one} hyperplane appears for  each couple $\{v,-v\}$. 
 In fact this is \cite[Cor.~2.13]{ghs}. Now let $\Gamma_0\triangleleft\Gamma$ be a finite index  normal subgroup acting freely on $\cD^{+}_{\Lambda}$, let $\cF_0:=\Gamma_0\backslash\cD^{+}_{\Lambda}$, and let $\pi_0\colon\cD^{+}_{\Lambda}\to \cF_0$ be the quotient map. The ramification divisor of the finite Galois cover 
 $\cF_0\to \cF$ is identified with the image by $\pi_0\colon\cD^{+}_{\Lambda}\to \cF_0$ of the ramification  divisor of  $\pi\colon\cD^{+}_{\Lambda}\to \cF$.  
Equation~\eqref{cansmooth} and Riemann-Hurwitz applied to  $\cF_0\to \cF$ give the proposition.  
\end{proof}
As before, let $B(N)$ be the Weil divisor on $\cF(N)$ defined by
\begin{equation}\label{branchen}
B(N)=H_n(N)+2\Delta(N)= \begin{cases}
H_n(N)+H_h(N) & \text{if $N\not\equiv 3,4\pmod{8}$,}\\
H_n(N)+H_h(N) +H_u(N) & \text{if $N\equiv 3,4\pmod{8}$.}
\end{cases}
\end{equation}
By \Ref{crl}{reflheeg}, we see that $B(N)$ is the branch divisor for $\calF(N)$ (i.e. $B(N)=B(\cF(\Lambda,\xi))$ for $(\Lambda,\xi)$ a dimension $N$ decorated $D$ lattice). Thus, 
\begin{corollary}\label{crl:propriemhur}
Keeping notation as above, 
\begin{equation}\label{eqcanonicalclass}
K_{\cF(N)}=N \lambda(N)- \frac{1}{2}B(N).
\end{equation}
\end{corollary}
%
%
%
%
%
Let $\lambda(II_{2,2+8k}):=\lambda_{\cF(II_{2,2+8k})}$, $\lambda(II_{2,2+8k}\oplus A_1):=\lambda_{\cF(II_{2,2+8k}\oplus A_1)}$, and $\lambda(II_{2,2+8k}\oplus A_2):=\lambda_{\cF(II_{2,2+8k}\oplus A_2)}$.
The proof of the result below is omitted,  because it is analogous to the proof of~\Ref{crl}{propriemhur}.
\begin{corollary}\label{crl:apercan}
Keep notation as above. The following equalities hold in $\Pic(\cF(II_{2,2+8k}))_{\QQ}$,  $\Pic(\cF(II_{2,2+8k}\oplus A_1)_{\QQ}$,  and  $\Pic(\cF(II_{2,2+8k}\oplus A_2)_{\QQ}$ respectively:
\begin{eqnarray*}
\scriptstyle K_{\cF(II_{2,2+8k})} & \scriptstyle = &  \scriptstyle (8k+2) \lambda(II_{2,2+8k})- \frac{1}{2}H_n(II_{2,2+8k}),\\
\scriptstyle K_{\cF(II_{2,2+8k}\oplus A_1)} & \scriptstyle = & \scriptstyle (8k+3) \lambda(II_{2,2+8k}\oplus A_1)
- \frac{1}{2}H_n(II_{2,2+8k}\oplus A_1)- \frac{1}{2} H_u(II_{2,2+8k}\oplus A_1), \\
\scriptstyle K_{\cF(II_{2,2+8k}\oplus A_2)} & \scriptstyle = & \scriptstyle (8k+4) \lambda(II_{2,2+8k}\oplus A_2)
- \frac{1}{2}H_n(II_{2,2+8k}\oplus A_2)- \frac{1}{2} H_u(II_{2,2+8k}\oplus A_2).
\end{eqnarray*}
\end{corollary}
\subsubsection{Normal bundle formula for the hyperelliptic divisor}
\begin{proof}[Proof of~\eqref{normalhyper}]  
By~\Ref{prp}{hyplocsymm}, the intersection $H_h(N)\cap \sing\cF(N)$ has codimension at least two in $H_h(N)$, and hence we may apply adjunction to  compute the canonical class of $H_h(N)$. Since $f_N\colon \cF(N-1)\to H_h(N)$ is an isomorphism, \Ref{crl}{propriemhur} gives
\begin{eqnarray*}
K_{\cF(N-1)}&=& f_N^{*}(K_{\cF(N)}+H_h(N)) = \\ &= & f_N^{*}\left(N\lambda(N)-\frac{1}{2}(B(N)-H_h(N))+\frac{1}{2}H_h(N)\right)=\\
&=& N\lambda(N-1)-\frac{1}{2}f_N^{*}(B(N)-H_h(N))+\frac{1}{2} f_N^{*} H_h(N).
\end{eqnarray*}
On the other hand the canonical class of $\cF(N-1)$ is given by~\eqref{eqcanonicalclass}; equating the two expressions for $K_{\cF(N-1)}$ one gets 
\begin{equation}
\scriptstyle
f_N^{*}H_h(N)=-2\lambda(N-1)+f_N^{*}(B(N)-H_h(N))-B(N-1).
\end{equation}
By~\eqref{effennehyper} and~\eqref{effenneunig}, 
\begin{equation*}
f_N^{*}(B(N)-H_h(N))-B(N-1)=H_h(N-1)+\begin{cases}
\scriptstyle 0 & \scriptstyle \text{if $N\not\equiv 4\pmod{8}$} \\
\scriptstyle H_u(N-1) & \scriptstyle \text{if $N\equiv 4\pmod{8}$}
\end{cases}
\end{equation*}
\end{proof}
\subsubsection{Normal bundle formula for the unigonal divisor, $N\equiv 3\pmod{8}$}
\begin{proof}[Proof of~\eqref{normunig3}]
We let $N=8k+3$. By~\Ref{prp}{unigcong3}, the intersection $H_u(N)\cap \sing\cF(N)$ has codimension at least two in $H_u(N)$, and hence we may apply adjunction to  compute the canonical class of $H_u(N)$. Since $l_N\colon \cF(II_{2,2+8k})\to H_u(N)$ is an isomorphism, \Ref{crl}{propriemhur} gives
\begin{eqnarray*}
K_{\cF(II_{2,2+8k})}&=& l_N^{*}(K_{\cF(N)}+H_u(N)) = \\ &= & l_N^{*}\left(N\lambda(N)-\frac{1}{2}(B(N)-H_u(N))+\frac{1}{2}H_u(N)\right)=\\
&=& N\lambda(II_{2,2+8k})-\frac{1}{2}l_N^{*}(B(N)-H_u(N))+\frac{1}{2} l_N^{*} H_u(N).
\end{eqnarray*}
On the other hand $K_{\cF(II_{2,2+8k})}  =  (2+8k) \lambda(II_{2,2+8k})- \frac{1}{2}H_n(II_{2,2+8k})$ by~\Ref{crl}{apercan}. Comparing the two expressions for 
 $K_{\cF(II_{2,2+8k})}$, and invoking~\eqref{ellennenod}, \eqref{ellennehyper}, one gets~\eqref{normunig3}. 
\end{proof}
\subsubsection{Normal bundle formula for the unigonal divisor, $N\equiv 4\pmod{8}$}
\begin{proof}[Proof of~\eqref{normunig4}]
We let $N=8k+4$. By~\Ref{prp}{unigcong4}, the intersection $H_u(N)\cap \sing\cF(N)$ has codimension at least two in $H_u(N)$, and hence we may apply adjunction to  compute the canonical class of $H_u(N)$. Since $m_N\colon \cF(II_{2,2+8k}\oplus A_1)\to H_u(N)$ is an isomorphism, \Ref{crl}{propriemhur} gives
\begin{eqnarray*}
\scriptstyle K_{\cF(II_{2,2+8k}\oplus A_1)} & \scriptstyle = & \scriptstyle m_N^{*}(K_{\cF(N)}+H_u(N)) = \\ 
&\scriptstyle = & \scriptstyle m_N^{*}\left(N\lambda(N)-\frac{1}{2}(B(N)-H_u(N))+\frac{1}{2}H_u(N)\right)=\\
&\scriptstyle =& \scriptstyle N\lambda(II_{2,2+8k}\oplus A_1)-\frac{1}{2}m_N^{*}(B(N)-H_u(N))+\frac{1}{2} m_N^{*} H_u(N).
\end{eqnarray*}
Now expand $K_{\cF(II_{2,2+8k}\oplus A_1)}$ according to~\Ref{crl}{apercan}, and compare the two expressions for 
 $K_{\cF(II_{2,2+8k}\oplus A_1)}$; then one gets~\eqref{normunig4} by invoking~\eqref{emmennenod} and~\eqref{emmennehyper}. 
\end{proof}
\subsubsection{Normal bundle formula for the (non reflective) unigonal divisor, $N\equiv 5\pmod{8}$}
\begin{proof}[Proof of~\eqref{normunig5}]
We let $N=8k+5$. By the usual arguments, \Ref{crl}{propriemhur} gives
\begin{eqnarray*}
\scriptstyle K_{\cF(II_{2,2+8k}\oplus A_2)} & \scriptstyle = & \scriptstyle q_N^{*}(K_{\cF(N)}+H_u(N)) = \\ 
&\scriptstyle = & \scriptstyle q_N^{*}\left(N\lambda(N)-\frac{1}{2}B(N)+H_u(N)\right)=\\
&\scriptstyle =& \scriptstyle N\lambda(II_{2,2+8k}\oplus A_2)-\frac{1}{2}q_N^{*} B(N)+ q_N^{*} H_u(N).
\end{eqnarray*}
Now expand $K_{\cF(II_{2,2+8k}\oplus A_2)}$ according to~\Ref{crl}{apercan}, and compare the two expressions for 
 $K_{\cF(II_{2,2+8k}\oplus A_2)}$; then one gets~\eqref{normunig5} by invoking~\eqref{quennenod} and~\eqref{quennehyper}.

\end{proof}
\begin{remark}\label{rmk:borchcompat}
The formulae proved in~\Ref{subsec}{restrram} and~\Ref{subsec}{normale} are compatible with the Borcherds' relations in~\Ref{thm}{thmborcherds1} and~\Ref{thm}{thmborcherds2}. More precisely, let $N\equiv 3,4,5\pmod{8}$. If we assume that~\eqref{borcherds1} holds, and we apply $f_N^{*}$, the validity of~\eqref{borcherds1} with $N$ replaced by $N-1$ holds, and similarly for~\eqref{borcherds2}. 
\end{remark}
\subsubsection{Pull-back of $(\lambda(N)+\beta\Delta(N))$}\label{subsubsec:italicum}
Equations~\eqref{normalhyper} and~\eqref{effenneunig} give the following formula: 
\begin{equation}\label{ennedelta}
\scriptstyle  f_N^{*}\Delta(N)= 
 \begin{cases}
\scriptstyle  -\lambda(N-1)+\Delta(N-1) & \scriptstyle  \text{if $N\not\equiv 4,5\pmod{8}$,} \\
\scriptstyle   -\lambda(N-1)+\frac{1}{2}H_h(N-1) & \scriptstyle  \text{if $N\equiv 5\pmod{8}$,} \\
\scriptstyle   -\lambda(N-1)+\Delta(N-1)+H_u(N-1) & \scriptstyle  \text{if $N\equiv 4\pmod{8}$.}
\end{cases}
\end{equation}
(Recall that $H_u(M)=0$ if $M\equiv 2\pmod{8}$.) 
Repeated application of~\eqref{ennedelta} gives the following result.
\begin{proposition}\label{prp:ennekappaelbeta}
Let $N\ge 4$,  and  $1\le k\le(N-3)$.  Then
\begin{equation*}
\scriptstyle  f_{N-k,N}^{*}(\lambda(N)+\beta\Delta(N))= 
 \begin{cases}
\scriptstyle  (1-k\beta)\lambda(N-k)+\beta\Delta(N-k) & \scriptstyle  \text{if $N-k\not\equiv 4\pmod{8}$  and $k\ge 2$, } \\
& \scriptstyle  \text{or $k=1$ and $N-1\not\equiv 3,4 \pmod{8}$,} \\
\scriptstyle   (1-k\beta)\lambda(N-k)+\frac{1}{2}\beta H_h(N-k) & \scriptstyle  \text{if $N-k\equiv 4\pmod{8}$,} \\
\scriptstyle  (1-k\beta)\lambda(N-k)+\beta\Delta(N-k)+\beta H_u(N-k) & \scriptstyle  \text{if $k=1$ and $N-1\equiv 3\pmod{8}$.}
\end{cases}
\end{equation*}
\end{proposition}
It will be convenient to let $f_{N,N}:=\Id_{\cF(N)}$. 
The  result below follows from~\Ref{prp}{ennekappaelbeta}, together with Equations~\eqref{ellennehyper} and~\eqref{normunig3}.
\begin{proposition}\label{prp:restunig3}
Suppose that  $N\ge 3$, $0\le k\le (N-3)$, and  $N-k\equiv 3\pmod{8}$ (hence $l_{N-k}$ makes sense, see~\eqref{ellekappa}). Then
\begin{equation*}
(f_{N-k,N}\circ  l_{N-k})^{*}(\lambda(N)+\beta\Delta(N))= 
\begin{cases}
  (1-(k+1)\beta)\lambda(\II_{2,N-k-1}) &   \text{if $k\not=1$,} \\
  (1-4\beta)\lambda(\II_{2,N-k-1}) &   \text{if $k=1$.} \\
\end{cases}
\end{equation*}
\end{proposition}
\begin{proposition}\label{prp:restunig4}
Suppose that $N\ge 4$, $0\le k\le (N-4)$, and  $N-k\equiv 4\pmod{8}$ (hence $m_{N-k}$  and $p_{N-k-1}$ make sense, see~\eqref{emmekappa} and~\eqref{pikappa})).  Then
\begin{eqnarray*}
\scriptstyle (f_{N-k,N}\circ  m_{N-k})^{*}(\lambda(N)+\beta\Delta(N)) & \scriptstyle = &  
\begin{cases}
 \scriptstyle (1-\beta)\lambda(\II_{2,N-k-2}\oplus A_1)+\frac{3}{2}\beta H_u(\II_{2,N-k-2}\oplus A_1) & \scriptstyle  \text{if $k=0$,} \\
 \scriptstyle   (1-k\beta)\lambda(\II_{2,N-k-2}\oplus A_1)+\beta H_u(\II_{2,N-k-2}\oplus A_1)&  \scriptstyle \text{if $k\ge 1$,} \\
\end{cases}\\
\scriptstyle (f_{N-k,N}\circ  m_{N-k}\circ p_{N-k-1})^{*}(\lambda(N)+\beta\Delta(N)) & \scriptstyle = & 
\begin{cases}
\scriptstyle (1-4\beta)\lambda(\II_{2,N-k-2}) & \scriptstyle   \text{if $k=0$,} \\
 \scriptstyle   (1-(k+2)\beta)\lambda(\II_{2,N-k-2}) &  \scriptstyle  \text{if $k\ge 1$.} \\
\end{cases}
\end{eqnarray*}
\end{proposition}
\begin{proof}
The first equation  follows from~\Ref{prp}{ennekappaelbeta}, together with Equations~\eqref{emmennehyper} and~\eqref{normunig4}. Next, note that  $f_{N-k,N}\circ  m_{N-k}\circ p_{N-k-1}=f_{N-k,N}\circ  f_{N-k}\circ l_{N-k-1}=f_{N-k-1,N}\circ   l_{N-k-1}$ by~\Ref{clm}{triangolo}, and  hence the second equation follows from~\Ref{prp}{restunig3}. 
\end{proof}
The proof of the result below is omitted, because it is similar to the proof of~\Ref{prp}{restunig4}. 
\begin{proposition}\label{prp:giaccherini}
Suppose that $N\ge 5$, $0\le k\le (N-5)$, and  $N-k\equiv 5\pmod{8}$ (hence $q_{N-k}$, $r_{N-k-1}$, and $p_{N-k-2}$ make sense, see~\eqref{qukappa}, \eqref{errekappa} and~\eqref{pikappa}).  Then
\begin{eqnarray*}
\scriptstyle (f_{N-k,N}\circ  q_{N-k})^{*}(\lambda(N)+\beta\Delta(N)) & \scriptstyle = & \scriptstyle  (1-k\beta)\lambda(\II_{2,N-k-3}\oplus A_2)+\beta H_u(\II_{2,N-k-3}\oplus A_2),\\
\scriptstyle  (f_{N-k,N}\circ  q_{N-k}\circ r_{N-k-1})^{*}(\lambda(N)+\beta\Delta(N)) & \scriptstyle  =  & \scriptstyle  (1-(k+1)\beta)\lambda(\II_{2,N-k-3}\oplus A_1)+\beta H_u(\II_{2,N-k-3}\oplus A_1), \\
\scriptstyle  (f_{N-k,N}\circ  q_{N-k}\circ r_{N-k-1}\circ p_{N-k-2})^{*}(\lambda(N)+\beta\Delta(N)) & \scriptstyle  =  & \scriptstyle   (1-(k+3)\beta)\lambda(\II_{2,N-k-3}). 
\end{eqnarray*}
\end{proposition}
\begin{remark}
Notice that the numbers in Table~\eqref{table:treinter} can be obtained by applying the results of the present section, and hence one can give an algebro-geometrical proof of
  Borcherds' relation~\eqref{borcherds1} for $N=19$. By~\Ref{rmk}{borchcompat}, Borcherds' first relation for $N= 18$ also follows.
\end{remark}
\subsection{Heuristics for the predictions}
\setcounter{equation}{0}
Let $15\le N$. We define a collection $\Tower(N)$ of closed subsets of $\cF(N)$, as follows: $X\subset\cF(N)$ belongs to $\Tower(N)$  if and only if one of the following holds:
\begin{enumerate}
\item
$X=\im f_{M,N}$ for $11\le M\le N$,
\item
$X=\im (f_{M,N}\circ l_M)$ for $11\le M\le N$  (recall that $f_{N,N}=\id_{\cF(N)}$), and $M\equiv 3\pmod{8}$,
\item
$X=\im (f_{M,N}\circ m_M)$ for $12\le M\le N$, and $M\equiv 4\pmod{8}$.
\item
$X=\im (f_{M,N}\circ q_M)$ for $13\le M\le N$, and $M\equiv 5\pmod{8}$.
\end{enumerate}
Every $X\in\Tower(N)$ is irreducible (and closed), because it is the image of a regular map,  whose domain is a projective irreducible set. 
\begin{definition}\label{dfn:keyofex}
Let $N\ge 15$. Given $X\in\Tower(N)$, we let
\begin{enumerate}
\item
$t_N(X)=(N-M)$, if 
\begin{enumerate}
\item[(1a)]
$X=\im f_{M,N}$, and $14\le M\le N$, or
\item[(1b)]
$X=\im (f_{M,N}\circ m_M)$, where $12\le M\le N-1$, and $M\equiv 4\pmod{8}$, or
\item[(1c)]
$X=\im (f_{M,N}\circ q_M)$, where $13\le M\le N$, and $M\equiv 5\pmod{8}$.
\end{enumerate}
\item
$t_N(X)=(N-M+1)$, if $X=\im (f_{M,N}\circ l_M)$,  where $11\le M\le N$, with $M\not=N-1$, and $M\equiv 3\pmod{8}$.
\item
$t_N(X)=N-14$, if $X=\im f_{M,N}$, and  $11\le M\le 13$.
\item
$t_N(X)=1$, if $X=\im  m_N$, where $N\equiv 4\pmod{8}$.
\item
$t_N(X)=4$, if $X=\im (f_{N-1,N}\circ l_{N-1})$, where $N\equiv 4\pmod{8}$.
\end{enumerate}
\end{definition}
\begin{proposition}\label{prp:whypred}
Let $15\le N$. Let $X\in\Tower(N)$ (and hence $X$ is $\QQ$-factorial). If $\beta\in[0,1]\cap\QQ$, then we have an equality of $\QQ$-Cartier divisor classes:
\begin{equation}\label{quartexit}
(\lambda(N)+\beta\Delta(N))|_{X}= (1- t_N(X)\cdot \beta)\lambda(N)|_{X}+D_X,
\end{equation}
where $D_X$ is an effective divisor whose support is a union of elements of $\Tower(N)$. 
\end{proposition}
\begin{proof}
First we consider the case
 $X=\im(f_{M,N})$ and $11\le M\le 13$. There exists $c\ge 0$ such that
\begin{equation*}
f_{M,N}^{*}(\lambda(N)+\beta\Delta(N))=(1-(N-M)\beta)\lambda(M)+\frac{1}{2} \beta H_h(M)+c\beta H_u(M),
\end{equation*}
by~\Ref{prp}{ennekappaelbeta}.
On the other hand,  by~\eqref{gritenne} (with $N$ replaced by $M$) we have 
\begin{equation*}
f_{M,N}^{*}(\lambda(N)+\beta\Delta(N))=(1-(N-14)\beta)\lambda(M)+(c+\frac{\tau(M)}{32}(\mu(M+8)-\mu(M))\beta H_u(M).
\end{equation*}
Thus Equation~\eqref{quartexit} holds because $\mu(M)\le \mu(M+8)$ for  $4\le M\le 14$, and 
\begin{equation*}
H_u(M)=
\begin{cases}
\im l_{11} & \text{if $M=11$,}\\
\im m_{12} & \text{if $M=12$,}\\
\im q_{13} & \text{if $M=13$.}
\end{cases}
\end{equation*}
This proves that Equation~\eqref{quartexit}  holds if  $X=\im(f_{M,N})$ and $11\le M\le 13$.  In all other cases, Equation~\eqref{quartexit} is an immediate consequence of the results in~\Ref{subsubsec}{italicum}, and the isomorphisms in~\eqref{ellekappa}, \eqref{emmekappa}, \eqref{pikappa}, \eqref{qukappa}, and~\eqref{errekappa}.
\end{proof}
Because of~\eqref{quartexit}, we expect that if $t_N(X)\not=0$,  then $(t_N(X)\lambda(N)+\Delta(N))$ generates a wall of the Mori chamber decomposition of the convex cone spanned by $\lambda(N)$ and $\Delta(N)$, and that the center of the corresponding flip contains the strict transform of $X$. Now notice that if $X,Y\in\Tower(N)$, and $X\subset Y$, then $t_N(X)\ge t_N(Y)$. Thus, in order to list candidates for the centers of the flips, we give the definition below.
\begin{definition}
Let $15\le N$. We let $\Center(N)\subset\Tower(N)$  be the subset of $X$ such that  
\begin{enumerate}
\item
$a_N(X)>0$, and
\item
 if $Y\in\Tower(N)$ contains properly $X$, then $a_N(X)>a_N(Y)$. 
\end{enumerate}
\end{definition}
Let $X\in\Tower(N)$, i.e.~one of Items~(1)-(5) of~\Ref{dfn}{keyofex} holds: then $X$ does \emph{not} belong to $\Center(N)$ if and only if Item~(3) holds. 

We can summarize~\Ref{pred}{mainpred} as follows: The walls of the Mori chamber decomposition of the convex cone spanned by $\lambda(N)$ and $\Delta(N)$ are generated by the vectors
 $t_N(X)\lambda(N)+\Delta(N)$, where $X$ runs through the elements of $\Center(N)$. If $t=t_N(X)$ for some $X\in\Center(N)$, and $t>1$, then the center of the flip corresponding to  $t_N(X)\lambda(N)+\Delta(N)$
is the union of the strict transforms of the elements $Y\in\Center(N)$ such that $t_N(Y)=t$. Lastly, $\lambda(N)+\Delta(N)$ contracts the strict transforms of  the elements $Y\in\Center(N)$ such that $t_N(Y)=1$. 

A more detailed description goes as follows. Consider the wall closest to the ray spanned by $\lambda(N)$, i.e.~that spanned by $(\lambda(N)+(N-10)^{-1}\Delta(N))$, with center $X=\im(f_{11,N}\circ l_{11})$. Then $D_X=0$, by~\Ref{prp}{restunig3}, and hence $(\lambda(N)+(N-10)^{-1}\Delta(N))|_X$ is the opposite of an ample class. 
What should we expect 
$\cF(N,(N-10)^{-1}+\epsilon)$ to look like? In order to answer this, we recall Looijenga's  key observation:   if $X$ is an irreducible component of $\Delta^{(k)}(N)$, then the exceptional divisor of $Bl_X\cF(N)$ (a weighted blow-up of $\cF(N)$ with center $X$) is a trivial  weighted projective bundle over $X$, \emph{away} from $\Delta^{(k+1)}(N)$ (this is because $\Delta^{(k)}(N)$ is the quotient by $\Gamma(N)$ of the points of intersection of $k$ hyperplane sections of $\cD^{+}_{\Lambda_N}$). For our initial $X$, i.e.~$\im(f_{11,N}\circ l_{11})$, this translates into the prediction that $\cF(N,(N-10)^{-1}+\epsilon)$ is obtained from $\cF(N,(N-10)^{-1}-\epsilon)$ by replacing (the closure of) $\im(f_{11,N}\circ l_{11})$ with a  
 $w\PP^{N-11}$. What about the remaining flips, corresponding to $\lambda(N)+k^{-1}\Delta(N)$ for $k\in\{N-12,N-11,\ldots,2\}$? First we note that if $X\in\Center(N)$, then Equation~\eqref{quartexit} holds with $D_X$ an effective divisor whose support is a union of elements of $\Tower(N)$, except if $X=\im f_{14,N}$. Suppose for the moment that  $X\not=\im f_{14,N}$; then since $t_N(Y)>t_N(X)$ for all prime divisors in the support of $D_X$ (by definition of $\Tower(N)$), and because  $\cF(N,k^{-1}+\epsilon)$ is obtained from $\cF(N,k^{-1}-\epsilon)$ by replacing (the closure of) centers $W$ such that $t_N(W)=k$ with projective spaces of dimension $N-1-\cod(W,\cF(N))$, we see that the picture replicates itself. Lastly, let us look at the picture  for   $X=\im f_{14,N}$. In this case Gritsenko's relation, i.e.~\Ref{crl}{keyrelations}, reads $H_h(14)=H_u(14)$. A geometric description of the rational equivalence  $H_h(14)=H_u(14)$ gives that $H_h(14)$ and $H_u(14)$ are actually trivial outside $H_h(14)\cap H_u(14)$, and  hence~\eqref{quartexit} for  $X=\im f_{14,N}$ holds with $D_X=0$. This then suggests that the usual description holds also for the flip corresponding to 
$\im f_{14,N}$. 
\begin{remark}\label{rmk:shiftone}
The set of values $t_N(X)$ for $X\in\Center(N)$  is obtained by adding $1$ to each value $t_{N-1}(Y)$ for $Y\in\Center(N-1)$, and adjoining the value $1=t_N(H_h(N))$. This is explained by the formulae 
in~\Ref{prp}{ennekappaelbeta}. In fact, for simplicity suppose that $N-1\not\equiv 3,4\pmod{8}$. Then, for $\beta\not=1$, we have
\begin{equation*}
f_N^{*}(\lambda(N)+\beta\Delta(N))=(1-\beta)\lambda(N-1)+\beta\Delta(N-1)=(1-\beta)(\lambda(N-1)+\frac{\beta}{1-\beta}\Delta(N-1)).
\end{equation*}
This explains the behaviour described above, because if $\beta/(1-\beta)=1/k$, then $\beta=1/(k+1)$. 
\end{remark}
An induction on $N$ (see~\Ref{rmk}{shiftone}) proves the following result.
\begin{proposition}\label{prp:primocrit}
Let $ N\ge 11$, and let 
\begin{equation*}
0\le \beta<
\begin{cases}
(N-10)^{-1} & \text{if $N\not=12$,} \\
1/4  & \text{if $N=12$.}
\end{cases}
\end{equation*}
 be rational. If $C\subset\cF(N)$ is a complete curve, then  
$$C\cdot (\lambda(N)+\beta\Delta(N))> 0.$$
\end{proposition}
\begin{remark}
The locally symmetric varieties $\cF(N)$ are not projective, but they are swept out by complete curves. In fact the complement of $\cF(N)$ in the Baily-Borel compactification $\cF(N)^{*}$ is of dimension $1$, and the assertion follows because $\dim\cF(N)=N\ge 3$. 
\end{remark}
%

 \bibliography{refk3}
 \end{document}